\documentclass[a4paper,10pt,reqno]{amsart}
\usepackage{dsfont,amsmath,amsfonts,amscd,amssymb,graphicx,mathrsfs}
\usepackage[usenames]{color}
\usepackage{setspace}
\usepackage{array,multirow,booktabs,longtable}
\RequirePackage{tikz-cd}
\usepackage{tikz}
\usepackage{array} 

\DeclareMathOperator{\End}{End} 
\DeclareMathOperator{\Ext}{Ext}

\DeclareMathOperator{\Hom}{Hom}

\DeclareMathOperator{\Supp}{Supp}

\DeclareMathOperator{\id}{id} 
\DeclareMathOperator{\Rep}{Rep}

\DeclareMathOperator{\Aut}{Aut}
\DeclareMathOperator{\GL}{GL}

\DeclareMathOperator{\Spec}{Spec} 
\definecolor{lblue}{rgb}{0.3,0.0,4.4}
\definecolor{lred}{rgb}{4.3,0.0,0.4}

\newcommand{\QCoh}{\mathrm{QCoh}}
\newcommand{\Lmod}[1]{#1\text{-}{\mathsf{mod}}}
\newcommand{\LMod}[1]{#1\text{-}{\mathsf{Mod}}}

\newcommand{\Ha}{\mathcal{H}} 
 
\newcommand{\Q}{\mathcal{Q}} 
 
\newcommand{\Orb}{\mathcal{O}\mathrm{rb}} 
\newcommand{\N}{\mathcal{N}}

\newcommand{\iso}{\stackrel{\sim}{\rightarrow}} 
 
\renewcommand{\H}{\mathsf{H}}
\newcommand{\Tr}{\mathrm{Tr}}
\newcommand{\s}{\mathfrak{S}}
\newcommand{\Z}{\mathbb{Z}}

\newcommand{\h}{\mathfrak{h}}

\newcommand{\C}{\mathbb{C}}
\newcommand{\res}{\mathrm{res}}
\newcommand{\sres}{\mathrm{sres}}
\newcommand{\T}{\mathbb{T}}
\newcommand{\Ham}{\mathbb{H}}
\newcommand{\dd}{\mathscr{D}}
\newcommand{\gr}{\mathrm{gr}}
\newcommand{\mc}{\mathcal}
\newcommand{\Add}{\mathscr{C}}
\newcommand{\ms}{\mathscr}
\newcommand{\g}{\mathfrak{g}}
\newcommand{\idot}{\bullet}
\newcommand{\ds}{\dots}
\newcommand{\Irr}{\mathrm{Irr}\ }
\newcommand{\reg}{\mathrm{reg}}
\renewcommand{\SS}{\mathrm{Ch}}
\newcommand{\mbf}{\mathbf}
\newcommand{\vdim}{\mbf{v}}
\newcommand{\Stab}{\mathrm{Stab}}
\newcommand{\Osph}{\mathcal{O}_{\mathbf{\kappa}}^{\mathrm{sph}}}
\newcommand{\Cs}{\mathbb{C}^{\times}}

\newcommand{\mf}{\mathfrak}
\newcommand{\sym}{\mathrm{Sym}\ }
\newcommand{\eu}{\mathsf{eu}}
\newcommand{\IC}{\mathrm{IC}}

\newcommand{\Hol}{\mathrm{Hol}}
\newcommand{\Coh}{\mathrm{Coh}}

\newtheorem{theorem}{Theorem}[section]
\newtheorem{lemma}[theorem]{Lemma}
\newtheorem{definition}[theorem]{Definition}
\newtheorem{proposition}[theorem]{Proposition}
\newtheorem{corollary}[theorem]{Corollary}
\newtheorem{remark}[theorem]{Remark}

\newtheorem{example}[theorem]{Example}

\newtheorem{claim}[theorem]{Claim}
 \definecolor{lightblue}{rgb}{0.8,0.8,0.9}
  \definecolor{lightred}{rgb}{0.9,0.8,0.8}
\begin{document}

\title{Semi-simplicity of the category of admissible $\dd$-modules}

\author{Gwyn Bellamy}

\address{School of Mathematics and Statistics, University Gardens, University of Glasgow, Glasgow,  G12 8QW, UK.}
\email{gwyn.bellamy@glasgow.ac.uk}

\author{Magdalena Boos}

\address{Ruhr-Universit\"at Bochum, Faculty of Mathematics,  D - 44780 Bochum, Germany.}
\email{Magdalena.Boos-math@ruhr-uni-bochum.de}

 \begin{abstract}
	Using a representation theoretic parametrization for the orbits in the enhanced cyclic nilpotent cone, derived by the authors in a previous article, we compute the fundamental group of these orbits. This computation has several applications to the representation theory of the category of admissible $\dd$-modules on the space of representations of the framed cyclic quiver. First, and foremost, we compute precisely when this category is semi-simple. We also show that the category of admissible $\dd$-modules has enough projectives. Finally, the support of an admissible $\dd$-module is contained in a certain Lagrangian in the cotangent bundle of the space of representations. Thus, taking characteristic cycles defines a map from the $K$-group of the category of  admissible $\dd$-modules to the $\Z$-span of the irreducible components of this Lagrangian. We show that this map is always injective, and a bijection if and only if the monodromicity parameter is integral. 
\end{abstract}

\maketitle

\tableofcontents
 
\section{Introduction}\label{sect:intro}

The notion of admissible $\dd$-modules first appeared (implicitly) in Harish-Chandra's seminal body of work on representations of real reductive groups. Later, admissible $\dd$-modules were formally introduced as the algebraic analogue, via the Riemann-Hilbert correspondence, of Lusztig's character sheaves \cite{AdmissibleModules}, \cite{MirkovicVilonen}, and \cite{GunnighamAbelian}. The summands of the Springer sheaf appear as a special case of Lusztig's character sheaves; their algebraic analogues are the summands of the Harish-Chandra $\dd$-module. They constitute important examples of admissible $\dd$-modules, and have been extensively studied e.g. \cite{HottaKashiwara}, \cite{LS}. 

More recently, Gan and Ginzburg \cite{AlmostCommutingVariety} defined admissible $\dd$-modules on the space of representations of a certain quiver naturally associated to the Hilbert scheme of points in the plane. Their motivation for doing so was the fact that the spherical subalgebra of the rational Cherednik algebra of a symmetric group can be realised as the quantum Hamiltonian reduction of the ring of differential operators on the space of representations. This means that there is a natural functor of Hamiltonian reduction that associates to each $G$-monodromic $\dd$-module a corresponding representation of the spherical rational Cherednik algebra. In particular, admissible $\dd$-modules are precisely those mapped to category $\mc{O}$ under Hamiltonian reduction. A key new feature of admissible modules in this setting is their dependence on a parameter $\chi$; the properties of the category vary greatly depending on the choice of $\chi$. This case was studied further in \cite{MirabolicCharacter} and \cite{mirabolicHam} where, amongst other things, the analogue of the Harish-Chandra module is considered. 

Finally, admissible $\dd$-modules have been shown to play a key role in understanding geometric category $\mc{O}$ associated to quantizations of Higgs branches (in particular, quiver varieties), as described in \cite{GenCatOWebster}. 

Despite their intimate relation with Springer theory, character sheaves, rational Cherednik algebras and quantized Higgs branches, admissible $\dd$-modules are surprisingly poorly understood, at least from the algebraic point of view. The goal of this article is to try and remedy this, by developing general algebraic results that can be applied to categories of admissible $\dd$-modules on a $G$-representation.

\subsection{Admissible $\dd$-modules}\label{sec:addintro} 

In sections \ref{sec:monodromicproperties} to \ref{sec:admissible}, which constitute the heart of the paper, we consider an arbitrary category of admissible $\dd$-modules, defined on an affine $G$-variety satisfying certain natural finiteness conditions; these are (F1)-(F3) of section \ref{sec:admissible}. However, our motivating example throughout has been admissible modules on the space of representations of the framed cyclic quiver, where the associated algebra of quantum Hamiltonian reduction is the spherical subalgebra of the rational Cherednik algebra for the wreath product $\s_n \wr \Z_{\ell}$. Therefore, for simplicity, we describe in the remainder of the introduction what our results mean in this case. 

Let $\Q(\ell)$ be the cyclic quiver with $\ell$ vertices, and $\Q_{\infty}(\ell)$ the framing $\infty \rightarrow 0$ of this quiver at the vertex $0$. Throughout, we let $X$ denote the space $\Rep(\Q_{\infty}(\ell),\vdim)$ of representations of the framed cyclic quiver, with dimension vector $\vdim := \epsilon_{\infty} + n \delta$, where $\delta$ is the minimal imaginary root of $\Q(\ell)$. The group $G = \prod_{i = 0}^{\ell-1} \GL_n$ acts by gauge transformations on $X$. Fix a character $\chi$ of the Lie algebra $\mf{g}$ of $ G$. The category $\Add_{\chi}$ of admissible $\dd$-modules on $X$ is the category of all smooth $(G,\chi)$-monodromic $\dd$-modules on $X$, whose singular support lies in a certain Lagrangian $\Lambda$. Essentially those modules whose singular support is nilpotent in the conormal direction; see section \ref{sec:admissible} for details. Admissible $\dd$-modules are always regular holonomic, and it is easily shown that there are only finitely many simple objects in $\Add_{\chi}$. However, it is generally very hard to say precisely how many simple objects there are in this category. 

\subsection{Counting simple objects}\label{sec:countingitro}

The enhanced cyclic nilpotent cone $\N_{\infty}(\ell,n)$ is the subspace of $X$ consisting of nilpotent representations. The group $G$ acts on $\N_{\infty}(\ell,n)$ with finitely many orbits. These orbits were first classified by Johnson \cite{Joh}, extending work of Achar-Henderson \cite{AH} and Travkin \cite{Tr}. In the article \cite{BB-enhnil-quiver}, we gave a different parametrization of these orbits in terms of the representation theory of the underlying quiver. Let $\mc{P}$ denote the set of all partitions and $\mc{P}_{\ell}$ the set of all $\ell$-multi-partitions. In the article \cite{BB-enhnil-quiver}, we showed that:

\begin{theorem}\label{thm:paramcomb}
	The $G$-orbits in the enhanced cyclic nilpotent cone $\N_{\infty}(\ell,n)$ are naturally labelled by the set 
	$$
	\mathcal{Q}(n,\ell) := \left\{ (\lambda;\nu) \in \mathcal{P} \times \mathcal{P}_{\ell} \ | \ \res_{\ell}(\lambda) + \sres_{\ell}(\nu) = n \delta\right\}. 
	$$
\end{theorem}

Here $\res_{\ell}(\lambda)$ and $\sres_{\ell}(\nu)$ are the (shifted) $\ell$-residues of the corresponding partitions; see section \ref{sec:sscriterionframed} for details. A similar result also appears in \cite{DoGinT}.  The utility of this theorem lies in the fact that the orbit $\mc{O}_{(\lambda;\nu)}$ labelled by $(\lambda;\nu)$ consists of representations of the quiver $\Q_{\infty}(\ell)$ where the dimension vector of each indecomposable summand can be easily recovered from $\res_{\ell}(\lambda)$ and $\sres_{\ell}(\nu)$. We show in section \ref{sect:fund_group} that the fundamental group of $\mc{O}_{(\lambda;\nu)}$ depends only on these dimension vectors. Thus, using Theorem \ref{thm:paramcomb}, we can compute the fundamental groups of each $G$-orbit in the enhanced cyclic nilpotent cone. The groups that appear are quotients of $\Z^{\ell}$, with $\pi_1(\mc{O}_{(\lambda;\nu)}) = \Z^{\ell}$ if and only if $\nu = \emptyset$; see Lemma \ref{lem:multitauf}. As a consequence, we can explicitly count the number of isomorphism classes of simple modules in $\Add_{\chi}$. Define 
$$
\mathcal{Q}_{\chi}(n,\ell) := \left\{ (\lambda;\nu) \in \mathcal{Q}(n,\ell) \ | \ \left\langle \exp(\chi),\sigma^i \res_{\ell}\left(\nu^{(i)}_j\right) \right\rangle = 1, \ \forall \ i,j \right\},
$$
where the notation is explained in section \ref{sec:framedcombinatorics}. 

\begin{theorem}\label{thm:maincount}
	There is a natural bijection $\mathcal{Q}_{\chi}(n,\ell) \stackrel{1:1}{\longleftrightarrow} \Irr \Add_{\chi}$.
\end{theorem}

The proof of Theorem \ref{thm:maincount} is explained in section \ref{sec:localcomp}. 

\subsection{The functor of Hamiltonian reduction} 

As mentioned previously, the category of admissible $\dd$-modules is closely related to representations of the rational Cherednik algebra $\H_{\mbf{\kappa}}(W)$, where $W$ is the wreath product $\Z_{\ell} \wr \s_n$. The functor of Hamiltonian reduction is an exact quotient functor $\Ham_{\chi} : \Coh (\dd_{X},G,\chi) \rightarrow \Lmod{e \H_{\mbf{\kappa}}(W) e}$ from the category of $(G,\chi)$-monodromic $\dd$-modules to representations of the spherical subalgebra $e \H_{\mbf{\kappa}}(W) e$. It maps the category $\Add_{\chi}$ of admissible $\dd$-modules onto spherical category $\Osph$. Being a quotient functor, there are objects that are killed by $\Ham_{\chi}$. Because of its connection to localization results in the style of Beilinson-Bernstein, it is an important problem, which has been intensively studied (see e.g. \cite{MNKNstrata} and \cite{mirabolicHam}), to try and characterize what exactly is killed by Hamiltonian reduction. Let $\mathsf{R}^+$ be the set of positive roots for the affine root system of type $\widetilde{\mathsf{A}}_{\ell-1}$, with $\delta$ the minimal imaginary root. We set 
\begin{equation}\label{eq:hyperplanes}
\mc{R}_{n} := \{ \alpha \in \mathsf{R}^+ \ |  \ \varepsilon_0 \cdot \alpha < n \} \cup \{ n \delta \}.
\end{equation}
Using Theorem \ref{thm:maincount}, we show:   

\begin{theorem}\label{thm:mainequiv1}
	The following are equivalent: 
	\begin{enumerate}
		\item[(a)] $\Ham_{\chi} : \Add_{\chi} \rightarrow \Osph$ is an equivalence. 
		\vspace{2mm}
		\item[(b)] $\Ham_{\chi} : \Coh (\dd_{X},G,\chi) \rightarrow \Lmod{e \H_{\mbf{\kappa}} e}$ is an equivalence. 
		\vspace{2mm}
		\item[(c)] $\Gamma(X,\ms{M})^G \neq 0$ for all non-zero objects $\ms{M}$ of $\Add_{\chi}$.
		\vspace{2mm}
		\item[(d)] $\Gamma(X,\ms{M})^G \neq 0$ for all non-zero $(G,\chi)$-monodromic $\dd_{X}$-modules $\ms{M}$.
	\end{enumerate}
Moreover, each of the above holds if and only if $\chi \cdot \alpha \notin \Z$ for all $\alpha \in \mc{R}_n$.  
\end{theorem}

In particular, Theorem \ref{thm:mainequiv1} says that if there is a $(G,\chi)$-monodromic $\dd_{X}$-module with no (non-zero) $G$-invariant global sections, then necessarily there exists an admissible $\dd_{X}$-module with no (non-zero) $G$-invariant global sections. 

Although admissible $\dd$-modules in our context were originally introduced by Gan and Ginzburg as a tool to study modules in category $\Osph$, we require a basic fact about category $\Osph$ in order to prove Theorem \ref{thm:mainequiv1}. Namely, we use Ariki's criterion on the semi-simplicity of the cyclotomic Hecke algebra to deduce that the simple objects in $\Osph$ are in bijection with $\ell$-multi-partitions of $n$, when $\chi \cdot \alpha \notin \Z$ for all $\alpha \in \mc{R}_n$. 

\subsection{Algebras of quantum Hamiltonian reduction} 

As noted above, the main motivation for studying the functor of Hamiltonian reduction is that it relates admissible $\dd$-modules to modules over the spherical rational Cherednik algebra. This is possible precisely because the latter can be realised as a quantum Hamiltonian reduction of the algebra $\dd(X)$ of differential operators on $X$. As we have seen, this does have a drawback in that the functor of Hamiltonian reduction is often not an equivalence. 

In order to remedy this, we study a natural generalization of the construction of Hamiltonian reduction. Namely, for any finite dimensional $G$-module $U$, there is an associated algebra of quantum Hamiltonian reduction $\mf{A}_{\chi}(U)$ of $\dd(X)$. Just as in the case where $U = \C$ is the trivial $G$-module, there is a functor of Hamiltonian reduction 
\begin{equation}\label{eq:HamredfuntcorU}
\Ham_{U,\chi} : \Add_{\chi} \rightarrow \mc{O}_{\chi}(U)
\end{equation}
to ``category $\mc{O}$'' for $\mf{A}_{\chi}(U)$. The fact that the category of admissible $\dd$-modules has only finitely many isomorphism classes of simple objects implies that, for most $U$, the functor $\Ham_{U,\chi}$ is an equivalence. We deduce: 

\begin{proposition}\label{prop:enoughproj}
	The category $\Add_{\chi}$ has enough projectives. 
\end{proposition}

It follows that the indecomposable projective objects in $\Add_{\chi}$ are labelled by the set $\mc{Q}_{\chi}(n,\ell)$. One can lift Proposition \ref{prop:enoughproj} to a statement about the category $\QCoh(\dd_X,G,\chi)$ of all quasi-coherent $(G,\chi)$-monodromic $\dd$-modules on $X$. Namely, there is a projective object $\ms{P}_{\chi}(U)$ in $\Coh(\dd_X,G,\chi)$, whose endomorphism algebra is the quantum Hamiltonian reduction $\mf{A}_{\chi}(U)$. We say that $U$ is \textit{sufficiently large} if (\ref{eq:HamredfuntcorU}) is an equivalence.  

\begin{theorem}\label{thm:projgen}
	For all $U$ sufficiently large, the $\dd$-module $\ms{P}_{\chi}(U)$ is a projective generator in $\QCoh(\dd_X,G,\chi)$ and hence $\Coh(\dd_X,G,\chi)$ is equivalent to $\Lmod{\mf{A}_{\chi}(U)}$. 
\end{theorem}

One can always choose finitely many simple $G$-modules $\{ U_i \}$ such that $[U : U_i] \neq 0$ implies that $U$ is sufficiently large. In particular, there are always infinitely many sufficiently large $G$-modules. If $M$ is a module for the quantum Hamiltonian reduction $\mf{A}_{\chi}(U)$, then one can consider it as a sheaf over $X/\!/G$. In particular, one can consider the support $V(M)$ of modules $M$ in $\mc{O}_{\chi}(U)$. In general it is difficult to describe the properties of the algebras $\mf{A}_{\chi}(U)$ since they tend to depend non-trivially on the parameter $\chi$. However, the following theorem describes the properties of the algebras $\mf{A}_{\chi}(U)$ when $\chi$ is generic. 

\begin{theorem}\label{thm:genericbehabousqhr1}
	If $\chi \cdot \alpha \notin \Z$ for all $\alpha \in \mc{R}_n$ then:
	\begin{enumerate}
		\item[(a)] $\mf{A}_{\chi}(U)$ is simple, for all non-zero $U \in \Rep(G)$. 
		\vspace{2mm}
		\item[(b)] The algebras $\mf{A}_{\chi}(U)$ and $\mf{A}_{\chi}(U')$ are Morita equivalent, for all $U,U' \in \Rep(G)$.  
		\vspace{2mm}
		\item[(c)] $V(M) = X/\!/G$ for all non-zero $U \in \Rep(G)$ and all non-zero $M$ in $\mc{O}_{\chi}(U)$.
	\end{enumerate}
\end{theorem}

Our proof of Theorem \ref{thm:genericbehabousqhr1} is rather indirect. One would like to understand what the algebras $\mf{A}_{\chi}(U)$ look like on some dense open subset of $X/\!/G$. Unfortunately, there is no open subset of $X/\!/G$ where the group $G$ acts freely on its preimage in $X$. To remedy this, we consider instead the $G$-stable open subset $X^{\circ}$ of $X$ consisting of all representations of the framed cyclic quiver admitting a cyclic vector at the framed vertex. In this case, there is indeed an open subset of $X/\!/G$ on whose preimage in $X^{\circ}$ the group $G$ acts freely. This implies that the algebras $\mf{A}_{\chi}^{\circ}(U)$, defined on $X^{\circ}$, are much better behaved than $\mf{A}_{\chi}(U)$. For instance, we show:

\begin{proposition}\label{prop:AcircUprime}
	The algebras $\mf{A}_{\chi}^{\circ}(U)$ are prime. 
\end{proposition}

In general the algebras $\mf{A}_{\chi}(U)$ are not prime. The properties listed in Theorem \ref{thm:genericbehabousqhr1} are much easier to establish for the algebras $\mf{A}_{\chi}^{\circ}(U)$. Thus, we are reduced to asking how similar the two algebras are. Since $X^{\circ}$ is an affine open subset of $X$, localization induces an algebra map $\varphi_U : \mf{A}_{\chi}(U) \rightarrow \mf{A}_{\chi}^{\circ}(U)$. Let $j : X^{\circ} \hookrightarrow X$ be the open embedding. In general, $\varphi_U$ is neither injective nor surjective. More precisely, we show that: 

\begin{theorem}\label{thm:mainequiv2}
 The following are equivalent:  
	\begin{enumerate}
		\item[(a)] $j^* : \Coh(\dd_X,G,\chi) \rightarrow \Coh(\dd_{X^{\circ}},G,\chi)$ is an equivalence. 
		\vspace{2mm}
		\item[(b)] $\varphi_U: \mf{A}_{\chi}(U) \rightarrow \mf{A}_{\chi}^{\circ}(U)$ is an isomorphism for all $U \in \Rep(G)$. 
	\end{enumerate}
Moreover, each of the above holds if and only if $\chi \cdot \alpha \notin \Z$ for all $\alpha \in \mc{R}_n$. 
\end{theorem}

In other words, the map $\varphi_U$ is an isomorphism for all $U \in \Rep(G)$ if and only if there are no $(G,\chi)$-monodromic $\dd$-modules supported on $X \smallsetminus X^{\circ}$, if and only if $\chi \cdot \alpha \notin \Z$ for all $\alpha \in \mc{R}_n$. 

\subsection{Characteristic cycles} 

As noted earlier, the singular support of an admissible $\dd$-module is, by definition, contained in a certain Lagrangian $\Lambda$. Theorem \ref{thm:paramcomb} implies that the Lagrangian $\Lambda$ has irreducible components $\Lambda_{(\lambda;\nu)}$ labelled by pairs $(\lambda;\nu) \in \mc{Q}(n,\ell)$. Therefore, there is a characteristic cycles map $\SS$ from the Grothendieck group $K_0( \Add_{\chi})$ of $\Add_{\chi}$ to the free abelian group $\bigoplus_{(\lambda;\nu)} \Z \Lambda_{(\lambda;\nu)}$. 

\begin{proposition}\label{prop:cyclesitro}
	The characteristic cycles map 
	$$
	\SS : K_0( \Add_{\chi}) \rightarrow \bigoplus_{(\lambda;\nu) \in \mc{Q}(n,\ell)} \Z \Lambda_{(\lambda;\nu)}
	$$
	is always injective. It is an isomorphism if and only if $\chi$ is integral. 
\end{proposition}

We note that this result is the analogue, in the setting of admissible $\dd$-modules, of the injectivity of cycle maps in geometric representation theory; see \cite{DoddCycle}. It would be interesting to explicitly describe the image under $\SS$ of the simple objects in $\Add_{\chi}$. 

\subsection{Semi-simplicity}

Finally, we turn to the main result mentioned in the abstract; that is, we consider when the category of admissible $\dd$-modules is semi-simple. As one might expect, this is closely related to the question of when category $\Osph$ is semi-simple. It is well-known that this latter category is semi-simple if and only if the corresponding cyclotomic Hecke algebra is semi-simple. As mentioned previously, Ariki gave an explicit numerical criterion for when the cyclotomic Hecke algebra is semi-simple. In particular, his criterion, together with the explicit expression for the parameters $\kappa$ of the rational Cherednik algebra in terms of the character $\chi$, shows that the constraint $\chi \cdot \alpha \notin \Z$ for all $\alpha \in \mc{R}_n$ is equivalent to:
\begin{center}
	\vspace{2mm}
	$\Osph$ is semi-simple and $k \notin \Z$.
	\vspace{2mm}
\end{center}
Here $k \in \C$ is the first of the entries of the tuple $\kappa$. The open set $X^{\circ}$ contains the subset $X^{\reg}$, where additionally we ask that the map "going once around the cycle" is regular semi-simple; see section \ref{sec:qhrframedcyclicproof}. Combining the above result with Theorem \ref{thm:mainequiv1}, we deduce: 

\begin{theorem}\label{thm:mainequiv}
	The following are equivalent:
	\begin{enumerate}
		\item[(a)] $\Add_{\chi}$ is semi-simple.
		\vspace{2mm}
		\item[(b)] $\ms{M} |_{X^{\reg}} \neq 0$ for all non-zero $\ms{M}$ in $\Add_{\chi}$. 
		\vspace{2mm}
		\item[(c)] $\Osph$ is semi-simple and $k \notin \Z$.
		\vspace{2mm}
	\end{enumerate}
Moreover, each of the above holds if and only if $\chi \cdot \alpha \notin \Z$ for all $\alpha \in \mc{R}_n$. 
\end{theorem}

Thus, $\Add_{\chi}$ is semi-simple away from a countable union of hyperplanes. The case where $k \in \Z$ and $\Osph$ is semi-simple corresponds to the situation where the character $\chi$ is integral i.e. the derivative of a character of $G$. In this case, the rational Cherednik algebra does not help us in analysing $\Add_{\chi}$ since $\Osph$ is semi-simple. Instead, we show directly, see Theorem \ref{thm:intss}, that the Harish-Chandra $\dd$-module $\mc{G}_{\chi} \in \Add_{\chi}$ is \textit{not} semi-simple. This implies that $\Add_{\chi}$ is not semi-simple. We note that Theorems \ref{thm:mainequiv1}, \ref{thm:mainequiv2} and \ref{thm:mainequiv} list in total 9 statements which, for $X = \Rep(\Q_{\infty}(\ell),\vdim)$, are \textit{all} equivalent, and hold if and only if $\chi \cdot \alpha \notin \Z$ for all $\alpha \in \mc{R}_n$. 

The study of admissible $\dd$-modules on $X$ is motivated, in part, by the fact that it is a variation on the idea of admissible $\dd$-modules on a simple Lie algebra $\mf{g}$. In that situation, it has been recently shown by Gunningham \cite{GunnighamAbelian} that many of the properties of the category of admissible $\dd$-modules on $\mf{g}$ lift to all quasi-coherent $G$-equivariant $\dd$-modules on $\mf{g}$ (similar in spirit to Theorem \ref{thm:projgen}). The methods of \cite{GunnighamAbelian} are completely different to our approach. We have been informed by Gunningham that he has, in addition, been able to relate the block decomposition of the algebras $\mf{A}(U)$ to Lusztig's cuspidal character sheaves on $\mf{g}$.

\subsection{Outline of the article}

In section \ref{sec:monodromicproperties} we recall the basic facts regarding monodromic $\dd$-modules that we will required later. Then, in section \ref{sec:QHR} we introduce, and study the algebras $\mf{A}_{\chi}(U)$ of Hamiltonian reduction. The proof of Proposition \ref{prop:AcircUprime} is given here. Admissible $\dd$-modules are considered in section \ref{sec:admissible}. We prove here Proposition \ref{prop:enoughproj} and Theorem \ref{thm:projgen}. Section \ref{sect:framedquivertheory} describes in more detail the geometry of the framed cyclic quiver. The proof of Theorem \ref{thm:maincount}, Theorem \ref{thm:mainequiv2} and Proposition \ref{prop:cyclesitro} are given in this section. Finally, section \ref{sec:sscriterionframed}  is devoted to the proof of the main results. In particular, the proofs of Theorems \ref{thm:mainequiv1}, \ref{thm:genericbehabousqhr1} and \ref{thm:mainequiv} are given in sections \ref{sec:proofthm:mainequiv1}, \ref{sec:genericbehabousqhr1proof} and \ref{sec:proofmainframedquiveradd} respectively. \\

{\bf Acknowledgements:} The authors would like to thank K. Brown and T. Schedler for helpful remarks on the subject. We would like to thank the referee for an extremely detailed and constructive review of an earlier version of the article. The first author was partially supported by EPSRC grant EP/N005058/1.

\section{Monodromic $\dd$-modules}\label{sec:monodromicproperties}

This section collects together all the results regarding monodromic $\dd$-modules that we will require later. 

\subsection{Notation} If $A$ is an algebra then $\LMod{A}$ will denote the category of left $A$-modules. If $A$ is noetherian, then $\Lmod{A}$ will denote the category of finitely generated $A$-modules.

If $X$ is a smooth variety over $\C$, then $\dd_X$ denotes the sheaf of differential operators on $X$ and $\Coh(\dd_X)$ denotes the category of coherent (left) $\dd_X$-modules. The sheaf of vector fields on $X$ is denoted $\Theta_X$. The singular support of a coherent $\dd$-module $\ms{M}$ is denoted $\SS(\ms{M})$. It is a coisotropic subvariety of $T^* X$. By a local system we always mean an algebraic vector bundle equipped with an integrable connection that has regular singularities. 

If the affine algebraic group $G$ acts on an affine variety $X$, then $X/G$ denotes the set of orbits and $X/\!/ G := \Spec \C[X]^G$ is the categorical quotient. For an element $x \in \g := \mathrm{Lie} \ G$, its centralizer in $G$ is denoted $Z_G(x)$, and its centralizer in $\g$ is $Z_{\mf{g}}(x)$. The category of all finite dimensional $G$-modules is denoted $\Rep (G)$.

\subsection{Monodromic modules}

 Fix $G$ a affine algebraic group with Lie algebra $\g$, and let $X$ be a smooth quasi-projective $G$-variety. The group $G$ acts via Hamiltonian automorphisms on $T^* X$ and there is an associated moment map $\mu \colon T^*X \rightarrow \g^*$. Dual to $\mu$ is the comoment map $\nu \colon \mf{g} \rightarrow \Theta_X$, coming from differentiating the action of $G$ on $X$. It is a morphism of Lie algebras and hence extends to an algebra morphism $\nu \colon U(\mf{g}) \rightarrow \dd_X$. 
 
 Let $\mathbb{X}^*(\mathfrak{g})$ denote the complex vector space of  linear characters $\chi \colon \mathfrak{g} \rightarrow \C$. Similarly, $\mathbb{X}^*(G)$ denotes the lattice of all complex characters $\psi \colon G \rightarrow \C^{\times}$. Differentiation defines a map $d \colon \mathbb{X}^*(G) \rightarrow \mathbb{X}^*(\mathfrak{g})$. . 
 
Differentiating the left regular action of $G$ on itself, we get a comoment map $\nu_L: \mf{g} \rightarrow \dd_G$ that identifies $\mf{g}$ with the space of \textit{right invariant} differential operators on $G$. Define
\begin{equation}\label{eq:ochiG}
   \mc{O}_{\chi}^G := \dd_G \otimes_{\nu_L(\mf{g})} \C_{\chi},
\end{equation}
where $\C_{\chi} = \C \cdot v_{\chi}$ with $x \cdot v_{\chi} = \chi(x) v_{\chi}$ for $x \in \mf{g}$. If $G$ is reductive, then $\mc{O}_{\chi}^G $ has regular singularities. 

Let $a \colon G \times X \rightarrow X$ denote the action map and let $s \colon X \rightarrow G \times X$ denote the closed embedding $x \mapsto (e,x)$

\begin{definition}\label{defn:strongmonodormic}
	A quasi-coherent $\dd$-module is said to be $(G,\chi)$-monodromic if there is a \textit{fixed} isomorphism $\theta_{\ms{M}} : \mc{O}_G^{\chi} \boxtimes \ms{M} \iso a^* \ms{M}$, satisfying 
\begin{enumerate}
\item (\textit{Rigidity}) $s^* \theta = \id_{\ms{M}}$. 
\item (\textit{Cocycle condition}) The following diagram is commutative:
	\begin{equation}\label{eq:cocycle}
	\begin{tikzcd}
		\mc{O}^{\chi}_G \boxtimes \mc{O}^{\chi}_G \boxtimes \ms{M} \ar[rr,"{\id_{G} \times \theta}"] \ar[d,equal] &&  \mc{O}^{\chi}_G \boxtimes a^* \ms{M} \ar[d,equal] \\
		(m \times \id)^* (\mc{O}_G^{\chi} \boxtimes \ms{M}) \ar[d,"{(m \times \id_X)^* \theta}"'] &&  (\id_G \times a)^*(\mc{O}_G^{\chi} \boxtimes \ms{M}) \ar[d,"{(\id_G \times a)^* \theta}"] \\
		(m \times \id_X)^* a^* \ms{M} \ar[rr,equal] && (\id_G \times a)^*a^* \ms{M}  
	\end{tikzcd}
	\end{equation}
\end{enumerate}
\end{definition}

The category of $(G,\chi)$-monodromic modules on $X$ is denoted $\QCoh(\dd_X,G,\chi)$. We recall, without proof, some of the standard properties of monodromic $\dd$-modules that we will use later. 

\begin{proposition}\label{prop:affinemono}
	If $X$ is affine then $\ms{M}$ is $(G,\chi)$-monodromic if and only if it is weakly $G$-equivariant and $\nu(x) - \chi(x) = \nu_M(x)$ on $\Gamma(X,\ms{M})$, for all $x \in \mf{g}$. Here $\nu_M : \mf{g} \rightarrow \End(\Gamma(X,\ms{M}))$ is the differential of the $G$-action. 
\end{proposition}

In the affine setting, weakly $G$-equivariant means that $\Gamma(X,\ms{M})$ is a rational $G$-module, with the action map $\dd(X) \times \Gamma(X,\ms{M}) \rightarrow \Gamma(X,\ms{M})$ being $G$-equivariant. More generally, the analogue of \cite[Proposition 2.6]{VdenBGeq} holds in the monodromic setting. This allows one to drop the affine assumption in Proposition \ref{prop:affinemono}. 

\begin{lemma}\label{lem:rationalsections}
	For any $G$-stable open subset $U$ of $X$, the space $\Gamma(U,\ms{M})$ is a rational $G$-module, such that 
	$$
	\Gamma(U,\ms{M})^G = \{ m \in \Gamma(U,\ms{M}) \ | \ \nu(x) \cdot m = \chi(x) m, \ \forall x \in \g \}. 
	$$
\end{lemma}

\begin{proof}[Proof of Lemma \ref{lem:rationalsections}]
	If $\ms{M}$ is $(G,\chi)$-monodromic, then as a quasi-coherent $\mc{O}_X$-module, it is $G$-equivariant. Therefore the first claim follows from the corresponding fact about $G$-equivariant $\mc{O}_X$-modules c.f. \cite[\S 9.10]{HTT}. The second claim follows easily from the definitions, since 
	$$
	\Gamma(G,\mc{O}^{\chi}_G)^G = \C \cdot 1 = \{ f | \ \nu(x) \cdot f = \chi(x) f \},
	$$
	c.f. the proof of \cite[Theorem 11.5.3]{HTT}.
\end{proof}

The following observation will also be required in section \ref{sec:proofmainframedquiveradd}. 

\begin{lemma}\label{lem:twistequivalence}
	Let $\psi \in \mathbb{X}^*(G)$. Then the functor $\ms{M} \mapsto \ms{M} \otimes \psi$ is an equivalence 
	$$
	\QCoh(\dd_X,G,\chi) \stackrel{\sim}{\longrightarrow} \QCoh(\dd_X,G,\chi + d \psi).
	$$
\end{lemma}

Assume now that $G$ is connected. Then the map $d : \mathbb{X}^*(G) \rightarrow \mathbb{X}^*(\mathfrak{g})$ is an embedding, and we write $\T(G)$ for the torus $\mathbb{X}^*(\mathfrak{g}) / d \mathbb{X}^*(G)$. Given $\chi \in \mathbb{X}^*(\mathfrak{g})$, its image in $\T(G)$ is denoted $\exp(\chi)$. One can easily check that $\mc{O}_G^{\chi} \simeq \mc{O}_G^{\chi'}$ if and only if $\exp(\chi) = \exp(\chi')$ in $\mathbb{T}(G)$. Therefore, for each $q \in \mathbb{T}(G)$, there is a $\dd_G$-module $\mc{O}_G^q$, well-defined up to isomorphism.  

\begin{definition}
	Fix $q \in \mathbb{T}(\g)$. A quasi-coherent $\dd_{X}$-module $\ms{M}$ is said to be \textit{$(G,q)$-monodromic} if there is an isomorphism $\mc{O}^q_G \boxtimes \ms{M} \simeq a^* \ms{M}$.   
\end{definition}

We note that $(G,q)$-monodromic $\dd$-modules are $\dd$-modules that \textit{can} be endowed with a weakly $G$-equivariant structure. In other words, the full subcategory $\QCoh(\dd_X,G,q)$ of $\QCoh(\dd_X)$ consisting of all $(G,q)$-monodromic $\dd$-modules is the image of the forgetful functor $\mathrm{For} : \QCoh(\dd_X,G,\chi) \rightarrow \QCoh(\dd_X)$, where $\chi$ is any character with $\exp(\chi) = q$. In fact, since $G$ is assumed to be connected, the forgetful functor $\mathrm{For} : \QCoh(\dd_X,G,\chi) \rightarrow \QCoh(\dd_X,G,q)$ is an equivalence.

\subsection{Homogeneous spaces}

In the case where $X = \mc{O} = G / K$ is a homogeneous $G$-space, one can describe the possible $(G,q)$-monodromic local systems on $X$ by considering the fundamental groups of $G$ and $K$. For brevity, write $\pi_1(G) := \pi_1(G;e)$ etc. In the long exact sequence 
\begin{equation}\label{eq:lespi}
\cdots \rightarrow \pi_1(K) \rightarrow \pi_1(G) \stackrel{\pi_*}{\longrightarrow} \pi_1(\mc{O}) \rightarrow \pi_0(K) \rightarrow 1,
\end{equation}
the local system $\mc{L} \in \Lmod{\pi_1(\mc{O})}$ is $(G,q)$-monodromic if and only if $\pi^* \mc{L} \simeq q^{\oplus \dim \mc{L}}$. In general, one sees that $(G,q)$-monodromic local systems do not correspond to representations of the component group $\pi_0(K)$. 

The restriction maps $\mathbb{X}^*(G) \rightarrow \mathbb{X}^*(K)$ and $\mathbb{X}^*(\mf{g}) \rightarrow \mathbb{X}^*(\mf{k})$, where $\mf{k}$ is the Lie algebra of $K$, induce a map $\mathbb{T}(G) \rightarrow \mathbb{T}(K)$.

\begin{lemma}\label{lem:qmonodromic}
If $K$ is connected, then there exist $(G,q)$-monodromic local systems on $\mc{O}$ if and only if the image of $q$ in $\mathbb{T}(K)$ is  $1$. 
\end{lemma}

\begin{proof}
First we recall some standard facts about $\pi_1(G)$. If $R_u(G)$ is the unipotent radical of $G$, then let $G^{\mathrm{red}} = G / R_u(G)$ and $G' = [G^{\mathrm{red}},G^{\mathrm{red}}]$ the derived subgroup of $G^{\mathrm{red}}$. Then $\pi_1(G)$ is a finitely generated abelian group with
$$
\pi_1(G)_{\mathrm{tor}} = \pi_1\left(G' \right), \quad \pi_1(G) /  \pi_1(G)_{\mathrm{tor}} = \pi_1\left(G^{\mathrm{red}} / G' \right). 
$$
In particular, this implies that 
$$
\mathbb{T}(G)  = \Hom\left(\pi_1\left(G^{\mathrm{red}} / G' \right),\Cs \right),
$$
parameterizes isomorphism classes of irreducible $\pi_1(G)$-modules $\mc{M}$ such that $\mc{M} |_{\pi_1(G)_{\mathrm{tor}}}$ is trivial. 

If $\mc{L}$ is an irreducible $(G,q)$-monodromic local system on $\mc{O}$, then its pull-back to $G$ is an irreducible $(G,q)$-monodromic local system (since the quotient map is smooth, with connected fibers). In other words, the pull-back of $\mc{L}$ is $q$. Thus, if $\pi_0(K) = 1$, then we deduce from the long exact sequence (\ref{eq:lespi}) that there exists a $(G,q)$-monodromic local system on $\mc{O}$ if and only if the irreducible representation $q$ of $\pi_1(G)$ restricts to the trivial representation of $\pi_1(K)$. 

As noted above, $\pi_1 (G)_{\mathrm{tor}}$ acts trivially on $q$. Therefore $\pi_1(K)_{\mathrm{tor}}$ also acts trivially. This means that $q$ restricts to a representation of $\pi_1(K^{\mathrm{red}} / K')$, and is the trivial representation of this group if and only if the image of $q$ in $\mathbb{T}(K)$ is  $1$.
\end{proof}

If there exist $(G,q)$-monodromic local systems on $\mc{O}$ when $K$ is connected, then the category of all $(G,q)$-monodromic local systems on $\mc{O}$ is just $\mathrm{Vect}$ i.e. is semi-simple with one simple object. 

\begin{proposition}\label{prop:qmonodromicqc}
	If $G$ is connected reductive and $K \subset G$ connected, then there exists a $(G,\chi)$-monodromic quasi-coherent $\dd$-module on $\mc{O} = G / K$ if and only if the image of $q$ in $\mathbb{T}(K)$ equals $1$. 
\end{proposition}

\begin{proof}
	Since $G$ is reductive, it is easily checked that $\mc{O}_G^{\chi}$ is regular holonomic. Then the fact that the map $G \rightarrow G / K$ is smooth implies, by \cite[Proposition 12.9]{BorelDmod} and the proof of \cite[Theorem 11.6.1]{HTT} that every coherent $(G,\chi)$-monodromic $\dd$-module on $\mc{O}$ is regular holonomic. Finally, we note that every $(G,\chi)$-monodromic quasi-coherent $\dd$-module on $\mc{O}$ is the colimit of coherent $(G,\chi)$-monodromic $\dd$-modules on $\mc{O}$. Thus, we deduce that the proposition is a consequence of Lemma \ref{lem:qmonodromic}.
\end{proof} 

\subsection{Duality for monodromic $\dd$-modules}\label{sec:dualitymonodromic}

Recall from (\ref{eq:ochiG}) that $v_{\chi}$ is the canonical generator of $\mc{O}_G^{\chi}$. We have identified $\mf{g}$ with right-invariant vector fields on $G$ via the morphism $\nu_L: \mf{g} \rightarrow \Theta(G)$. Fix throughout $N = \dim G$. Dual to $\Gamma(G,\Theta_G)^G$ is the space $\Gamma(G,\Omega^1_G)^G$ of right-invariant one-forms on $G$, which we can identify with $\mf{g}^*$. Define the \textit{modular character} of $\mf{g}$ to be $\delta: \mf{g} \rightarrow \C$, $\delta(x) = \Tr \ \mathrm{ad}(x)$. We note that if $\mf{g}$ is reductive or nilpotent then $\delta = 0$. The Lie algebra $\mf{g}$ acts on $\bigwedge^{N} \mf{g}$ by 
$$
x \cdot x_1 \wedge \cdots \wedge x_N = \sum_{i = 1}^N x_1 \wedge \cdots \wedge [x,x_i] \wedge \cdots \wedge x_N. 
$$
Recall that $\Theta(G)$ acts on $\Gamma(G,\Omega^N_G)$ by Lie derivatives, $\omega \mapsto -L_x \omega$, for $x \in \Theta(G)$.

\begin{lemma}\label{lem:topwedgegact}
	Fix non-zero $s \in \bigwedge^{N} \mf{g}$ and $\omega \in \Gamma(G,\Omega^N_G)^G$. Then
	$$
	x \cdot s= \delta(x) s, \quad \textrm{and} \quad \omega \cdot x= - L_x \omega  = \omega \delta(x), \quad \forall \ x \in \mf{g}. 
	$$ 
\end{lemma}

\begin{proof}
The Jacobi identity implies that $x \cdot s = \psi(x) s$ for some character $\psi$. Moreover, the fact that $\dim \bigwedge^{N} \mf{g} =1$ implies that $\psi$ is independent of $s$. Fix $x \in \mf{g}$ and choose an ordered basis $\{ x_1, \ds, x_N \}$ of $\mf{g}$ such that $[x,x_i]  = a_i x_i + y$, where $y \in \C\{ x_j \ | \ j > i \}$. Then $\Tr\ \mathrm{ad}(x) = \sum_{i = 1}^N a_i$. Also, 
$$
x \cdot s = \sum_{i = 1}^N x_1 \wedge \cdots \wedge [x,x_i] \wedge \cdots \wedge x_N  = \left( \sum_{i = 1}^N a_i \right) x_1 \wedge \cdots \wedge x_N = \delta(x) s.
$$
If $\omega \in \Gamma(G,\Omega^N_G)^G$ and $x \in \mf{g}$, then the action of $x$ on $\omega$ is dual to the action of $x$ on $s \in \bigwedge^{N} \mf{g}$. Thus, $L_x \omega = - \delta(x) \omega$. Notice that this means that considered as a section of the right $\dd_G$-module $\Omega_G^N$, we have $\omega \cdot x= \omega \delta(x)$. 	
\end{proof}

\begin{lemma}\label{lem:Afreemfgmod}
	Let $A$ be an associative algebra, with homomorphism $\nu : \mf{g} \rightarrow A$ of Lie algebras making $A$ into a flat $\mf{g}$-module. Let $\nu_{\delta} : \mf{g} \rightarrow A$, with $\nu_{\delta}(x) = \nu(x) + \delta(x)$. Then $$
	\Ext_{A}^{i}(A / A \mf{g},A) = 0 \quad \textrm{for $i \neq N$ and} \quad \Ext_{A}^{N}(A / A \mf{g},A) \simeq \nu_{\delta}(\mf{g}) A \setminus A,
	$$
	as right $A$-modules. 
\end{lemma}

\begin{proof}
To compute $\Ext_{A}^{i}(A / A \mf{g},A)$, we explicitly resolve $A / A \mf{g}$. The fact that $A$ is a flat $\mf{g}$-module implies that the Chevalley-Eilenberg complex, 
$$
\begin{tikzcd}
0 \ar[r] & A \otimes \bigwedge^N \mf{g} \ar[r] & \cdots  \ar[r] & A \otimes \mf{g} \ar[r] & A \ar[r] & 0, 
\end{tikzcd}
$$
is exact except at the very right. Applying the contravariant functor $\Hom_{A}( - , A)$ gives 
$$
0 \rightarrow \Hom_{A}(A , A)  \rightarrow \Hom_{A}(A \otimes \mf{g} , A) \rightarrow \cdots \rightarrow \Hom_{A}(A \otimes \bigwedge^N \mf{g} , A) \rightarrow 0. 
$$
or, 
$$
\begin{tikzcd}
0 \ar[r] & A  \ar[r] & \mf{g}^* \otimes A \ar[r] & \cdots  \ar[r] & \bigwedge^N \mf{g}^* \otimes A \ar[r] & 0. 
\end{tikzcd}
$$
This is everywhere exact, except at the very left, where the corresponding cohomology group is $\Ext_{A}^{N}(A / A \mf{g},A) = H^N(\mf{g},A)$. By Poincar\'e duality, \cite[Theorem 6.10]{KnappCohomology}, we have 
$$
H^N(\mf{g},A) = H_0\left(\mf{g},\bigwedge^N \mf{g}^* \otimes A \right) = \nu(\mf{g}) \cdot \left(\bigwedge^N \mf{g}^* \otimes A \right) \setminus \bigwedge^N \mf{g}^* \otimes A.
$$
If $s \in \bigwedge^N \mf{g}^*$ is any non-zero section, then Lemma \ref{lem:topwedgegact} implies that $x \cdot s = - \delta(x) s$. Thus, there is a canonical isomorphism of right $A$-modules, 
$$
\nu_{\delta}(\mf{g}) A \setminus A \iso \nu(\mf{g}) \cdot \left(\bigwedge^N \mf{g}^* \otimes A \right) \setminus \bigwedge^N \mf{g}^* \otimes A, \quad 1 \mapsto s \otimes 1. 
$$
\end{proof}

The full subcategory of $\QCoh(\dd_X,G,\chi)$ consisting of all holonomic $\dd$-modules is denoted $\Hol(G,\dd_X,\chi)$. We denote the usual exact duality on holonomic $\dd$-modules by $\mathbb{D}$; explicitly, 
$$
\mathbb{D}(\ms{M}) = \mc{E}xt_{\dd}^{\dim X}\left(\ms{M},\dd_X \otimes_{\mc{O}} \Omega_X^{\otimes -1}\right).
$$ 

\begin{proposition}\label{prop:dualityOchiG}
There is a canonical isomorphism $\mathbb{D}(\mc{O}^{\chi}_G) \stackrel{\sim}{\longrightarrow} \mc{O}_G^{- \chi}$. 
\end{proposition}

\begin{proof}
Let ${}_G^{\chi} \mc{O}$ be the right $\dd_G$-module generated by $v_{\chi}'$, with defining relation $v_{\chi}' \nu_L(x) = v_{\chi}' \chi(x)$ for all $x \in \mf{g}$. Applying Lemma \ref{lem:Afreemfgmod}, where $\dd(G)$ is a right $\mf{g}$-module via $(\nu_L - \chi)$, we deduce that 
$$
\Ext_{\dd}^N(\mc{O}_G^{\chi},\dd) \simeq {}_G^{\chi - \delta} \mc{O}. 
$$
Therefore the map 
$$
\Hom_{\dd}(\dd,\dd) \rightarrow {}_G^{\chi} \mc{O}, \quad \phi \mapsto v_{\chi - \delta}' \phi(1),
$$
descends to an isomorphism $\Ext^N_{\dd}(\mc{O}_G^{\chi},\dd) \iso {}_G^{\chi - \delta} \mc{O}$. Thus, it suffices to check that the rule $v_{\chi - \delta}' \otimes s \mapsto v_{- \chi}$ defines a morphism of left $\dd$-modules ${}_G^{\chi - \delta} \mc{O} \otimes_{\mc{O}_G} \Omega_G^{\otimes -1} \rightarrow \mc{O}_G^{\chi}$; being non-zero, such a morphism must be an isomorphism. Explicitly, 
\begin{align*}
\nu_L(x) \cdot (v_{\chi - \delta}' \otimes s) & = - (v_{\chi - \delta}' \cdot \nu_L(x) ) \otimes s + v_{\chi - \delta}' \otimes (\nu_L(x) \cdot s) \\
 &  = (\delta - \chi)(x) (v_{\chi}' \otimes s) - \delta (x) (v_{\chi}' \otimes s) = - \chi(x) (v_{\chi}' \otimes s),  
\end{align*}
since $\nu_L(x) \cdot s = - \delta(x) s$ by Lemma \ref{lem:topwedgegact}. 
\end{proof}

\begin{proposition}\label{prop:dualholo}
Duality lifts to a contravariant equivalence 
$$
\mathbb{D} : \Hol(G,\dd_X,\chi) \stackrel{\sim}{\longrightarrow} \Hol(G,\dd_X,-\chi). 
$$
\end{proposition}

\begin{proof}
The key point here is that duality commutes with pull-back for non-characteristic modules. That is, if $f : X \rightarrow Y$ and $\ms{M}$ is non-characteristic for $f$, then the \textit{canonical} morphism $f^* \mathbb{D}(\ms{M}) \rightarrow  \mathbb{D}(f^* \ms{M})$ is an isomorphism; see \cite[Theorem 2.7.1 (ii)]{HTT}. Noting that each of the morphisms $m$ and $a$ are smooth (and hence all modules are non-characteristic for them), applying $\mathbb{D}$ to diagram (\ref{eq:cocycle}), and using Proposition \ref{prop:dualityOchiG} shows that the cocycle condition holds. Similarly, the fact that $\mc{O}_G^{\chi}$ is non-characteristic for $i_e : \{ \mathrm{pt} \} \hookrightarrow G$, shows that the rigidity condition holds for $\mathbb{D}(\ms{M})$. Thus, $\mathbb{D}(\ms{M})$ is $(G,-\chi)$-monodromic if $\ms{M}$ is $(G,\chi)$-monodromic.
\end{proof}

\subsection[The fundamental group of G-orbits]{The fundamental group of $G$-orbits}\label{sect:fund_group}

The goal of this section is to compute the fundamental groups of the $\GL_{\mbf{d}}$-orbits in the space $\Rep(\C \Q / I, \mbf{d})$. We do this for an arbitrary quiver $\Q$ with relations $I \subset \C \Q$. Combining this result with the classification of orbits in the enhanced cyclic nilpotent cone given in \cite{BB-enhnil-quiver}, we are able to explicitly compute the fundamental group of the orbits in the cone. Recall that a dimension vector $\mbf{d}$ is called \textit{sincere} if $d_i \neq 0$ for all $i \in \Q_0$. We assume, without loss of generality, that $\mbf{d}$ is sincere. 

First we recall the following standard fact, see \cite[Proposition 2.2.1]{BrionQuiver}. Let $M$ be a representation of $\C \Q / I$ with dimension vector $\mbf{d}$. Then,  
\begin{equation}\label{eq:Sconnected}
\Aut_{\Q}(M) = \mathrm{Stab}_{\GL_{\mbf{d}}}(M) \textrm{ is connected.}
\end{equation}  
Therefore the long exact sequence in homotopy groups associated to the fibre bundle 
$$
\begin{tikzcd}
\Aut_{\Q}(M) \ar[r,hook] & \GL_{\mbf{d}} \twoheadrightarrow \mathcal{O},
\end{tikzcd}
$$
where $\mathcal{O} = \GL_{\mbf{d}} \cdot M$, implies that it suffices to compute $\pi_1(\Aut_{\Q}(M))$ and describe the map $\pi_{1}(\Aut_{\Q}(M)) \rightarrow \pi_1(\GL_{\mbf{d}})$. We begin with some preparatory results. Let $E = \End_{\Q}(M)$ and notice that $\Aut_{\Q}(M)$ equals the group $E^{\times}$ of invertible elements in $E$. Decompose 
$$
M = \bigoplus_{i = 1}^k M^{(i)} \otimes V_i,
$$
where the $M^{(i)}$ are pairwise non-isomorphic, indecomposable $\C \Q/I$-modules. Then $E$ decomposes as
\begin{equation}\label{eq:Edecomp}
E = \bigoplus_{i, j = 1}^k \Hom_{\Q}(M^{(i)},M^{(j)}) \otimes \Hom_{\C}(V_i,V_j). 
\end{equation} 
By \cite[Proposition 2.2.1 (ii)]{BrionQuiver}, we have 

\begin{lemma}\label{lem:Erad}
The embedding $\gamma : \prod_{i = 1}^k \GL(V_i) \hookrightarrow E^{\times}$ induces an identification 
$$
\begin{tikzcd}
\gamma_{*} :  \pi_1\left( \prod_{i = 1}^k \GL(V_i) \right) \ar[r,"\sim"] & \pi_1(E^{\times}).
\end{tikzcd}
$$
\end{lemma}

In particular, $\pi_1(E^{\times}) = \Z^k$. Next, we calculate the morphism $\pi_1(E^{\times}) \rightarrow \pi_1(\GL_{\mbf{d}})$. First we note that the embedding $\prod_{i = 1}^k \GL(V_i) \hookrightarrow E^{\times}$ induces an isomorphism of fundamental groups because the composite
$$
\begin{tikzcd}
\prod_{i = 1}^k \GL(V_i) \ar[r,hook] & E^{\times} \ar[r,"\eta",two heads] & \prod_{i = 1}^k \GL(V_i) 
\end{tikzcd}
$$
is the identity. Therefore, it suffices to compute the map $\Z^k \rightarrow \pi_1(\GL_{\mbf{d}})$ induced by the embeddings $\prod_{i = 1}^k \GL(V_i) \hookrightarrow E^{\times} \hookrightarrow \GL_{\mbf{d}}$.  If $M_j$ denotes the vector space of $M$ at vertex $j$, then $M_j = \bigoplus_{i = 1}^k M_j^{(i)} \otimes V_i$ and the composite $\prod_{i = 1}^k \GL(V_i) \hookrightarrow \GL_{\mbf{d}} \twoheadrightarrow \GL_{d_j}$ identifies $\prod_{i = 1}^k \GL(V_i)$ with the reductive subgroup $\prod_{i = 1}^k \GL(V_i) \otimes \mathrm{Id}_{M_j^{(i)}}$ of $\GL_{d_j}$. Hence, the map
$$
\begin{tikzcd}
\Z^k =\pi_1\left( \prod_{i = 1}^k \GL(V_i) \right) \ar[r] & \pi_1( \GL_{d_j}) = \Z
\end{tikzcd}
$$
is given by $(\beta_j^{(1)}, \dots, \beta_j^{(k)})$, where $\beta^{(i)}$ is the dimension vector of $M^{(i)}$. Since we have assumed that $\mbf{d}$ is sincere, $\pi_1(\GL_{\mbf{d}}) = \Z^{| \Q_0 |}$. We have shown that: 

\begin{proposition}\label{thm:cokerfun}
Assume that $\mbf{d}$ is sincere. The fundamental group of the $\GL_{\mbf{d}}$-orbit of $M$ is given by the cokernel of the map 
\begin{equation}\label{eq:cokerpi}
\begin{tikzcd}
\Z^k \ar[r,"B"] & \Z^{|\Q_0|},
\end{tikzcd}
\end{equation}
where $B = \left( \beta_j^{(i)} \right)$ for $i = 1, \dots, k$ and $j \in \Q_0$. 
\end{proposition}

In particular, Proposition \ref{thm:cokerfun} shows that the fundamental group of $\mathcal{O}$ only depends on the combinatorial data of the dimension vector of the indecomposable summands of $M$. In applications to the enhanced cyclic nilpotent cone, we will need a minor refinement of Proposition \ref{thm:cokerfun}. Namely, the quiver $\Q$ is assumed to have a preferred vertex $\infty \in \Q_0$ such that $\mbf{d}_{\infty} = 1$. In this case, if  
\[ G' := \prod_{i \neq \infty} \GL_{d_i} \subset \GL_{\mbf{d}}\]

then the $\GL_{\mbf{d}}$-orbits in $\Rep (\C \Q/I, \mbf{d})$ are the same as the $G'$-orbits. 
There is a unique $i$ such that $\beta^{(i)}_{\infty} \neq 0$. Let $\Q_0' := \Q_0 \smallsetminus\{ \infty \}$ and $\ell:= |\Q_0'|$. We may assume without loss of generality that $i = k$. Since $\beta_{\infty}^{(k)} = 1$, we can disregard the indecomposable summand $M^{(k)}$ of $M$ when applying Proposition \ref{thm:cokerfun}. More specifically, if $\mbf{d}' = \mbf{d} - \varepsilon_{\infty}$ and $B'$ is the matrix $( \beta_j^{(i)})$ for $i = 1, \dots, k-1$ and $j \in \Q_0'$ then the fundamental group of the $\GL_{\mbf{d}}$-orbit of $M$ is given by the cokernel of the map 
\begin{equation}\label{eq:cokerpiinfty}
\begin{tikzcd}
\Z^{k-1} \ar[r,"{B'}"] & \Z^{\ell}.
\end{tikzcd}
\end{equation}

Since Proposition \ref{thm:cokerfun} implies that the fundamental group of every orbit is abelian, the irreducible complex representations of $\pi_1(\mathcal{O})$ are all one dimensional. If $q \in (\C^{\times})^{\ell}$ and $\alpha \in \Z^{\ell}$ then $q^{\alpha}$ denotes the complex number $q^{\alpha_1}_1 \cdots q_{\ell}^{\alpha_{\ell}}$. Notice that the irreducible complex representations of $\Z^{\ell}$ are parametrized by $(\C^{\times})^{{\ell}}$. 

\begin{corollary}\label{cor:irrepO}
The irreducible complex representations of $\pi_1(\GL_{\mbf{d}} \cdot M)$ are given by 
$$
\left\{ q \in (\C^{\times})^{{\ell}} \ \left| \ q^{\beta^{(i)}} = 1, \ \forall \ i = 1, \ds, k-1 \right\}. \right. 
$$
\end{corollary}

\begin{proof}
The representation $q$ of $\Z^{\ell}$ descends to a representation of $\pi_1(\GL_{\mbf{d}} \cdot M)$ if and only if it is identically one on the kernel of the surjection $\Z^{\ell} \twoheadrightarrow \pi_1(\GL_{\mbf{d}} \cdot M)$ . By Proposition \ref{thm:cokerfun}, this kernel is generated by the $\beta^{(i)}$. 
\end{proof}

\section{Quantum Hamiltonian reduction}\label{sec:QHR}

Throughout this section, we assume that $X$ is an arbitrary smooth \textit{affine} $G$-variety, with $G$ a connected reductive group. For each character $\chi \in \mathbb{X}^*(\mf{g})$ and finite dimensional $G$-module $U$, one can define an algebra of quantum Hamiltonian reduction $\mf{A}_{\chi}(U)$, which sheafifies over the base $X /\!/ G$. We study how the various $\mf{A}_{\chi}(U)$ are related when one varies the choice of representation $U$. Our motivation for introducing different $U$ is clear - the major deficiently of the usual functor of Hamiltonian reduction is that its kernel is non-trivial. By choosing $U$ large enough one can ensure that the kernel becomes zero.

\subsection{Quantum Hamiltonian reduction}\label{sec:QHRUbasic} Fix a finite-dimensional $G$-module $U$. Differentiating the $G$-action gives an algebra map $\Phi_U : U(\g) \rightarrow \End_{\C}(U)$. Let $\dd(X) = \Gamma(X,\dd_X)$. For $\chi \in \mathbb{X}^*(\mf{g})$, define $\nu_{\chi} : \mf{g} \rightarrow \dd(X)$ by $\nu_{\chi}(x) = \nu(x) - \chi(x)$. The algebra of \textit{quantum Hamiltonian reduction} associated to $U$ is defined to be
$$
\mf{A}_{\chi}(U) := \left( \frac{\dd(X) \otimes \End_{\C}(U)}{\dd(X) \otimes \End_{\C}(U) \cdot \g_{\chi}} \right)^G,
$$
where $\g_{\chi} = \{ \nu_{\chi}(x) \otimes 1 + 1 \otimes \Phi_U(x)  \  | \ x \in \g \}$ in $\dd(X) \otimes \End_{\C}(U)$. 

\begin{lemma}\label{lem:endU}
Let $U$ be a finite-dimensional $G$-module. 
\begin{enumerate}
\item[(a)] The $\dd_X$-module 
$$
\ms{P}_{\chi}(U) := \frac{\dd_X \otimes U}{\dd_X \cdot \{ \nu_{\chi}(x) \otimes u + 1 \otimes \Phi_U(x)(u) \  | \ x \in \g, u \in U \} },
$$
with the natural $G$-action, is $(G,\chi)$-monodromic. It is projective in $\QCoh(\dd_X,G,\chi)$. 
\item[(b)] $\End_{\dd_{X}}(\ms{P}_{\chi}(U))^{\mathrm{op}} = \mf{A}_{\chi}(U)$.
\end{enumerate}
\end{lemma}

\begin{proof}
Part (a) is well-known, see \cite[Lemma 4.3]{MNDerived}. For part (b), we give a detailed proof. The group $G$ acts diagonally on $\dd(X) \otimes \End_{\C}(U)$ and by differentiating this action we get an action of $\mf{g}$ on the algebra. Explicit, for $x \in \mf{g}$, $D \in \dd(X)$ and $E \in \End_{\C}(U)$, 
\begin{align}
x \cdot (D \otimes E) & = \frac{d}{d t} \exp(t x) \cdot \left( D \otimes E \right) \Big|_{t = 0} \nonumber \\
& = [\nu(x), D] \otimes E + D \otimes [E, \Phi_U(x)]. \label{eq:diffGactionDoEnd}
\end{align}
First, we define the isomorphism  
$$
\phi : \dd(X) \otimes \End_{\C}(U)^{\mathrm{op}} \rightarrow \End_{\dd_{X}}(\dd(X) \otimes U)^{\mathrm{op}},
$$
given by 
$$
\phi \left( \sum_i D_i \otimes E_i \right) \left(\sum_j F_j \otimes u_j \right) = \sum_{i,j} F_j D_i \otimes E_i(u_j). 
$$
Let 
$$
I = \dd_X \cdot \{ \nu_{\chi}(x) \otimes u + 1 \otimes \Phi_U(x)(u) \  | \ x \in \g, u \in U \},
$$
so that $\ms{P}_{\chi}(U) $ is the quotient of $\dd(X) \otimes U$ by $I$. 

\begin{claim}\label{claim:commute}
If $\sum D_i \otimes E_i \in (\dd(X) \otimes \End_{\C}(U)^{\mathrm{op}})^G$, then $\phi(\sum D_i \otimes E_i)(I) \subset I$ . 
\end{claim}

\begin{proof}
Assume that $\sum D_i \otimes E_i \in \dd(X) \otimes \End_{\C}(U)^{\mathrm{op}}$. If $x \in \g$ and $u \in U$, then 
\begin{align*}
\phi \left( \sum_i D_i \otimes E_i \right) ( \nu_{\chi}(x) \otimes u + 1 \otimes \Phi_U(x)(u)) & = \sum_i D_i \nu_{\chi}(x) \otimes E_i(u) + D_i \otimes \Phi_U(x)(E_i(u)) \\
 & + \sum_i [\nu_{\chi}(x), D_i] \otimes E_i (u) + D_i \otimes [E_i, \Phi_U(x)](u) \\
 & = \sum_i D_i \nu_{\chi}(x) \otimes E_i(u) + D_i \otimes \Phi_U(x)(E_i(u))\\
 &  + \sum_i [\nu(x), D_i] \otimes E_i (u) + D_i \otimes [E_i, \Phi_U(x)](u).
\end{align*}
Moreover, by equation \eqref{eq:diffGactionDoEnd},
$$
\sum_i [\nu(x), D_i] \otimes E_i (u) + D_i \otimes [E_i, \Phi_U(x)](u) = x \cdot \left( \sum_i D_i \otimes E_i \right) (1 \otimes u). 
$$
Thus, 
\begin{align}
\phi \left( \sum_i D_i \otimes E_i \right) ( \nu_{\chi}(x) \otimes u + 1 \otimes \Phi_U(x)(u)) & = \sum_i D_i \nu_{\chi}(x) \otimes E_i(u) + D_i \otimes \Phi_U(x)(E_i(u)) \nonumber \\
& \ + x \cdot \left( \sum_i D_i \otimes E_i \right) (1 \otimes u).  \label{eq:phiImapI}
\end{align}
This is zero if $\sum_i D_i \otimes E_i \in (\dd(X) \otimes \End_{\C}(U)^{\mathrm{op}})^G$. This completes the proof of the claim. 
\end{proof}

Therefore $\phi$ descends to a well-defined map $\psi : (\dd(X) \otimes \End_{\C}(U)^{\mathrm{op}})^G \rightarrow \End_{\dd_{X}}(\ms{P}_{\chi}(U))^{\mathrm{op}}$. Next, we check that 
$$
\psi \left(\left(\dd(X) \otimes \End_{\C}(U)^{\mathrm{op}} \cdot \g_{\chi}\right)^G \right) = 0.
$$ 
We have
$$
\psi\left( \sum_{i,\alpha} (D_i \otimes E_i) (\nu_{\chi}(x_{\alpha}) \otimes 1 + 1 \otimes \Phi_U(x_{\alpha}))\right)
$$
$$
= \psi\left( \sum_{i,\alpha}  D_i \nu_{\chi}(x_{\alpha}) \otimes E_i + D_i \otimes \Phi_U(x_{\alpha}) \circ E_i \right) 
$$
$$
=  \sum_{i,\alpha}  D_i \nu_{\chi}(x_{\alpha}) \otimes E_i + D_i \otimes \Phi_U(x_{\alpha}) \circ E_i = 0.
$$
This shows that $\psi$ factors through $\mf{A}_{\chi}(U)$. 

Now we show that $\psi$ is surjective. Each endomorphism $f \in \End_{\dd_{X}}(\ms{P}_{\chi}(U))^{\mathrm{op}}$ is uniquely defined by what it does on $1 \otimes U$. Therefore there exist $D_i \in \dd(X)$ and $E_i \in \End_{\C}(U)^{\mathrm{op}}$ such that $f(F \otimes u) = \sum_i F D_i \otimes E_i(u)$ for all $F \otimes u \in \ms{P}_{\chi}(U)$. The element $E = \sum D_i \otimes E_i \in \dd(X) \otimes \End_{\C}(U)^{\mathrm{op}}$ must satisfy $\phi(E)(I) \subset I$. Then equation \eqref{eq:phiImapI} in the proof of Claim \ref{claim:commute} shows that this implies that 
$$
\phi(x \cdot E)(1 \otimes u) \in I, \quad \forall \, u \in U \, x \in \g,
$$
where $x \cdot E$ is defined in \eqref{eq:diffGactionDoEnd}. Thus, $E$ belongs to the rational $G$-representation 
$$
H := \{ F \in \dd(X) \otimes \End_{\C}(U)^{\mathrm{op}} \ | \ \phi(x \cdot F)(1 \otimes U) \subset I \ \forall \ x \in \g \}.
$$
If $E_0$ is the invariant component of $E$, then $E - E_0 \in \g \cdot H$. But, by definition, any element in $\g \cdot H$ induces the trivial endomorphism of $\End_{\dd_{X}}(\ms{P}_{\chi}(U))^{\mathrm{op}}$. Thus, $E_0 \in (\dd(X) \otimes \End_{\C}(U)^{\mathrm{op}})^G$ induces the endomorphism $f$ in $\End_{\dd_{X}}(\ms{P}_{\chi}(U))^{\mathrm{op}}$. That is, $\psi$ is surjective. 

Finally, we show $\psi$ is injective. Assume that we are given $E \in (\dd(X) \otimes \End_{\C}(U)^{\mathrm{op}} )^G$ such that $\phi(E)(u) \in I$ for all $u \in U$. Fix a basis $u_1, \ds, u_r$ of $U$. Then 
$$
\phi(E)(u_j) = \sum_{i,\alpha} D_{i,j,\alpha} (\nu_{\chi}(x_{\alpha}) \otimes u_i + 1 \otimes \Phi_U(x_{\alpha})(u_i))
$$
for some (non-unique) $D_{i,j,\alpha} \in \dd(X)$. If $E_{i,j} \in \End_{\C}(U)$ satisfies $E_{i,j}(u_j) = u_i$ and $E_{i,j}(u_k) = 0$ for $j \neq k$, then 
$$
\phi \left( \sum_{i,j,\alpha} (D_{i,j,\alpha} \otimes E_{i,j}) (\nu_{\chi}(x_{\alpha}) \otimes 1 + 1 \otimes \Phi_U(x_{\alpha}))  \right)(u)  = \phi(E)(u).
$$
Since the map $\phi$ is an isomorphism, this implies that $E = \sum (D_{i,j,\alpha} \otimes E_{i,j}) (\nu_{\chi}(x_{\alpha}) \otimes 1 + 1 \otimes \Phi_U(x_{\alpha}))$. Hence $E = 0$ as an element of $\mf{A}_{\chi}(U)$.
\end{proof}

Just as in the usual case where $U = \C$ is the trivial $G$-module, modules over the algebra $\mf{A}_{\chi}(U)$ are very closely related to $(G,\chi)$-monodormic $\dd$-modules. The functor of \textit{Hamiltonian reduction} is the exact functor $\Ham_{U,\chi} : \Coh(\dd_{X},G,\chi) \rightarrow \Lmod{\mf{A}_{\chi}(U)}$ defined by 
$$
\Ham_{U,\chi}(\ms{M}) = (\Gamma(X,\ms{M}) \otimes U^*)^G = \Hom_{\dd_{X}}(\ms{P}_{\chi}(U),\ms{M}). 
$$
It admits a left adjoint ${}^{\perp} \Ham_{U,\chi}$ given by 
$$
{}^{\perp} \Ham_{U,\chi}(N) = \ms{P}_{\chi}(U) \otimes_{\mf{A}_{\chi}(U)} N,
$$
such that the canonical adjunction $\mathrm{Id}_{\Lmod{\mf{A}_{\chi}(U)}} \rightarrow \Ham_{U,\chi} \circ {}^{\perp} \Ham_{U,\chi}$ is an isomorphism. This implies that $\Ham_{U,\chi}$ is essentially surjective and that the kernel of $\Ham_{U,\chi}$ is the Serre subcategory of $\Coh(\dd_X,G,\chi)$ consisting of all objects $\ms{M}$ such that $\Ham_{U,\chi}(\ms{M}) = 0$.

For a pair $U_1, U_2$ of finite-dimensional $G$-modules, define 
$$
\mf{A}_{\chi}(U_1,U_2) = \left( \frac{\dd(X) \otimes \Hom_{\C}(U_2,U_1)}{\dd(X) \otimes \End_{\C}(U_1) \cdot \mf{g}_{\chi}(U_1,U_2) } \right)^G,
$$
where
$$
\mf{g}_{\chi}(U_1,U_2) := \{ \nu_{\chi}(x) \otimes F + 1 \otimes F \Phi_{U_2}(x) \ | \ x \in \g, F \in \Hom_{\C}(U_2,U_1) \}.
$$

\begin{lemma}\label{lem:Moritacontext}
	The space $\mf{A}_{\chi}(U_1,U_2) $ is a $\mf{A}_{\chi}(U_1)$-$\mf{A}_{\chi}(U_2)$-bimodule and if $U = U_1 \oplus U_2$ then 
	\begin{equation}\label{eq:matrixendU}
	\mf{A}_{\chi}(U) =  \left( \begin{array}{cc} 
	\mf{A}_{\chi}(U_1) & \mf{A}_{\chi}(U_1,U_2) \\
	\mf{A}_{\chi}(U_2,U_1)  & \mf{A}_{\chi}(U_2)
	\end{array} \right) 
	\end{equation}
	\vspace{4mm}
	is a Morita context (in the sense of \cite[\S 1.1.6]{MR}).
\end{lemma}

\begin{proof} 
It is a bi-module because 
$$
\sum_{i} [\nu(x),D_i] \otimes F_i + D_i \otimes (\Phi_{U_1}(x) F_i - F_i \Phi_{U_2}(x)) =  \frac{d}{d t} \exp(t x) \cdot \left( \sum_i D_i \otimes F_i \right) \Big|_{t = 0} = 0
$$
for $x \in \g$ and $\sum_{i} D_i \otimes F_i \in (\dd(X) \otimes \Hom_{\C}(U_2,U_1))^G$. 
It is straight-forward to check that $\mf{A}_{\chi}(U)$ is a Morita context.
\end{proof}

\begin{lemma}\label{lem:leftrightnoetherian}
	Fix representations $U_1,U_2$ and $U$. 
	\begin{enumerate}
		\item[(a)] The algebra $\mf{A}_{\chi}(U)$ is noetherian.
		\vspace{2mm}
		\item[(b)] The bimodule $\mf{A}_{\chi}(U_1,U_2)$ is finitely generated \textit{both} as a left $\mf{A}_{\chi}(U_1)$-module \textit{and} as a right $\mf{A}_{\chi}(U_2)$-module.   
	\end{enumerate}
\end{lemma}

\begin{proof}
	Extend the order filtration on $\dd(X)$ to a filtration on $\dd(X) \otimes \End_{\C}(U)$ by putting $\End_{\C}(U)$ in degree zero. This induces a filtration $\mc{F}_{\idot}$ on $\mf{A}_{\chi}(U)$ such that 
	$$
	(\C[\mu^{-1}(0)] \otimes \End_{\C}(U))^G \twoheadrightarrow \gr_{\mc{F}} \mf{A}_{\chi}(U)
	$$
	(we do not require that $\mu$ be flat for surjectivity). By Hilbert's Theorem \cite[Zusatz 3.2]{Kraft}, $(\C[\mu^{-1}(0)] \otimes \End_{\C}(U))^G$ is a finite module over $\C[\mu^{-1}(0)]^G$. Since the latter is finite type, the former is noetherian. Since the filtration is exhaustive, we deduce that $\mf{A}_{\chi}(U)$ is noetherian \cite[Theorem 1.6.9]{MR}. 
	
	Similarly, $\mf{A}_{\chi}(U_1,U_2)$ has a filtration, making its associated graded a quotient of the space $(\C[\mu^{-1}(0)] \otimes \Hom_{\C}(U_2,U_1))^G$. Arguing as above, the latter is finitely generated both as a left module over  $(\C[\mu^{-1}(0)] \otimes \End_{\C}(U_1))^G$, and as right a $(\C[\mu^{-1}(0)] \otimes \End_{\C}(U_2))^G$-module. This implies that $\mf{A}_{\chi}(U_1,U_2)$ is finitely generated on both the left and the right. 
\end{proof}

Motivated by Lemma \ref{lem:leftrightnoetherian}, we define a partial ordering on the isomorphism classes of objects in $\Rep(G)$ by saying that $U \le U'$ if $U$ is isomorphic to a direct summand of $U'$. For any associative algebra $A$, we let $\mathrm{Spec}\ A$ denote the topological space (with Zariski topology) of prime ideals in $A$.

\begin{lemma}\label{lem:closedembeddingspecUU}
    If $U_1 \le U_2$ then $\mf{A}_{\chi}(U_1,U_2) \mf{A}_{\chi}(U_2,U_1) = \mf{A}_{\chi}(U_1)$ as $\mf{A}_{\chi}(U_1)$-bimodules and hence there is a closed embedding 
	$$
	\mathrm{Spec} \ \mf{A}_{\chi}(U_1) \hookrightarrow \mathrm{Spec} \ \mf{A}_{\chi}(U_2).
	$$ 
\end{lemma}

\begin{proof}
	If $U_1 \le U_2$, then $U_1$ is a summand of $U_2$ and hence $\mf{A}_{\chi}(U_1)$ is a direct summand of $\mf{A}_{\chi}(U_1,U_2)$ as left $\mf{A}_{\chi}(U_1)$-modules. Similarly, $\mf{A}_{\chi}(U_1)$ is a direct summand of $\mf{A}_{\chi}(U_2,U_1)$ as right $\mf{A}_{\chi}(U_1)$-modules. Hence $\mf{A}_{\chi}(U_1,U_2) \mf{A}_{\chi}(U_2,U_1) = \mf{A}_{\chi}(U_1)$. By Lemma \ref{lem:Moritacontext}, the matrix 
	$$
	\left( \begin{array}{cc} 
	\mf{A}_{\chi}(U_1) & \mf{A}_{\chi}(U_1,U_2) \\
	\mf{A}_{\chi}(U_2,U_1)  & \mf{A}_{\chi}(U_2)
	\end{array} \right) 
	$$
	is a Morita context. Therefore, it follows from \cite[Theorem 3.6.2]{MR} that $\mathrm{Spec} \ \mf{A}_{\chi}(U_1)$ embeds in $\mathrm{Spec}\ \mf{A}_{\chi}(U_2)$.
\end{proof}

\subsection{Generic behaviour}\label{sec:QHRgenericpoint}

In order to better understand the algebras $\mf{A}_{\chi}(U)$, we consider what they look like on some dense open subset of $X$. Let $H \subset G$ be a closed subgroup and assume that there is a $H$-stable closed subvariety $Y \subset X$ such that the canonical map $G \times_H Y\rightarrow X$ is an isomorphism. Write $i : Y \hookrightarrow X$ for the closed embedding. The Lie algebra of $H$ is denoted $\mf{h}$.  

\begin{lemma}\label{lem:nonzerorestopenU}
	For $G,H$ and $i : Y \hookrightarrow X$ as above, 
	\begin{enumerate}
		\item[(a)] $i^* \ms{P}_{\chi}(U) \simeq \ms{P}_{\chi |_\h} (U |_H)$. 
		\vspace{2mm}
		\item[(b)] $\mf{A}_{\chi}(U) \simeq \mf{A}_{\chi |_{\h}}(U |_H)$. 
	\end{enumerate}
	In particular, if $H = 1$ then,
	\begin{enumerate}
		\item[(c)] The algebras $\mf{A}_{\chi}(U)$ are all Morita equivalent to $\dd(Y)$.  
		\vspace{2mm}
		\item[(d)] The bimodules $\mf{A}_{\chi}(U,U')$ are non-zero for all $U,U'$ non-zero.  
	\end{enumerate}
\end{lemma}

\begin{proof}
	Recall that $\Coh(\dd_X,G,\chi)$ denotes the category of $(G,\chi)$-monodromic coherent $\dd$-modules on $X$, and $\Coh(\dd_Y,H,\chi')$ the category of $(H,\chi')$-monodromic coherent $\dd$-modules on $Y$, where $\chi' := \chi|_{\h}$. We recall from the proof of \cite[Proposition 9.1.1]{mirabolicHam} that $i^*$ defines an equivalence $\Coh(\dd_X,G,\chi) \iso \Coh(\dd_Y,H,\chi')$ with quasi-inverse given by $\ms{N} \mapsto (\pi_{\idot}(\mc{O}_G^{\chi} \boxtimes \ms{N}))^H$, where $\pi : G \times Y \rightarrow G \times_H Y = X$ is the quotient map. The $\dd$-module $\ms{P}_{\chi}(U)$ is (up to unique isomorphism) the unique object of $\Coh(\dd_X,G,\chi)$ such that $\Hom_{\Coh(\dd_X,G,\chi)}(\ms{P}_{\chi}(U),\ms{M}') = \Hom_G(U,\Gamma(X,\ms{M}'))$ for all $\ms{M}' \in \Coh(\dd_X,G,\chi)$. The module $\ms{P}_{\chi'}(U |_H)$ is similarly characterized in $\Coh(\dd_Y,H,\chi')$. Therefore, in order to show that (a) holds, it suffices to show that $i^* \ms{P}_{\chi}(U)$ has the appropriate property. Let $\ms{N} \in \Coh(\dd_Y,H,\chi')$. Then,
	\begin{align*}
	\Hom_{\Coh(\dd_Y,H,\chi')}(i^* \ms{P}_{\chi}(U),\ms{N}) & \simeq \Hom_{\Coh(\dd_X,G,\chi)}\left(\ms{P}_{\chi}(U),\left(\pi_{\idot}(\mc{O}_G^{\chi} \boxtimes \ms{N})\right)^H\right) \\
	& =  \Hom_{G}\left(U,\Gamma\left(X,\left(\pi_{\idot}(\mc{O}_G^{\chi} \boxtimes \ms{N})\right)^H\right)\right) \\
	& =  \Hom_{G \times H}\left(U,\Gamma\left(G \times Y,\mc{O}_G^{\chi} \boxtimes \ms{N}\right)\right) \\
	& = \Hom_{G \times H}(U,\mc{O}(G) \otimes \Gamma(Y,\ms{N})) \\
	& = \Hom_{G \times H}\left(U,\bigoplus_{V \in  \mathrm{Irrep}(G)} V \otimes V^* \otimes \Gamma(Y,\ms{N})\right) \\
	& =  \prod_{V \in \mathrm{Irrep}(G)} \Hom_{G \times H}(U, V \otimes V^* \otimes \Gamma(Y,\ms{N})) \\
	& = \Hom_{H}(\C, U^* \otimes \Gamma(Y,\ms{N})) = \Hom_H(U,\Gamma(Y,\ms{N})),
	\end{align*}
	as required. Here $H$ acts trivially on $U$ in the third, fourth and fifth line, and we have used the fact that $\Gamma(G,\mc{O}_G^{\chi}) \simeq \mc{O}(G)$ as a $G$-module. 
	
	The fact that $i^*$ is an equivalence, together with Lemma \ref{lem:endU} and the isomorphism $i^* \ms{P}_{\chi}(U) \simeq \ms{P}_{\chi'} (U |_H)$, implies that  
	$$
	\mf{A}_{\chi}(U) \simeq \End_{\dd_{X}}(\ms{P}_{\chi}(U))^{\mathrm{op}} \simeq \End_{\dd_{Y}}(\ms{P}_{\chi'}(U|_H))^{\mathrm{op}} \simeq \mf{A}_{\chi |_{\h}}(U |_H),
	$$
	as stated in (b). 
	
	In the case where $H = 1$, (a) and (b) imply that $\mf{A}_{\chi}(U) \simeq \dd(Y) \otimes \End_{\C}(U)$ and $\mf{A}_{\chi}(U,U') \simeq \dd(Y) \otimes \Hom_{\C}(U',U)$. Thus, (c) and (d) are trivially true.  
\end{proof}

The algebras $\mf{A}_{\chi}(U)$ naturally sheafify over $X /\!/ G$, and the bimodules $\mf{A}_{\chi}(U,U')$ sheafify to sheaves of $\mf{A}_{\chi}(U)$-$\mf{A}_{\chi}(U')$-bimodules on $X/\!/G$. Recall that an open subset of $X$ is said to be $G$-saturated if it is the preimage, under $\pi: X \rightarrow X/\!/G$ of an open subset of $X/\!/G$.  

\begin{proposition}\label{prop:nonzeroUUbi}
	Assume that there is a $G$-saturated dense open subset of $X$ on which $G$ acts freely. 
	\begin{enumerate}
		\item[(a)] The bimodules $\mf{A}_{\chi}(U,U')$ are non-zero for all $U,U'$ non-zero. 
		\vspace{2mm}
		\item[(b)] Let $S = \{ U \in \Rep(G) \  | \ \mf{A}_{\chi}(U) \textrm{ is simple} \}$. Then $\mf{A}_{\chi}(U)$ is Morita equivalent to $\mf{A}_{\chi}(U')$, for all $U,U' \in S$. 
	\end{enumerate}
\end{proposition}

\begin{proof}
	Part (a). As noted above, the bimodules $\mf{A}_{\chi}(U,U')$ sheafifiy over $X/\!/G$. Since the space $X /\!/G$ is affine, it suffices to show that there is some open set $V \subset X/\!/G$ such that $\mf{A}_{\chi}(U,U') |_V \neq 0$. We take $V$ to be the image of a $G$-saturated open subset of $X$ on which $G$ acts freely. Then it follows from Lemma \ref{lem:nonzerorestopenU} that $\mf{A}_{\chi}(U,U') |_V \neq 0$.   
	
	Part (b). Let $U, U' \in S$. We have shown in Lemma \ref{lem:leftrightnoetherian} that the bimodule $\mf{A}_{\chi}(U,U')$ is finitely generated both as a left $\mf{A}_{\chi}(U)$-module and as a right $\mf{A}_{\chi}(U')$-module. By part (a), it is non-zero. Therefore, the claim that $\mf{A}_{\chi}(U)$ is Morita equivalent to $\mf{A}_{\chi}(U')$ follows from \cite[Lemma 1.3]{Hompropquantum}.  
\end{proof}

\subsection{Symplectic leaves}\label{sec:sympleavesopenXcirc}

As in the previous setting, let $X$ be a smooth affine $G$-variety and $U$ a finite-dimensional $G$-module. Let $Z(U)$ denote the centre of the non-commutative algebra $(\C[\mu^{-1}(0)] \otimes \End (U))^G$. Since $(\C[\mu^{-1}(0)] \otimes \End (U))^G$ is a finite $\C[\mu^{-1}(0)]^G$-module, and the latter algebra is of finite type, $Z(U)$ is a finite module over $\C[\mu^{-1}(0)]^G$. In particular, it is also of finite type. Abusing notion, we write $\mu_U^{-1}(0) / \!/ G$ for $\Spec Z(U)$. The finite morphism $\mu_U^{-1}(0) / \!/ G \rightarrow \mu^{-1}(0) / \!/ G$ coming from the embedding $\C[\mu^{-1}(0)]^G \hookrightarrow Z(U)$ is denoted $\Xi$. Recall from the proof of Lemma \ref{lem:leftrightnoetherian} that the order filtration on $\dd(X) \otimes \End_{\C}(U)$ induces a filtration $\mc{F}_{\idot}$ on $\mf{A}_{\chi}(U)$.

\begin{lemma}\label{lem:asscrgrflat}
	If the moment map $\mu$ is flat, then 
	$$
	\C[\mu^{-1}(0)]^G = \gr_{\mc{F}} \mf{A}_{\chi}(\C), \quad Z(U) = Z(\gr_{\mc{F}} \mf{A}_{\chi}(U)).
	$$
\end{lemma}

\begin{proof}
	The first statement is a special case of the second statement, so we concentrate on the latter. Since the moment map $\mu$ is flat, \cite[Proposition 2.4]{Holland} implies that the map $(\C[\mu^{-1}(0)] \otimes \End (U))^G \rightarrow \gr_{\mc{F}} \mf{A}_{\chi}(U)$ is an isomorphism.  
\end{proof}

Lemma \ref{lem:asscrgrflat} implies that both $\C[\mu^{-1}(0)]^G$ and $Z(U)$ inherit a natural Poisson structure (which agrees, in the case of $\C[\mu^{-1}(0)]^G$, with the one coming from Hamiltonian reduction). 

\begin{lemma}\label{lem:XiPoisson}
The morphism $\Xi$ is Poisson. 
\end{lemma}

\begin{proof}
If we write $\mbf{1}$ for the line spanned by the identity in $E := \End (U)$, then $E = E_0 \oplus \mbf{1}$ as a $G$-module, where $E_0$ is the subspace of traceless endomorphisms. We write $(\nu_{\chi} + \Phi_U)(\g)$ for the space $\{ \nu_{\chi}(x) \otimes 1 + 1 \otimes \Phi_U(x) \ | \ x \in \g \}$. Recall that $\mc{F}_{\idot} \dd(X) \otimes E$ is the filtration obtained by taking the order filtration on $\dd(X)$ and putting $E$ in degree zero. We claim that, for each $m \ge {0}$,
\begin{multline}\label{eq:filtreq1}
\left[ (\mc{F}_{m-1} \dd(X) \otimes E)^G + (\mc{F}_m(\dd(X) \otimes E (\nu_{\chi} + \Phi_U)(\g)))^G \right] \cap ((\mc{F}_m \dd(X)) \otimes \mbf{1})^G \\ = \left[ (\mc{F}_{m-1} \dd(X))^G + (\mc{F}_m \dd(X) \nu_{\chi}(\g))^G \right] \otimes \mbf{1} 
\end{multline} 
First notice that 
\begin{multline*}
\mc{F}_{m-1} \dd(X) \otimes E + \mc{F}_m(\dd(X) \otimes E (\nu_{\chi} + \Phi_U)(\g)) \\
=  \mc{F}_{m-1} \dd(X) \otimes E + \mc{F}_m(\dd(X) \otimes E (\nu_{\chi}(\g)))
\end{multline*} 
and hence 
\begin{multline*}
(\mc{F}_{m-1} \dd(X) \otimes E)^G + (\mc{F}_m(\dd(X) \otimes E (\nu_{\chi} + \Phi_U)(\g)))^G \\
=  (\mc{F}_{m-1} \dd(X) \otimes E)^G + (\mc{F}_m(\dd(X) \otimes E (\nu_{\chi}(\g))))^G.
\end{multline*} 
Next, 
$$
(\mc{F}_{m-1} \dd(X) \otimes E)^G  = (\mc{F}_{m-1} \dd(X) \otimes E_0)^G  \oplus (\mc{F}_{m-1} \dd(X))^G \otimes \mbf{1}
$$
and
$$
(\mc{F}_m(\dd(X) \otimes E (\nu_{\chi}(\g))))^G =  (\mc{F}_m(\dd(X) \otimes E_0 (\nu_{\chi}(\g))))^G \oplus  (\mc{F}_m(\dd(X)  (\nu_{\chi}(\g))))^G \otimes \mbf{1}.
$$
Combining these equations gives (\ref{eq:filtreq1}).\\

Equality (\ref{eq:filtreq1}) implies that, for each $m \ge 0$, there is a commutative diagram 
$$
\begin{tikzcd}
0 \ar[r] & \mc{J}_m \ar[r] & (\mc{F}_{m} \dd(X) \otimes E)^G \ar[r] & (\C[\mu^{-1}(0)] \otimes E)^G_m \ar[r] & 0 \\
0 \ar[r] & \mc{K}_m \ar[r] \ar[u,hook] & (\mc{F}_{m} \dd(X))^G \otimes \mbf{1} \ar[r] \ar[u,hook]  & \C[\mu^{-1}(0)]_m^G \otimes \mbf{1} \ar[r] \ar[u,hook]  & 0 
\end{tikzcd}
$$
with exact rows, where
$$
\mc{J}_m := (\mc{F}_{m-1} \dd(X) \otimes E)^G + (\mc{F}_m(\dd(X) \otimes E (\nu_{\chi} + \Phi_U)(\g)))^G
$$
and
$$
\mc{K}_m := ((\mc{F}_{m-1} \dd(X))^G + (\mc{F}_m \dd(X) \nu_{\chi} (\g))^G) \otimes \mbf{1}.
$$
Therefore, we have shown that there is a linear splitting 
$$
\eta : (\C[\mu^{-1}(0)] \otimes E)^G_m \rightarrow (\mc{F}_{m} \dd(X) \otimes E)^G 
$$
of the surjection $(\mc{F}_{m} \dd(X) \otimes E)^G \rightarrow (\C[\mu^{-1}(0)] \otimes E)^G_m$ whose restriction to $\C[\mu^{-1}(0)]_m^G \otimes \mbf{1}$ lands in $\mc{F}_{m} \dd(X))^G \otimes \mbf{1} $. This means that if $f \in \C[\mu^{-1}(0)]_m^G$ and $g \in \C[\mu^{-1}(0)]_k^G$, then their Poisson bracket can be computed as the image of $[\eta(f \otimes 1), \eta(g \otimes 1)]$ in $(\C[\mu^{-1}(0)] \otimes E)^G_{m+k-1}$. This implies that $\Xi$ is a Poisson morphism. 
\end{proof}

Lemma \ref{lem:XiPoisson} implies that if $\mu^{-1}(0) /\! / G$ has finitely many symplectic leaves, then so too does $\mu^{-1}_U(0) /\! / G$. 

\subsection{Prime endomorphism rings}\label{sec:equvendring}

In order to prove the results described in the introduction, we need to relate the simplicity of the algebras $\mf{A}_{\chi}(U)$ to the properties of the category of admissible $\dd$-modules on $X$. This is done by applying the results from \cite{Primitive} which show that the primitive ideals in $\mf{A}_{\chi}(U)$ can be realised as the annihilators of modules coming from the category of admissible $\dd$-modules. We would like to conclude that if $\mf{A}_{\chi}(U)$ has only one primitive ideal then it is a simple algebra. For this, we need $\mf{A}_{\chi}(U)$ to be prime. If we consider the case that is most important to us, $X = \Rep (\Q_{\infty}(\ell),\mathbf{v})$ and $G = G(n \delta)$, then it is well-known that $\mf{A}_{\chi}(\C)$ is prime. However, if $\dim U > 1$, then explicit examples show that $\mf{A}_{\chi}(U)$ is \textit{not} prime in general. Whether $\mf{A}_{\chi}(U)$ is prime or not depends heavily on the parameter $\chi$. We expect that locus of points $\chi$ in $\mathbb{X}^*(\mf{g})$ where $\mf{A}_{\chi}(U)$ is prime is a \textit{Zariski} open subset. 


In order to force $\mf{A}_{\chi}(U)$ to be prime, we must perform a small localization. We assume that there is a $G$-semi-invariant $s : X \rightarrow \C$, such that if $X^{\circ}= X \smallsetminus s^{-1}(0)$ then there is a non-empty $G$-saturated open subset of $X^{\circ}$ on which $G$ acts freely. 

\begin{remark}
	Clearly if there is a $G$-stable open subset of $X^{\circ}$ where $G$ acts freely, then there is a $G$-stable open subset of $X$ where $G$ acts freely. However, in the example of interest to us, there is no $G$-saturated open subset on which $G$ acts freely. Hence we must make the restriction to $X^{\circ}$.  
\end{remark} 

Let $\mf{A}_{\chi}^{\circ}(U)$ be the endomorphism algebra of $\ms{P}_{\chi}(U) |_{X^{\circ}}$. Since $X^{\circ}$ is an affine open subset of $X$, the restriction map $\Gamma(X,\ms{P}_{\chi}(U)) \rightarrow \Gamma( X^{\circ},\ms{P}_{\chi}(U))$ induces a map 
$$
\End_{\dd}(\ms{P}_{\chi}(U)) \rightarrow \End_{\dd}(\ms{P}_{\chi}(U) |_{X^{\circ}})
$$
which, by Lemma \ref{lem:endU}, we can identify with an algebra homomorphism $\varphi_U : \mf{A}_{\chi}(U) \rightarrow \mf{A}_{\chi}^{\circ}(U)$. In general this map is not injective. 

Let $\mu^{-1}(0)^{\circ} = \mu^{-1}(0) \cap T^* X^{\circ}$ and set 
$$
A = \C[\mu^{-1}(0)^{\circ}], \quad E = A \otimes \End_{\C} (U).
$$
Since there is a $G$-saturated open set on which $G$ acts freely, there is a function $t \in \C[X^{\circ}]^G$ such that $G$ acts freely on $(t \neq 0) \subset X^{\circ}$. Pulling $t$ back to $T^* X^{\circ}$, we deduce that $G$ acts freely on $\mu^{-1}(0) \cap (t \neq 0)$. 


\begin{proposition}\label{prop:EGprime}
If $A^G$ is a domain then the algebra $E^G$ is prime. 
\end{proposition}



\begin{proof}
	First, descent for $G$-equivariant coherent sheaves implies that the canonical morphism 
	$$
	A[t^{-1}] \otimes_{A^G[t^{-1}]} E^G[t^{-1}] \rightarrow E[t^{-1}]
	$$
	is an isomorphism. Therefore $E^G[t^{-1}]$ is an Azumaya algebra because its fibre, considered as a sheaf of algebras on $\Spec \, A[t^{-1}]$, over any closed point can be identified with $\End_{\C}(U)$. In particular, the centre of $E^G[t^{-1}]$ equals $A^G[t^{-1}]$.
	
	Next, we note that $E$ is $t$-torsion free, being free over $A$. Therefore $E \hookrightarrow E[t^{-1}]$. This implies that we have inclusions
	$$
	A^G \hookrightarrow Z(E^G) \hookrightarrow Z(E^G[t^{-1}]) = A^G[t^{-1}].
	$$
	In particular, since $A$ is a domain, both $A^G[t^{-1}]$ and $Z(E^G)$ are also domains. 
	
	Let $I$ be a non-zero ideal in $E^G$. We claim that $I \cap Z(E^G) \neq \{ 0 \}$. Consider the localized ideal $I[t^{-1}]$ in $E^G[t^{-1}]$. If $I[t^{-1}] = E^G[t^{-1}]$ then some power of the central element $t$ is in $I$. So assume that $I[t^{-1}] $ is a proper ideal. It is non-zero since $E$ embeds in $E[t^{-1}]$. Then, the fact that $E^G[t^{-1}]$ is an Azumaya algebra implies that every ideal is centrally generated \cite[Proposition 13.7.9]{MR}. In particular, $I[t^{-1}]\cap Z(E^G[t^{-1}]) \neq \{ 0 \}$, which implies that $I \cap Z(E^G) \neq \{ 0 \}$. Since $Z(E^G)$ is a domain, we deduce that $E^G$ is prime.  
\end{proof}

Finally, we lift Proposition \ref{prop:EGprime} to a statement about $\mf{A}_{\chi}^{\circ}(U)$. 

\begin{proposition}\label{prop:UcircAisprime}
	Assume that:
	\begin{enumerate}
		\item[(i)] the moment map $\mu : T^* X^{\circ} \rightarrow \mf{g}^*$ is flat,
		\item[(ii)] there is a $G$-saturated open subset of $X^{\circ}$ on which $G$ acts freely, and
		\item[(iii)] the scheme $\mu^{-1}(0)^{\circ} /\!/ G$ is reduced and irreducible.  
	\end{enumerate}
	Then $\mf{A}_{\chi}^{\circ}(U)$ is prime for all $\chi$. 
\end{proposition}

\begin{proof}
	Let $E = \C[\mu^{-1}(0)] \otimes \End (U)$. By Lemma \ref{lem:asscrgrflat}, the fact that $\mu$ is flat implies that $\gr_{\mc{F}} \mf{A}_{\chi}(U) = E^G$. Since there is a $G$-saturated open set on which $G$ acts freely, there is a function $t \in \C[X^{\circ}]^G$ such that $G$ acts freely on $(t \neq 0) \subset X^{\circ}$. Pulling $t$ back to $T^* X^{\circ}$, we deduce that $G$ acts freely on $\mu^{-1}(0) \cap (t \neq 0)$. Then, using (iii), we can deduce from Proposition \ref{prop:EGprime} that $E^G$ is prime. Therefore $\mf{A}_{\chi}^{\circ}(U)$ is prime \cite[Proposition 1.6.6]{MR}.  
\end{proof}

\section{Admissible $\dd$-modules}\label{sec:admissible}

As explained in the introduction, one major motivation for studying the enhanced cyclic nilpotent cone is the important role it plays in the theory of admissible $\dd$-modules, and via the functor of quantum Hamiltonian reduction, to category $\mc{O}$ for the cyclotomic rational Cherednik algebra. In this section we recall the definition of admissible and orbital $\dd$-modules, as introduced by Gan-Ginzburg \cite{AlmostCommutingVariety}. 

\subsection{Definitions}\label{sec:admissibledefn} We specialize to the following situation: $X = Y \times R$ is a trivial $G$-equivariant vector bundle over $Y$, where $Y$ is an \textit{affine} $G$-variety, $R$ is a finite dimensional $G$-module and $G$ acts diagonally on $Y \times R$. The group $\C^{\times}$ also acts on $X$ by dilations along the fibre $R$. Let $\mathsf{eu} \in \dd_{X}$ be the corresponding vector field. We say that a coherent $\dd_{X}$-module $\ms{M}$ is \textit{smooth} if $\mathsf{eu}$ acts locally finitely on $\Gamma(X,\ms{M})$. Given any $G$-representation $V$, let $\mc{N}(V) := \pi^{-1}(0)$ be the nilcone of $V$, where $\pi : V \rightarrow V/\!/G$ is the categorical quotient.

\begin{definition}\label{defn:admissible}
	The category $\Add_q$ of $q$-admissible $\dd$-modules on $X$ is defined to be the full subcategory of $\Coh(\dd_{X})$ consisting of all $\ms{M}$ such that 
	\begin{enumerate}
		\item $\ms{M}$ is $(G,q)$-monodromic. 
		\item $\ms{M}$ is smooth.
		\item $\SS(\ms{M}) \subset T^* Y \times R \times \mc{N}(R^{\vee}) \subset T^* X$.
	\end{enumerate}
\end{definition}

Here $R^{\vee}$ denotes the dual representation. Similarly, we denote by $\Add_{\chi}$ the abelian category of all $(G,\chi)$-monodromic $\dd_{X}$-modules satisfying properties (2) and (3) of Definition \ref{defn:admissible}. Notice that $\Add_{\chi}$ is not a full subcategory of $\Coh(\dd_{X})$. There is a forgetful functor $F : \Add_{\chi} \rightarrow \Coh(\dd_{X})$, forgetting the isomorphism $\phi : \mc{O}^{\chi}_G \boxtimes \ms{M} \stackrel{\sim}{\longrightarrow} a^* \ms{M}$. The image of $F$ is $\Add_{q}$. 

\begin{remark}
	In \cite{MirabolicCharacter} and \cite{mirabolicHam}, the admissible $\dd$-modules were called mirabolic because of the relation to the mirabolic (= ``miraculous parabolic'') subgroup $P = \Stab_{GL(V)}(v)$, where $v$ is any non-zero vector in $V$. Since admissible $\dd$-modules make sense for any $G$-representation, not just $\mf{gl}(V) \times V$, we stick to the latter terminology in this article. 
\end{remark}

In order to be able to use our results on the fundamental groups of the orbits $\mc{O} \subset V \times \mc{N}$, we use the Fourier transform to relate admissible $\dd$-modules to orbital $\dd$-modules. As in the proof of Proposition 5.3.2 of \cite{AlmostCommutingVariety}, the Fourier transform along the trivial vector bundle $X = Y \times R \rightarrow Y$ defines an equivalence $\mathbb{F} : \Coh(\dd_X) \stackrel{\sim}{\longrightarrow} \Coh(\dd_{X^{\vee}})$, where $X^{\vee} = Y \times R^{\vee}$. On the level of geometry, the Fourier transform is the identification 
$$
\mathbb{F} : T^* X = T^* Y \times R \times R^{\vee} \stackrel{\sim}{\longrightarrow} T^* Y \times R^{\vee} \times R = T^* X^{\vee},
$$
with $\mathbb{F}(y,r,f) = (y,f,-r)$ for $y \in T^*Y$, $r \in R$ and $f \in R^{\vee}$. Let 
$$
\Lambda = \mu^{-1}(0) \cap (T^* Y \times R \times \mc{N}(R^{\vee})).
$$
Repeating the proof of \cite[Lemma 4.4]{AlmostCommutingVariety}, we have:

\begin{lemma}
	If $Y \times \N(R^{\vee})$ has finitely many $G$-orbits, then
	$$
	\Lambda = \bigcup_{\mc{O} \subset Y \times \N(R^{\vee})} \mathbb{F}\left( \overline{T^*_{\mc{O}} X^{\vee}} \right),
	$$
	where the union is over all $G$-orbits. In particular, $\Lambda$ is Lagrangian. 
\end{lemma}

\begin{definition}
	The category $\Orb_q$ is the category of all $(G,q)$-monodromic  $\dd_{X^{\vee}}$-modules supported on $Y \times \mc{N}(R^{\vee})$. Similarly, $\Orb_{\chi}$ is the category of all $(G,\chi)$-monodromic $\dd_{X^{\vee}}$-modules supported on $Y \times \mc{N}(R^{\vee})$. 
\end{definition}

Again, we have a forgetful functor $F : \Orb_{\chi} \rightarrow \Orb_q$. 

\begin{proposition}\label{prop:orbprop}
	Fix $\chi \in \mathbb{X}^*(\g)$ and $q$ its image in $\mathbb{T}(G)$. 
	\begin{enumerate} 
		
		\item[(a)] Fourier transform defines an equivalence $\mathbb{F}  : \Orb_{\idot} \stackrel{\sim}{\longrightarrow} \Add_{\idot}$, where $\idot \in \{ \chi, q\}$, such that the diagram 
		$$
		\begin{tikzcd}
		\Orb_{\chi} \ar[r,"{\mathbb{F}}"] \ar[d,"F"'] & \Add_{\chi} \ar[d,"F"] \\
		\Orb_{q} \ar[r,"{\mathbb{F}}"]  & \Add_{q}
		\end{tikzcd}
		$$
		commutes.
	\end{enumerate}
	Moreover, if $Y \times \N(R^{\vee})$ has finitely many $G$-orbits, then 
	\begin{enumerate}
		\item[(b)] Every module in each of $\Add_q,\Add_{\chi}, \Orb_q$ and $\Orb_{\chi}$ is regular holonomic. 
		\item[(c)] The objects in each of $\Add_q,\Add_{\chi}, \Orb_q$ and $\Orb_{\chi}$ have finite length and there are only finitely many simple objects, up to isomorphism. 
	\end{enumerate}
\end{proposition}

\begin{proof}
	The proof of the proposition is identical to the proof of Proposition 5.3.2 in \cite{AlmostCommutingVariety}. The only thing that is not immediate is that the Fourier transform restricts to an equivalence $\mathbb{F} : \QCoh (\dd_X,G,\chi) \stackrel{\sim}{\longrightarrow} \QCoh (\dd_{X^{\vee}}, G, \chi)$ of $(G,\chi)$-monodromic modules. This is an elementary direct calculation. 
\end{proof}

The existence of the Fourier transform on $X$, and hence the equivalence between admissible $\dd$-modules and orbital modules means that this situation is much simpler than the corresponding group-like setup considered in \cite{mirabolicHam}. However, there is no natural group-like analogue in the cyclic quiver situation. We also note that, combinatorially, this situation is much richer than the case $X = \mathfrak{gl}_n \times V$. 

The reader might also ask why there is a need to consider the category of admissible $\dd$-modules, when one can just work directly with the category of orbital $\dd$-modules. The reason for not abandoning the admissible point of view is because there are some properties of the category that are much harder (if not impossible) to see from the orbital point of view, that are clear when considering admissible $\dd$-modules. For instance, the Knizhnik-Zamolodchikov functor is easy to define and study for admissible modules, but involves a rather tricky microlocalization construction for orbital modules. Similarly, it is easier to relate category $\Osph$ to the admissible category. 

\subsection{Quantum Hamiltonian reduction}

Recall that we have assumed $X = Y \times R$, with $R$ a finite-dimensional $G$-module. The canonical map $\C[R]^G \otimes \C[R^\vee]^G \rightarrow \C[\mu^{-1}(0)]^G$ defines a morphism 
$$
\Upsilon : \mu^{-1}(0)/ \!/ G \rightarrow R /\!/ G \times R^{\vee} /\!/ G.
$$
Throughout the remainder of section \ref{sec:admissible}, \textbf{we assume}:

\begin{enumerate}
	\item[(F1)] The morphism $\Upsilon$ is finite.
	\item[(F2)] $Y \times \mc{N}(R^{\vee})$ has finitely many $G$-orbits.
\end{enumerate}

\begin{remark}
	The situation of interest to us is where $Y = V \simeq \C^n$, $R = \Rep(\Q(\ell),n\delta)$, and hence $X = \Rep(\Q_{\infty}(\ell),\mbf{v})$. Then conditions (F1) and (F2) hold. See \cite{AlmostCommutingVariety} and references therein. 
\end{remark}

Identifying $\sym R$ with constant coefficient differential operators on $R$, the algebras $\C[R]^G$ and $(\sym R)^G$ are commutative subalgebras of $\mf{A}_{\chi}(U)$. Given a connected graded algebra $S$, let $S_+$ be the augmentation ideal. We say that the action of $S$ on an $S$-module $M$ is locally nilpotent if for each $m \in M$ there is some $N \gg 0$ such that $(s_1 \cdots s_N) \cdot m = 0$ for all $s_1, \ds, s_N \in S_+$. 

\begin{definition}
	Category $\mc{O}_{\chi}(U)$ is defined to be the category of finitely generated $\mf{A}_{\chi}(U)$-modules that are locally nilpotent for $(\sym R)^G$. 
\end{definition}

The element $\eu$ belongs to $\mf{A}_{\chi}(U)$ such that $\mathrm{ad}(\eu)$ is semi-simple with integer eigenvalues. In particular, this makes $\mf{A}_{\chi}(U)$ a $\Z$-graded algebra. 

\begin{proposition}\label{prop:catOchi}
For all $\chi \in \mathbb{X}^*(\mf{g})$,
	\begin{enumerate}
		\item[(a)] Every object in $\mc{O}_{\chi}(U)$ has finite length. 
		\item[(b)] Up to isomorphism, there are finitely many simple objects in $\mc{O}_{\chi}(U)$. 
		\item[(c)] $\mc{O}_{\chi}(U)$ has enough projectives (and hence enough injectives).
		\item[(d)] The action of $\eu$ on any object $M$ of $\mc{O}_{\chi}(U)$ is locally finite, and all generalized eigenspaces of this action are finite-dimensional. 
	\end{enumerate}
\end{proposition}

\begin{proof}
	If $A = \mf{A}_{\chi}(U)$, $A^+ = \C[R]^G$ and $A^- = (\sym R)^G$, then $(A^{\pm} ,\eu)$ is a commutative triangular structure on $A$, in the sense of \cite{Primitive}. Therefore all claims follow from \cite[Theorem 2.5]{Primitive} once we establish (b) that $\mc{O}_{\chi}(U)$ has finitely many simple objects, up to isomorphism. 
	
	Since $X$ is affine, the category $\Add_{\chi}$ can also be characterized as the category of smooth $(G,\chi)$-monodromic $\dd_X$-modules $\ms{M}$ such that $(\sym R)^G$ acts locally nilpotent on $\Gamma(X,\ms{M})$; see section 5 of \cite{AlmostCommutingVariety}. This implies that $\Ham_{U,\chi}$ restricts to a quotient functor $\Ham_{U,\chi} : \Add_{\chi} \rightarrow \mc{O}_{\chi}(U)$. As noted in Proposition \ref{prop:orbprop}, assumption (F2) implies that $\Add_{\chi}$ has finitely many simple modules. Therefore, $\mc{O}_{\chi}(U)$ has finitely many simple objects too.
\end{proof}

Let $\mathrm{Prim} \ \mf{A}_{\chi}(U)$ denote the set of primitive ideals in $\mf{A}_{\chi}(U)$. In this case it follows immediately from \cite[Theorem 2.3]{Primitive} that:

\begin{corollary}
	For all $\chi \in \mathbb{X}^*(\mf{g})$, 
	\begin{enumerate}
		\item $\mf{p} \lhd \mf{A}_{\chi}(U)$ is prime if and only if it is primitive i.e. $\Spec \mf{A}_{\chi}(U) = \mathrm{Prim} \ \mf{A}_{\chi}(U)$. 
		\item $\mathrm{Prim} \ \mf{A}_{\chi}(U)$ is finite. 
	\end{enumerate}
\end{corollary}

Recall that we have defined a partial order on the set of all finite-dimensional $G$-modules by setting $U \le U'$ if and only if $U$ is isomorphic to some summand of $U'$.   

\begin{lemma}\label{lem:Ubig}
 There exists a finite-dimensional $G$-module $U_0$ such that for all $U' \ge U_0$, $\Add_{\chi} \simeq \mc{O}_{\chi}(U')$. 
\end{lemma}

\begin{proof}
	As noted in the proof of Proposition \ref{prop:catOchi}, the functor of Hamiltonian reduction $\Ham_{U,\chi}$ is a quotient functor. Therefore it is an equivalence if $\Ham_{U,\chi}(\ms{M}) \neq 0$ for every simple $\ms{M}$. Since there are only finitely many simple modules in $\Add_{\chi}$, one can choose $U$ sufficiently large to guarantee this. 
\end{proof} 

We fix, for the remainder of this section, a representation $U_0$ as in the statement of Lemma \ref{lem:Ubig}. 

\begin{lemma}\label{lem:Ochiequiv}
For $U_1, U_2 \ge U_0$, 
$$
\mf{A}_{\chi}(U_2,U_1) \otimes_{\mf{A}_{\chi}(U_1)} - : \mc{O}_{\chi}(U_1) \longrightarrow \mc{O}_{\chi}(U_2)
$$
is an equivalence. 
\end{lemma}

\begin{proof}
It suffices, by Lemma \ref{lem:Ubig}, to check that the diagram 
$$
\begin{tikzcd}
 & & \mc{O}_{\chi}(U_1) \ar[dd,"{\mf{A}_{\chi}(U_2,U_1) \otimes_{\mf{A}_{\chi}(U_1)} -}"] \\
\Add_{\chi} \ar[rru,"{\Ham_{U_1,\chi}}"] \ar[rrd,"{\Ham_{U_2,\chi}}"'] & & \\
 & & \mc{O}_{\chi}(U_2)
\end{tikzcd}
$$
is commutative. Lemma \ref{lem:endU}, together with the decomposition (\ref{eq:matrixendU}), implies that 
$$
\mf{A}_{\chi}(U_2,U_1) = \Hom_{\dd_{X}}(\ms{P}_{\chi}(U_2),\ms{P}_{\chi}(U_1)).
$$
Take $\ms{N} \in \Add_{\chi}$. By Lemma \ref{lem:Ubig}, the adjunction  $\ms{P}_{\chi}(U_1) \otimes_{\mf{A}_{\chi}(U_1)} \Ham_{U_1,\chi}(\ms{N}) \rightarrow \ms{N}$ is an isomorphism. Therefore, the transformation
$$
\Ham_{U_2,\chi}(\ms{P}_{\chi}(U_1) \otimes_{\mf{A}_{\chi}(U_1)} \Ham_{U_1,\chi}(\ms{N})) \rightarrow \Ham_{U_2,\chi}(\ms{N})
$$
is an isomorphism i.e.
$$
\Hom_{\dd_X}(\ms{P}_{\chi}(U_2), \ms{P}_{\chi}(U_1) \otimes_{\mf{A}_{\chi}(U_1)} \Ham_{U_1,\chi}(\ms{N})) \stackrel{\sim}{\longrightarrow} \Ham_{U_2,\chi}(\ms{N}). 
$$
Finally, since $\ms{P}_{\chi}(U_2)$ is projective in $\Coh(\dd_X,G,\chi)$, and $ \Ham_{U_1,\chi}(\ms{N})$ is finitely generated as a $\mf{A}_{\chi}(U_1)$-module, fixing a finite presentation of $ \Ham_{U_1,\chi}(\ms{N})$ and applying the exact functor $\Hom_{\dd_X}(\ms{P}_{\chi}(U_2), -)$ shows that the canonical map 
\begin{multline*}
\Hom_{\dd_X}(\ms{P}_{\chi}(U_2), \ms{P}_{\chi}(U_1)) \otimes_{\mf{A}_{\chi}(U_1)} \Ham_{U_1,\chi}(\ms{N}) \rightarrow \\
 \Hom_{\dd_X}(\ms{P}_{\chi}(U_2), \ms{P}_{\chi}(U_1) \otimes_{\mf{A}_{\chi}(U_1)} \Ham_{U_1,\chi}(\ms{N})) 
\end{multline*}
is also an equivalence.  
\end{proof}

Combining Lemma \ref{lem:Ubig}, with Proposition \ref{prop:catOchi}, proves Proposition \ref{prop:enoughproj}.

\begin{theorem}\label{thm:Hamequivall}
	For all $U \ge U_0$, the functors
	$$
	\Ham_{U,\chi} : \Coh(\dd_X,G,\chi) \rightarrow \Lmod{\mf{A}_{\chi}(U)},
	$$
	and
	$$
	\mf{A}_{\chi}(U,U_0) \otimes_{\mf{A}_{\chi}(U_0)} - : \Lmod{\mf{A}_{\chi}(U_0)} \rightarrow \Lmod{\mf{A}_{\chi}(U)}
	$$
	are equivalences, making the diagram 
	$$
	\begin{tikzcd}
	 & \Coh(\dd_X,G,\chi) \ar[dl,"{\Ham_{U_0,\chi}}"'] \ar[dr,"{\Ham_{U,\chi}}"] & \\
	 \Lmod{\mf{A}_{\chi}(U_0)} \ar[rr,"{\mf{A}_{\chi}(U,U_0) \otimes_{\mf{A}_{\chi}(U_0)} - }"'] & & \Lmod{\mf{A}_{\chi}(U)} 
	 \end{tikzcd}
	 $$
	 commute. 
\end{theorem}

\begin{proof}
	First we check that for all $U_1 \ge U_2 \ge U_0$, the $\mf{A}_{\chi}(U_1)$-$\mf{A}_{\chi}(U_2)$-bimodule $\mf{A}_{\chi}(U_1,U_2)$ induces a Morita equivalence $\Lmod{\mf{A}_{\chi}(U_1)} \simeq \Lmod{\mf{A}_{\chi}(U_2)}$. Let $U_2'$ be a submodule of $U_1$ such that $U_1 = U_2 \oplus U_2'$. Under the identification (\ref{eq:matrixendU}) of $\mf{A}_{\chi}(U_1)$ with a matrix algebra, let $e \in \mf{A}_{\chi}(U_1)$ be the idempotent such that $e \mf{A}_{\chi}(U_1) e = \mf{A}_{\chi}(U_2)$. Then 
	$$
	\mf{A}_{\chi}(U_1) e = \mf{A}_{\chi}(U_2) \oplus \mf{A}_{\chi}(U_2',U_2) = \mf{A}_{\chi}(U_1,U_2)
	$$
	and $\End_{\mf{A}_{\chi}(U_1)}(\mf{A}_{\chi}(U_1)e)^{\mathrm{op}} = e \mf{A}_{\chi}(U_1) e = \mf{A}_{\chi}(U_2)$. Moreover, notice that $e \mf{A}_{\chi}(U_1)$ equals $\Hom_{\mf{A}_{\chi}(U_1)}(\mf{A}_{\chi}(U_1) e,\mf{A}_{\chi}(U_1) )$. Thus, it follows that $\mf{A}_{\chi}(U_1,U_2)$ defines a Morita equivalence if and only if the multiplication map $\mf{A}_{\chi}(U_1,U_2) \otimes_{\mf{A}_{\chi}(U_2)} \mf{A}_{\chi}(U_2,U_1) \rightarrow \mf{A}_{\chi}(U_1)$ is surjective i.e. we just need to check that $\mf{A}_{\chi}(U_1,U_2)$ is a generator in $\Lmod{\mf{A}_{\chi}(U_1)}$. Assume otherwise. Then there exists a (proper) primitive ideal $J \lhd \mf{A}_{\chi}(U_1)$ such that the image of $\mf{A}_{\chi}(U_1,U_2) \otimes_{\mf{A}_{\chi}(U_2)} \mf{A}_{\chi}(U_2,U_1) $ is contained in $J$. By the generalized Duflo Theorem \cite[Theorem 2.3]{Primitive}, there exists a simple object $L \in \mc{O}_{\chi}(U_1)$ such that $\mathrm{ann}_{\mf{A}_{\chi}(U_1)} L = J$. But then 
	$$
	\mf{A}_{\chi}(U_1,U_2) \otimes_{\mf{A}_{\chi}(U_2)} \mf{A}_{\chi}(U_2,U_1) \otimes_{\mf{A}_{\chi}(U_1)} L = 0. 
	$$
	By Lemma \ref{lem:Ochiequiv}, this implies that $L = 0$, contradicting the fact that $J \subsetneq \mf{A}_{\chi}(U_1)$. We deduce that $\mf{A}_{\chi}(U_1,U_2)$ is a generator. 
	
	Next, $\Ham_{U_0,\chi}$ is a quotient functor. Therefore it suffices to check that $\Ham_{U_0,\chi}(\ms{N}) \neq 0$ for all non-zero objects $\ms{N}$ in $\Coh(\dd_{X},G,\chi)$. Since there are only countably many isomorphism classes of simple $G$-modules, we can choose $U_0 \le U_1 \le U_2 \le \cdots $ such that for any simple $G$-module $W$ there exists some $N \gg 0$ with $W \le U_N$. Now assume that $\Ham_{U_0,\chi}(\ms{N}) = 0$. If $\ms{N} \neq 0$ then there exists $W$ such that $(\Gamma(X,\ms{N}) \otimes W^*)^G \neq 0$. Hence, for any $U_N \ge W$, $\Ham_{U_N,\chi}(\ms{N}) \neq 0$. Taking a quotient if necessary, we may assume $\ms{N}$ is simple. We argue as in the proof of Lemma \ref{lem:Ochiequiv}.  The adjunction  $\ms{P}_{\chi}(U_N) \otimes_{\mf{A}_{\chi}(U_N)} \Ham_{U_N,\chi}(\ms{N}) \rightarrow \ms{N}$ has non-zero image, hence is surjective (using the fact that the adjunction $\mathrm{Id}_{\Lmod{\mf{A}_{\chi}(U_N)}} \rightarrow \Ham_{U_N,\chi} \circ {}^{\perp} \Ham_{U_N,\chi}$ is an isomorphism). Therefore, there is a short exact sequence 
	\begin{equation}\label{eq:seskill}
	0 \rightarrow \ms{K} \rightarrow \ms{P}_{\chi}(U_N) \otimes_{\mf{A}_{\chi}(U_N)} \Ham_{U_N,\chi}(\ms{N}) \rightarrow \ms{N} \rightarrow 0,
	\end{equation}
	with $\Ham_{U_N,\chi}(\ms{K}) = 0$. This implies that $\Ham_{U_0,\chi}(\ms{K}) = 0$ too. Hence the transformation
	$$
	\Ham_{U_0,\chi}(\ms{P}_{\chi}(U_N) \otimes_{\mf{A}_{\chi}(U_N)} \Ham_{U_N,\chi}(\ms{N})) \rightarrow \Ham_{U_0,\chi}(\ms{N})
	$$
	is an isomorphism i.e.
	$$
	\Hom_{\dd_X}(\ms{P}_{\chi}(U_0), \ms{P}_{\chi}(U_N) \otimes_{\mf{A}_{\chi}(U_N)} \Ham_{U_N,\chi}(\ms{N})) \stackrel{\sim}{\longrightarrow} \Ham_{U_0,\chi}(\ms{N}). 
	$$
	Finally, since $\ms{P}_{\chi}(U_0)$ is projective in $\Coh(\dd_X,G,\chi)$, and $ \Ham_{U_N,\chi}(\ms{N})$ is finitely generated as a $\mf{A}_{\chi}(U_N)$-module, the canonical map 
	\begin{multline*}
	\Hom_{\dd_X}(\ms{P}_{\chi}(U_0), \ms{P}_{\chi}(U_N)) \otimes_{\mf{A}_{\chi}(U_N)} \Ham_{U_N,\chi}(\ms{N}) \rightarrow \\
	\Hom_{\dd_X}(\ms{P}_{\chi}(U_0), \ms{P}_{\chi}(U_N) \otimes_{\mf{A}_{\chi}(U_N)} \Ham_{U_N,\chi}(\ms{N})) = \Ham_{U_0,\chi}(\ms{N})
	\end{multline*}
	is also an isomorphism. But $\Hom_{\dd_X}(\ms{P}_{\chi}(U_0), \ms{P}_{\chi}(U_N)) = \mf{A}_{\chi}(U_0,U_N)$ defines a Morita equivalence between $\Lmod{\mf{A}_{\chi}(U_N)}$ and $\Lmod{\mf{A}_{\chi}(U_0)}$. Therefore, $\Ham_{U_N,\chi}(\ms{N}) \neq 0$ implies that $\Ham_{U_0,\chi}(\ms{N})  \neq 0$. This contradicts our initial assumption. 
\end{proof}

\begin{corollary}
	The module $\ms{P}_{\chi}(U)$ is a projective generator in $\Coh(\dd_X,G,\chi)$ for all $U \ge U_0$. 
\end{corollary}

\begin{example}
	A classical situation where one can apply the results of this section is the case where $G$ is a simple affine algebraic group and $X = R = \g$. We remark that the category $\Add$ is semi-simple and, via the Riemann-Hilbert correspondence, this category is equivalent to Lusztig's character sheaves on $\g$. Thus, $\mc{O}(U)$ is also semi-simple for all $U$, and for $U$ sufficiently large the simple modules in $\mc{O}(U)$ are in bijection with the simple character sheaves on $\g$. The results of Gunningham \cite{GunnighamAbelian} imply that for $U$ sufficiently large, the blocks of the algebra $\mf{A}(U)$ are labelled by Lusztig's cuspidal character sheaves. 
\end{example}

\subsection{Holonomic modules}\label{sec:holossQHR}

As in section \ref{sec:equvendring}, we assume that there exists a semi-invariant $s$ on $X$, such that the conditions of Proposition \ref{prop:UcircAisprime} hold. Namely, if $X^{\circ} = (s \neq 0)$, then: 
\begin{enumerate}
	\item[(i)] the moment map $\mu \colon T^* X^{\circ} \rightarrow \mf{g}^*$ is flat,
\item[(ii)] there is a $G$-saturated open subset of $X^{\circ}$ on which $G$ acts freely, and
\item[(iii)] the scheme $\mu^{-1}(0)^{\circ} /\!/ G$ is reduced and irreducible. 
\end{enumerate}
Let $j : X^{\circ} \hookrightarrow X$ denote the affine open embedding. Throughout the remainder of section \ref{sec:admissible}, \textbf{we assume} in addition that:
\begin{enumerate}
	\item[(F3)] The moment map $\mu \colon T^* X \rightarrow \mf{g}^*$ is flat.
\end{enumerate}

Let $M$ be a finitely generated $\mf{A}_{\chi}(U)$-module. By Lemma \ref{lem:asscrgrflat}, we may choose a good filtration $\mc{F}_{\idot} M$ on $M$, so that $\gr_{\mc{F}} M$ is a finitely generated $(\C[\mu^{-1}(0)] \otimes \End(U))^G$-module. By restriction, it is a finitely generated $Z(U)$-module and we write $V(M)$ for its support in $\mu^{-1}_U(0) /\!/ G$. We are in the general setting of \cite[\S 2.1]{BernsteinLosev}, therefore the results of \cite{BernsteinLosev} apply. Following \cite{BernsteinLosev}, we say that $M$ is \textit{holonomic} if $V(M)$ is isotropic. Here $V \subset \mu^{-1}_U(0)/\!/G$ is said to be isotropic if $V \cap \mc{L}$ is isotropic in $\mc{L}$ for all leaves $\mc{L} \subset \mu^{-1}_U(0)/\!/G$. 

\begin{lemma}
	Every module in $\mc{O}_{\chi}(U)$ is holonomic. 
\end{lemma}

\begin{proof}
	Since $\mf{A}_{\chi}(U)$ is a finite $\C[R]^G \otimes (\sym R)^G$-bimodule, every object of $\mc{O}_{\chi}(U)$ is finitely generated over $\C[R]^G$. Therefore a good filtration of $M \in \mc{O}_{\chi}(U)$ is given by putting $\mc{F}_{-1} M = \{ 0 \}$ and $\mc{F}_m M = M$ for all $m \ge 0$. This implies that $V(M)$ is contained in the zero set $\mc{N}$ of $(\sym R)^G_+$. To see that the latter is isotropic, first note that it suffices to check that $\Xi(\mc{N})$ is isotropic i.e. we may assume without loss of generality that $U = \C$. Next, if $\pi : \mu^{-1}(0) \rightarrow \mu^{-1}(0) /\! / G$ is the quotient map, then by \cite[Proposition 4.1]{BernsteinLosev}, $\mc{N}$ is isotropic if and only if $\pi^{-1}(\mc{N})$ is isotropic in $\mu^{-1}(0)$. The closed set $\pi^{-1}(\mc{N})$ equals the Lagrangian $\Lambda$ defined in section \ref{sec:admissibledefn}. In particular, it is isotropic.  
\end{proof}

\begin{proposition}\label{prop:simpleholonomicO}
	If $\mf{A}_{\chi}(U)$ is simple then $\dim V(L) = \dim X /\!/ G$, for all simple objects $L \in \mc{O}_{\chi}(U)$. 
\end{proposition}

\begin{proof}
	Let $L \in \mc{O}_{\chi}(U)$ be simple. If $\mf{p}$ is the annihilator of $L$ in $\mf{A}_{\chi}(U)$, then the generalized Bernstein inequality \cite[Theorem 1.2]{BernsteinLosev} says that $\dim V(\mf{p}) = 2 \dim V(L)$. Since $\mf{A}_{\chi}(U)$ is assumed to be simple, $\mf{p} = (0)$ and hence $2 \dim V(L) = \dim \mu^{-1}(0)/\!/ G = 2 \dim X /\!/ G$. 
\end{proof}

\begin{corollary}
		If $\mf{A}_{\chi}(U)$ is prime, then the following are equivalent:
	\begin{enumerate}
		\item $\mf{A}_{\chi}(U)$ is simple. 
		\item $\dim V(L) = \dim X /\!/ G$, for all simple objects $L \in \mc{O}_{\chi}(U)$. 
	\end{enumerate}
\end{corollary}

\begin{proof}
	By Proposition \ref{prop:simpleholonomicO}, we may assume that $\dim V(L) = \dim X /\!/ G$, for all simple objects $L \in \mc{O}_{\chi}(U)$. If $I \neq (0)$ is a proper ideal of $\mf{A}_{\chi}(U)$, then it is contained in some maximal ideal $\mf{m}$. By the generalized Duflo Theorem \cite[Theorem 2.3]{Primitive}, there exists some simple object $L \in \mc{O}_{\chi}(U)$ whose annihilator equals $\mf{m}$. Again, by the generalized Bernstein inequality \cite[Theorem 1.2]{BernsteinLosev}$, \dim V(\mf{m}) = 2 \dim V(L)$. Thus, $\dim V(\mf{m}) = \dim \mu^{-1}(0)/\!/ G$, which equals the Gelfand-Kirillov dimension of $\mf{A}_{\chi}(U)$. But, since we have assumed that $\mf{A}_{\chi}(U)$ is prime, \cite[Korollar 3.5]{BKGelfandKirillov} says that $\dim V(\mf{m}) < \dim \mu^{-1}(0)/\!/ G$ if $\mf{m} \neq 0$. Thus, we deduce that $\mf{A}_{\chi}(U)$ is simple. 	
\end{proof}

Finally, we consider how the algebras $\mf{A}_{\chi}^{\circ}(U)$ are related to $\mf{A}_{\chi}(U)$. This will be developed further in section \ref{sec:qhrframedcyclicproof}. 

\begin{lemma}\label{lem:circcommutediagam}
For all $U \in \Rep(G)$, the diagrams 
\begin{equation}\label{eq:commutecircdiag1}
\begin{tikzcd}
\QCoh(\dd_{X^{\circ}},G,\chi) & & \ar[ll,"{}^{\perp} \Ham_{U,\chi}"'] \LMod{\mf{A}_{\chi}^{\circ}(U)} \\
\QCoh(\dd_{X},G,\chi) \ar[u,"j^*"] & & \ar[ll,"{}^{\perp} \Ham_{U,\chi}"'] \LMod{\mf{A}_{\chi}(U)} \ar[u,"{\mathrm{Ind}_{\mf{A}}^{\mf{A}^{\circ}}}"'] 
\end{tikzcd}
\end{equation}
and
\begin{equation}\label{eq:commutecircdiag2}
\begin{tikzcd}
\QCoh(\dd_{X^{\circ}},G,\chi) \ar[rr,"{\Ham_{U,\chi}}"] \ar[d,"{j_*}"'] & & \LMod{\mf{A}_{\chi}^{\circ}(U)}  \ar[d,"{\mathrm{Res}_{\mf{A}}^{\mf{A}^{\circ}}}"] \\   
\QCoh(\dd_{X},G,\chi) \ar[rr,"{\Ham_{U,\chi}}"] & &  \LMod{\mf{A}_{\chi}(U)}  
\end{tikzcd}
\end{equation}
commute. 
\end{lemma}

\begin{proof}
For brevity, write $\mf{A} = \mf{A}_{\chi}(U)$ and $\mf{A}^{\circ} = \mf{A}_{\chi}^{\circ}(U)$. For $M$ a left $\mf{A}_{\chi}(U)$-module, 
$$
{}^{\perp} \Ham_{U,\chi} \left( \mathrm{Ind}_{\mf{A}}^{\mf{A}^{\circ}} M \right) = (\mc{P}_{\chi}(U) |_{X^{\circ}}) \otimes_{\mf{A}^{\circ}} \mf{A}^{\circ}\otimes_{\mf{A}} M, 
$$
and
$$
j^*( {}^{\perp} \Ham_{U,\chi}(M)) = \C[X^{\circ}] \otimes_{\C[X]} \mc{P}_{\chi}(U) \otimes_{\mf{A}} M.
$$
Therefore, the commutativity of (\ref{eq:commutecircdiag1}) follows from the fact that the canonical morphism 
$$
 \C[X^{\circ}] \otimes_{\C[X]} \mc{P}_{\chi}(U) \rightarrow (\mc{P}_{\chi}(U) |_{X^{\circ}}) \otimes_{\mf{A}^{\circ}} \mf{A}^{\circ}
$$
of $(\dd_{X^{\circ}}, \mf{A}_{\chi}(U) )$-bimodules is an isomorphism. The fact that (\ref{eq:commutecircdiag2}) commutes follows from the commutativity of (\ref{eq:commutecircdiag1}) by noting that $\mathrm{Res}_{\mf{A}}^{\mf{A}^{\circ}} \circ   \Ham_{U,\chi}$ is right adjoint to ${}^{\perp} \Ham_{U,\chi} \circ \mathrm{Ind}_{\mf{A}}^{\mf{A}^{\circ}}$ and $\Ham_{U,\chi} \circ j_*$ is right adjoint to $ j^*\circ {}^{\perp} \Ham_{U,\chi}$. 
\end{proof}

\begin{proposition}\label{prop:resXcircequigeneric}
The following are equivalent:
\begin{enumerate}
\item[(a)] The map $\mf{A}_{\chi}(U) \rightarrow \mf{A}_{\chi}^{\circ}(U)$ is an isomorphism for all $U \in \Rep(G)$. 
\item[(b)] Restriction is an equivalence $j^* : \QCoh(\dd_{X},G,\chi) \stackrel{\sim}{\longrightarrow} \QCoh(\dd_{X^{\circ}},G,\chi)$, with quasi-inverse $j_*$. 
\end{enumerate}
\end{proposition}

\begin{proof}
Let $A \rightarrow B$ be a ring homomorphism. Then we exploit the fact that this morphism is an isomoprhism if and only if the adjunction $\mathrm{id} \rightarrow \mathrm{Res}^B_A \mathrm{Ind}_A^B $ is an isomorphism.  Lemma \ref{lem:circcommutediagam} implies that 
$$
\mathrm{Res}_{\mf{A}}^{\mf{A}^{\circ}} \circ \mathrm{Ind}_{\mf{A}}^{\mf{A}^{\circ}} \simeq \Ham_{U,\chi} \circ j_* j^* \circ {}^{\perp} \Ham_{U,\chi}. 
$$
By Theorem \ref{thm:Hamequivall}, we may choose $U$ sufficiently large so that $\Ham_{U,\chi}$ is an equivalence, with inverse ${}^{\perp} \Ham_{U,\chi}$. Then if $\mf{A}_{\chi}(U) \rightarrow \mf{A}_{\chi}^{\circ}(U)$ is an isomorphism, it follows that $j^*$ is an equivalence, with inverse $j_*$. 

Conversely, if $j_* j^* \simeq \mathrm{id}$, then for all $U \in \Rep(G)$, we have  
$$
\mathrm{Res}_{\mf{A}}^{\mf{A}^{\circ}} \circ \mathrm{Ind}_{\mf{A}}^{\mf{A}^{\circ}} \simeq \Ham_{U,\chi} \circ {}^{\perp} \Ham_{U,\chi} \simeq \mathrm{id}. 
$$
and hence $\mf{A}_{\chi}(U) \rightarrow \mf{A}_{\chi}^{\circ}(U)$ is an isomorphism.
\end{proof}

\section{The framed cyclic quiver}\label{sect:framedquivertheory}

Generalizing the case considered originally by Gan and Ginzburg \cite{AlmostCommutingVariety}, admissible $\dd$-modules on the framed cyclic quiver are a rich class of examples of admissible $\dd$-modules in large part because of their connection with the representation theory of rational Cherendnik algebras. This connection will be recalled in the final section. First, as in the introduction, $\Q(\ell)$ denotes the cyclic quiver, with vertices $\{ 0, \ds, \ell-1 \}$ and arrows $a_i : i \rightarrow i+1$. Then $\Q_{\infty}(\ell)$ is the framed cyclic quiver with framing $\infty \rightarrow 0$. If $\varepsilon_i$ is the dimension vector with $1$ at vertex $i$ and zero elsewhere, then $\delta = \varepsilon_{0} + \cdots + \varepsilon_{\ell-1}$ is the minimal imaginary root for $\Q(\ell)$, and we fix $\mbf{v} := \varepsilon_{\infty} + n \delta$, a dimension vector for $\Q_{\infty}(\ell)$. Fix $G = G(n \delta)$ and denote by $X$ the space $\Rep(\Q_{\infty}(\ell),\mbf{v})$ of $\mbf{v}$-dimensional representations of the framed cyclic quiver. Then $X$ is a finite dimensional $G$-representation. We write $G(n \delta) = G_0 \times \cdots \times G_{\ell-1}$, where $G_i \simeq \mathrm{GL}_n$ acts on the vector space at vertex $i$ of $\Q_{\infty}(\ell)$. The corresponding decomposition of $\mf{g}$ is $\mf{g}_0 \oplus \cdots \oplus \mf{g}_{\ell-1}$. The character $g \mapsto \det(g_i)$ of $G$ is denoted $\det_i$ and its differential is $\Tr_i$. This defines an isomorphism $\C^{\ell} \iso \mathbb{X}^*(\mf{g})$, given by
\begin{equation}\label{eq:Xgidentification}
(\chi_0, \ds, \chi_{\ell-1}) \mapsto \chi = \sum_{i = 0}^{\ell-1} \chi_i \Tr_i. 
\end{equation}

\subsection{Combinatorics}\label{sec:framedcombinatorics}

The set of all partitions is denoted $\mathcal{P}$ and $\mathcal{P}_{\ell}$ denotes the set of all $\ell$-multipartitions. The subset of $\mc{P}$, resp. of $\mc{P}_{\ell}$, consisting of all partitions of $n\in\mathbf{N}$, resp. of all $\ell$-multipartitions of $n$, is denoted $\mc{P}(n)$, resp. $\mc{P}_{\ell}(n)$. Let $\lambda$ be a partition. The associated Young diagram is
$$
Y(\lambda) = \left\{ (i,j) \in \Z_{\ge 0}^2 \ | \ 1 \le j \le \ell(\lambda), 1 \le i \le \lambda_j \right\}
$$
and the content $\mathrm{ct}(\Box)$ of the box $\Box \in Y(\lambda)$ in position $(i,j)$ is the integer $j - i$.

The underlying graph of the quiver $\Q(\ell)$ is the Dynkin diagram of type $\widetilde{\mathsf{A}}_{\ell-1}$. Therefore we will identify the lattice of virtual dimension vectors $\Z^{\ell}$ for $\Q(\ell)$ with the root lattice $Q$ of type $\widetilde{\mathsf{A}}_{\ell-1}$. We denote by $\mathsf{R} \subset Q$ the set of \textit{roots} and $\mathsf{R}^+ = \mathsf{R} \cap Q^+$ the subset of \textit{positive roots} associated to $\widetilde{\mathsf{A}}_{\ell-1}$. If $\delta = \varepsilon_{0} + \cdots + \varepsilon_{\ell-1}$ denotes the minimal imaginary root and $\Phi := \{ \alpha \in \mathsf{R} \ | \ \varepsilon_0 \cdot \alpha = 0 \}$ is the finite root system of type $\mathsf{A}_{\ell-1}$, then 
$$
\mathsf{R} = \{ n \delta + \alpha \ |  \ n \in \Z, \ \alpha \in \Phi \cup \{ 0 \} \} \smallsetminus \{ 0 \}. 
$$
We fix a generator $\sigma$ of the cyclic group $\Z_{\ell}$. Given a partition $\lambda$, the $\ell$-residue of $\lambda$ is defined to be the element $\res_{\ell}(\lambda) := \sum_{\Box \in \lambda} \sigma^{\mathrm{ct}(\Box)}$ in the group algebra $\Z[\Z_{\ell}]$. Similarly, given an $\ell$-multipartition $\nu$, the shifted $\ell$-residue of $\nu$ is defined to be 
$$
\sres_{\ell}(\nu) = \sum_{i = 0}^{\ell-1} \sigma^i \res_{\ell}(\nu^{(i)}). 
$$
We identify the root lattice $Q$ with $\Z [\Z_{\ell}]$ by $\varepsilon_i \mapsto \sigma^i$.

\subsection{Local systems}\label{sec:localcomp} If we take $\Q = \Q_{\infty}(\ell)$ and let $I$ be the two-sided ideal of $\C \Q$ generated by $(a_{\ell-1} \circ \cdots \circ a_0)^n$, then the space of representations $\Rep(\C \Q/I, \mbf{v})$ is the enhanced cyclic nilpotent cone $\mc{N}_{\infty}(\ell,n)$. 

Recall from Theorem \ref{thm:paramcomb} that the $G$-orbits $\mathcal{O}_{(\lambda;\nu)}$ in the enhanced cyclic nilpotent cone $\mc{N}_{\infty}(\ell,n)$ are labelled by the set 
$$
\mathcal{Q}(n,\ell) = \left\{ (\lambda;\nu) \in \mathcal{P} \times \mathcal{P}_{\ell} \ | \ \res_{\ell}(\lambda) + \sres_{\ell}(\nu) = n \delta\right\}. 
$$
We sketch this parametrization here; the reader is referred to \cite{BB-enhnil-quiver} for details. The $G$-orbits in $\mc{N}_{\infty}(\ell,n)$ correspond to the isomorphism classes of representations of the framed cyclic quiver of dimension $\mathbf{v}$, where the operator $\mathbf{a} := a_{\ell-1} \circ \cdots \circ a_0$ acts nilpotent. The indecomposable representations $M$ of the framed cyclic quiver where $\dim M_{\infty} = 1$ and $\mathbf{a}$ acts nilpotent are parametrized up to isomorphism by the set of all partitions: $\lambda \in \mc{P}$ corresponds to $M_{\lambda}$ such that 
$$
\dim M_{\lambda} = \epsilon_{\infty} + \res_{\ell}(\lambda). 
$$
The indecomposable representations $U$ of the framed cyclic quiver where $\dim U_{\infty} = 0$ and $\mathbf{a}$ acts nilpotent are parametrized up to isomorphism by pairs of positive integers $(i,N)$, where if we think of $N$ as being the partition $(N)$ of $N$, then 
$$
\dim U(i,N) = \sigma^i \res_{\ell}(N)
$$
Therefore an arbitrary representation $M \in \mc{N}_{\infty}(\ell,n)$ will, up to isomorphism, decompose as 
\begin{equation}\label{eq:Mlambdanumodule}
M = M_{\lambda} \oplus \bigoplus_{i = 0}^{\ell-1} U(i,\nu^{(i)}_1) \oplus \cdots \oplus U(i,\nu^{(i)}_r)
\end{equation}
such that 
\begin{align*}
\mathbf{v} & = \dim M = \epsilon_{\infty} + \res_{\ell}(\lambda) + \sum_{i = 0}^{\ell-1} \sigma^i(\res_{\ell}(\nu^{(i)}_1) + \cdots + \res_{\ell}(\nu^{(i)}_r)) \\
 & = \epsilon_{\infty} + \res_{\ell}(\lambda) + \sres_{\ell}(\nu). 
\end{align*}
Thus, to each $(\lambda;\nu) \in \mathcal{Q}(n,\ell)$, we associate the $G$-orbit $\mathcal{O}_{(\lambda;\nu)}$ of the module $M(\lambda;\nu)$ defined in \eqref{eq:Mlambdanumodule}.


\begin{lemma}\label{lem:Gchilocal}
	There exists a $(G,\chi)$-monodromic local system on $\mathcal{O}_{(\lambda;\nu)}$ if and only if 
	$$
	\chi \cdot \sigma^i \res_{\ell}(\nu^{(i)}_j) \in \Z, \quad \forall \ 0 \le i \le \ell-1, 1 \le j \le \ell(\nu^{(i)}). 
	$$
\end{lemma}

\begin{proof}
	First we note that $\mathbb{T}(G) = (\Cs)^{\ell}$ such that the image of $\chi \in \mathbb{X}(\mf{g})$ is given by 
	$$
	q = \exp(\chi) := \left(\exp\left(2 \pi \sqrt{-1} \chi_0\right), \ds,\exp\left(2 \pi \sqrt{-1} \chi_{\ell-1}\right)\right).
	$$
	Let $K$ denote the stabilizer of $M(\lambda;\nu) \in \mathcal{O}_{(\lambda;\nu)}$ so that $\mathcal{O}_{(\lambda;\nu)} \simeq G / K$. Lemma \ref{lem:qmonodromic}, applied with $q = \exp(\chi)$, says that there is a $(G,\chi)$-monodromic local system on $\mathcal{O}_{(\lambda;\nu)}$ if and only if the image of $q$ in $\mathbb{T}(K)$ is $1$. If we set $\beta^{(i,j)} = \dim U(i,\nu^{(i)}_j) = \sigma^i \res_{\ell}(\nu^{(i)}_j)$, then the paragraph after Proposition~\ref{thm:cokerfun} explains that the map $\mathbb{T}(G) \rightarrow \mathbb{T}(K)$ is precisely $\Hom( - , \Cs)$ applied to the morphism \eqref{eq:cokerpiinfty}. This map sends $q = (q_0, \ds, q_{\ell-1})$ in $\mathbb{T}(G)$ to 
	$$
	\left(q^{\beta^{(0,1)}},\ds, q^{\beta^{(\ell-1,k)}}\right) \in \mathbb{T}(K),
	$$
	where $k = \ell(\nu^{(\ell-1)}))$. This equals $1 \in \mathbb{T}(K)$ if and only if 
	$$
	q^{\beta^{(i,j)}} = \exp\left(2 \pi \sqrt{-1} \left(\chi \cdot \sigma^i \res_{\ell}(\nu^{(i)}_j)\right)\right) = 1
	$$
	for all $i,j$ i.e. if and only if $\chi \cdot \sigma^i \res_{\ell}(\nu^{(i)}_j) \in \Z$. 
\end{proof}

Recall from section \ref{sec:countingitro} that we have defined the set $\mathcal{Q}_{\chi}(n,\ell)$. Lemma \ref{lem:Gchilocal} and Proposition \ref{prop:orbprop} imply Theorem \ref{thm:maincount}, as stated in the introduction. Next we note that:

\begin{lemma}\label{lem:multitauf}
	The fundamental group $\pi_1(\mathcal{O}_{(\lambda;\nu)})$ equals $\Z^{\ell}$ if and only if $\nu = \emptyset$. 
\end{lemma}

\begin{proof}
	The orbit $\mc{O}_{(\lambda;\nu)}$ corresponds to an indecomposable representation (with dimension vector $\varepsilon_{\infty} + n \delta$) if and only if $\nu = \emptyset$. In this case, $k = 1$ and the cokernel of (\ref{eq:cokerpiinfty}) is just $\Z^{\ell}$. If $\nu \neq \emptyset$, then $B' \neq 0$ and hence the rank of the image of $B'$ is at least one, implying that the rank of the cokernel is at most $\Z^{\ell - 1}$. 
\end{proof}

We recall from (\ref{eq:hyperplanes}) in the introduction that $\mc{R}_{n} = \{ \alpha \in \mathsf{R}^+ \  | \ \varepsilon_0 \cdot \alpha < n \} \cup \{ n \delta \}$.  

\begin{lemma}\label{lem:3equiv4}
	We have $| \mathcal{Q}_{\chi}(n,\ell)| = |\mathcal{P}_{\ell}(n)|$ if and only if $\chi \cdot \alpha \notin \Z$ for all $\alpha \in \mc{R}_n$. 
\end{lemma}

\begin{proof}
	We identify $\mathcal{P}_{\ell}(n)$ with the set of all partitions $\lambda \in \mc{P}$ with $\res_{\ell}(\lambda) = n \delta$; this is combinatorial identification between the set of all partitions of $n \ell$ with trivial $\ell$-core and the set of all $\ell$-multipartitions of $n$. Therefore we can identify $\mathcal{P}_{\ell}(n)$ with the subset of $\mathcal{Q}(n,\ell)$ consisting of all $(\lambda;\nu)$ such that $\nu = \emptyset$. By Lemma \ref{lem:Gchilocal}, there exists a $(G,\chi)$-monodromic local system on $\mathcal{O}_{(\lambda;\nu)}$ if and only if 
	\begin{equation}\label{eq:existmonolocalsystem}
	\chi \cdot \sigma^i \res_{\ell}(\nu^{(i)}_j) \in \Z, \quad \forall \ 0 \le i \le \ell-1, 1 \le j \le \ell(\nu^{(i)}).
	\end{equation}
	In particular, if $\nu = \emptyset$ then this condition is vacuous and hence there is always a $(G,\chi)$-monodromic local system on $\mathcal{O}_{(\lambda;\emptyset)}$. Thus, $\mathcal{P}_{\ell} \subset \mathcal{Q}_{\chi}(n,\ell)$ for all $\chi$ and the lemma is really claiming that there exists $(\lambda;\nu) \in \mathcal{Q}_{\chi}(n,\ell)$ with $\nu \neq \emptyset$ if and only if $\chi \cdot \alpha \in \Z$ for some $\alpha \in \mc{R}_n$. Recall from section \ref{sec:framedcombinatorics} that $\mathsf{R}^+$ denotes the set of positive roots in the root lattice $Q$. For each positive integer $N$ and $0 \le i \le \ell-1$, the dimension vector $\sigma^i \res_{\ell}(N)$ is a root in $\mathsf{R}^+$. \\
	
	First, let us assume that $\chi \cdot \alpha \notin \Z$ for all $\alpha \in \mc{R}_n$ and choose $(\lambda;\nu)$ with $\nu \neq \emptyset$. If $\lambda \neq \emptyset$ then $\sum_{i,j}  \sigma^i \res_{\ell}(\nu^{(i)}_j) <  n \delta$ and hence $\epsilon_0 \cdot \sigma^i \res_{\ell}(\nu^{(i)}_j) < n$ for all $i,j$. This means that $\sigma^i \res_{\ell}(\nu^{(i)}_j)$ belongs to $\{ \alpha \in \mathsf{R}^+ \  | \ \varepsilon_0 \cdot \alpha < n \}$ and \eqref{eq:existmonolocalsystem} is violated for all $i,j$. Hence $(\lambda;\nu) \notin \mathcal{Q}_{\chi}(n,\ell)$. The only way one can have $\epsilon_0 \cdot \sigma^i \res_{\ell}(\nu^{(i)}_j) = n$ for all $i,j$ is if $\nu = (\ds,\emptyset, (n\ell),\emptyset, \ds)$. In this case there is only one term in the sum, which equals $n \delta$, and $\chi \cdot n \delta \notin \Z$ again violates \eqref{eq:existmonolocalsystem}. \\    
		
	The converse requires some case by case analysis. To begin with, we will assume that $\chi \cdot n \delta \in \Z$. Let $M_0$ be the unique representation with dimension vector $\varepsilon_{\infty}$ and let $M_1$ be any indecomposable nilpotent representation with dimension vector $n \delta$; there are $\ell$ to choose from. Take $\mc{O} = G \cdot M$, where $M = M_0 \oplus M_1$. Then there is some $i$ such that $\nu^{(i)} = (n \ell)$ and $\nu^{(j)} = \lambda = \emptyset$ for $j \neq i$. The orbit $\mc{O}_{(\lambda;\nu)}$ admits a $(G,\chi)$-monodromic local system and hence $| \mathcal{Q}_{\chi}(n,\ell)| > |\mathcal{P}_{\ell}(n)|$. 
	
	Next, we consider the case $\chi \cdot \alpha \in \Z$ for some $\alpha \in \mathsf{R}^+$ with $\varepsilon_0 \cdot \alpha < n$. Recall from section \ref{sec:framedcombinatorics} that $\Phi = \{ \alpha \in \mathsf{R} \ | \ \varepsilon_0 \cdot \alpha = 0 \}$ is the finite root system of type $\mathsf{A}_{\ell-1}$ and $\mathsf{R}^+$ is the set of all $m \delta + \alpha$ with $m \ge 0$ and $\alpha \in \Phi \cup \{ 0 \}$ such that $m \delta + \alpha \ge 0$. For $\alpha \in \mathsf{R}^+$ with $\varepsilon_0 \cdot \alpha < n$, there are three specific cases to consider: (1) $\alpha = m \delta$, (2) $\alpha = m \delta + \beta$, for $\beta \in \Phi^+$, and (3) $\alpha = m \delta - \beta$, for $\beta \in \Phi^+$.\\ 
	
	Case (1). $\alpha =  m \delta$ for some $0 < m < n$. The indecomposable module $U(0,m \ell)$ has dimension vector $m \delta$. Choose any partition $\lambda$ of $n - m$. Then there is a unique indecomposable nilpotent representation $M_{\lambda}$ corresponding to $\lambda$. The orbit $G \cdot M$, where $M = M_{\lambda} \oplus U(0,m \ell)$, equals $\mc{O}_{(\lambda,\nu)}$, where $\nu^{(0)} = (\ell m)$ and $\nu^{(i)} =\emptyset$ otherwise. Since $\chi \cdots \alpha \in \Z$, this orbit admits a $(G,\chi)$-monodromic local system. \\
	
	Case (2). $ \alpha = m \delta + \beta$, for $\beta \in \Phi^+$, with $0 \le m < n$. Then 
	$$
	(\varepsilon_{\infty} + n \delta) - \alpha = \varepsilon_{\infty} + \varepsilon_{j+1} + \varepsilon_{j+2} + \cdots + \varepsilon_{\ell-1} + \varepsilon_{0} + \varepsilon_{1} + \cdots + \varepsilon_{i-1} + (n-m-1) \delta
	$$
	for some $1 \le i \le j \le \ell-1$. This means that $(\varepsilon_{\infty} + n \delta) - \alpha  = \mbf{d}_{\lambda}$, for $\lambda$ the hook partition $(i + (n-m-1) \ell, 1^{\ell-1-j})$. Since $\alpha$ is a real root, there is a unique indecomposable (nilpotent) representation $M_1$ with dimension vector $\alpha$. If $M_0$ is the indecomposable labeled by $\lambda$, then take $M = M_0 \oplus M_1$. Its orbit $\mc{O}_{(\lambda,\nu)}$ has the property that $(\lambda,\nu)$ belongs to $\mathcal{Q}_{\chi}(n,\ell) \smallsetminus \mathcal{P}_{\ell}(n)$. \\
	
	Case (3). $ \alpha = m \delta - \beta$, for $\beta \in \Phi^+$, with $0 < m < n$. Then 
	$$
	(\varepsilon_{\infty} + n \delta) - \alpha = \varepsilon_{\infty} + \varepsilon_{i} + \varepsilon_{i+} + \cdots + \varepsilon_{j-1} + \varepsilon_{j} + (n-m) \delta
	$$
	for some $1 \le i \le j \le \ell-1$. This means that $(\varepsilon_{\infty} + n \delta) - \alpha  = \mbf{d}_{\lambda}$, for $\lambda$ the hook partition $(j+1 + (n-m-1) \ell, 1^{\ell-i})$. Again, since $\alpha$ is a real root, there is a unique indecomposable (nilpotent) representation $M_1$ with dimension vector $\alpha$. If $M_0$ is the indecomposable labeled by $\lambda$, then take $M = M_0 \oplus M_1$. Its orbit $\mc{O}_{(\lambda,\nu)}$ has the property that $(\lambda,\nu)$ belongs to $\mathcal{Q}_{\chi}(n,\ell) \smallsetminus \mathcal{P}_{\ell}(n)$. 
\end{proof}

We say that $\chi \in \mathbb{X}^*(\mf{g})$ is \textit{integral} if it is in the image of $d : \mathbb{X}^*(G) \rightarrow \mathbb{X}^*(\mf{g})$. Under the identification (\ref{eq:Xgidentification}) this is equivalent to $\chi_i \in \Z$ for all $i$. 

\begin{lemma}\label{lem:intequality}
	We have $\mc{Q}_{\chi}(n,\ell) = \mc{Q}(n,\ell)$ if and only if $\chi$ is integral. 
\end{lemma}

\begin{proof}
	That we have equality if $\chi$ is integral is clear; the trivial local system on each orbit is $(G,\chi)$-monodromic in this case. Conversely, if $\chi$ is not integral, then there exists some $i$ such that $\chi_i = \varepsilon_i \cdot \chi \notin \Z$. If $i \neq 0$, then, just as in the proof of Lemma \ref{lem:3equiv4}, since $\varepsilon_i \in \mc{R}_n$ there exists a nilpotent representation $M = M_0 \oplus M_1$, with $M_0$ and $M_1$ indecomposable such that $\dim M_1 = \varepsilon_i$; this is Case (2) of the proof of Lemma \ref{lem:3equiv4}. By Lemma \ref{lem:Gchilocal}, there does not exists a $(G,\chi)$-monodromic local system on $G \cdot M$. Hence $\mc{Q}_{\chi}(n,\ell) \subsetneq \mc{Q}(n,\ell)$. If $0$ is the only $i$ for which $\chi_i \notin \Z$, then $\chi \cdot \delta \notin \Z$ too. Again, as in Case (1) of the proof of Lemma \ref{lem:3equiv4}, this implies that there is some orbit $\mc{O}_{(\lambda,\nu)}$ such that $(\lambda,\nu)$ is not in $\mc{Q}_{\chi}(n,\ell)$. 
\end{proof} 

\begin{example}
	Let $\ell = n = 2$. Then, a computer computation shows that, for integral $\chi$, $\mc{Q}_{\chi}(2,2) = \mc{Q}(2,2)$ has $41$ elements. But for generic $\chi$, Lemma \ref{lem:3equiv4} implies that
	$$
	\mc{Q}_{\chi}(2,2) = \{ (\lambda;\emptyset) \, | \, \lambda \in \mc{P}_2(2) \}, 
	$$
	which has only $5$ elements.   
\end{example}

\subsection{Quantum Hamiltonian reduction}\label{sec:qhrframedcyclicproof}

As noted in section \ref{sec:admissible}, the characteristic variety of each admissible $\dd$-module lies in 
$$
\Lambda = \bigcup_{(\lambda;\mu) \in \mc{Q}(n,\ell)} \Lambda_{(\lambda;\nu)}, 
$$
where $ \Lambda_{(\lambda;\nu)} = \mathbb{F}\left( \overline{T^*_{X^{\vee}} \mc{O}_{(\lambda;\nu)}} \right)$ is the Fourier transform of the conormal to the orbit $\mc{O}_{(\lambda;\nu)}$ in $X^{\vee}$. In particular, the characteristic cycle  $\SS(\ms{M})$ belongs to $\bigoplus_{(\lambda;\nu) \in \mc{Q}(n,\ell)} \Z \Lambda_{(\lambda;\nu)}$. 

\begin{proposition}
	The characteristic cycles map 
	$$
	\SS : K_0( \Add_{\chi}) \rightarrow \bigoplus_{(\lambda;\nu) \in \mc{Q}(n,\ell)} \Z \Lambda_{(\lambda;\nu)}
	$$
	is always injective. It is an isomorphism if and only if $\chi$ is integral. 
\end{proposition}

\begin{proof}
	Since $\Add_{\chi}$ has only finitely many simple objects and every object has finite length, $K_0( \Add_{\chi})$ is a free $\Z$-module with basis given by the class of the simple objects. Each of these objects is the Fourier transform of an intersection cohomology sheaf $\IC(\mc{O}_{(\lambda;\nu)},\mc{L}_{\chi})$, where $(\lambda;\mu) \in \mc{Q}_{\chi}(n,\ell)$ and $\mc{L}_{\chi}$ is the corresponding $(G,\chi)$-monodromic local system on $\mc{O}_{(\lambda,\nu)}$. In particular, we note that the multiplicity of the cycle $\mathbb{F}\left( \overline{T^*_{X^{\vee}} \mc{O}_{(\lambda;\nu)}} \right)$ in $\SS(\mathbb{F}(\IC(\mc{O}_{(\lambda;\nu)},\mc{L}_{\chi})))$ is precisely one since the rank of $\mc{L}_{\chi}$ is one. Thus, 
	$$
	K_0( \Add_{\chi}) = \bigoplus_{(\lambda;\nu) \in \mc{Q}_{\chi}(n,\ell)} \Z[\mathbb{F}(\IC(\mc{O}_{(\lambda;\nu)},\mc{L}_{\chi}))]
	$$
	and the map $\SS$ is injective. In fact, its image, as a $\Z$-module, is a direct summand. The final statement follows from Lemma \ref{lem:intequality}.  
\end{proof}

A point of $X$ is a pair $(v,x_{\idot})$, where $v \in V_0$ and $x_i : V_i \rightarrow V_{i+1}$, for $i = 0, \ds, \ell-1$. Here $V_i$ is the $n$-dimensional vector space at vertex $i$. Let $\mathbf{x} : V_0 \rightarrow V_0$ be the product $\mathbf{x} = x_{\ell-1} \cdots x_0$. We say that a representation $(v,x_{\idot})$ is \textit{cyclic} if $\C[\mathbf{x}] \cdot v= V_0$. The open set of all cyclic representations is denoted $X^{\circ}$. Define
$$
s : X \rightarrow \C, \quad s(v,x_{\idot}) = \det(v,\mathbf{x}(v), \ds, \mathbf{x}^{n-1}(v)). 
$$
It is a semi-invariant of $G$ and $X^{\circ}$ is precisely the non-vanishing locus of $s$. In particular, $X^{\circ}$ is an affine open $G$-stable subset of $X$. Let $X^{\reg}$ denote the open subset of $X^{\circ}$, where $\mathbf{x}$ is regular semi-simple as an element of $\mathfrak{gl}(V_0)$. Under the quotient map $X^{\circ} \rightarrow X/\!/G \simeq \mf{h} / W$, $X^{\reg}$ is the preimage of the regular locus. It is a principal open set, and the quotient map is a principal $G$-bundle.

We check that the assumptions (i)-(iii) of Proposition \ref{prop:UcircAisprime} hold in our situation.

\begin{lemma}\label{lem:cyclicindecomp}
	\begin{enumerate}
		\item[(0)] The set $X^{\circ}$ equals the open subset of $X$ consisting of all indecomposable representations of $\Q_{\infty}(\ell)$. 
		\item[(i)] The moment map $\mu : T^* X \rightarrow \mf{g}^*$ is flat.
		\item[(ii)] There is a $G$-saturated open subset of $X^{\circ}$ on which $G$ acts freely.
		\item[(iii)] The scheme $\mu^{-1}(0)^{\circ} /\!/ G$ is reduced and irreducible.
	\end{enumerate}
\end{lemma}

\begin{proof}
	Part (0). If $\mathbf{x}$ is nilpotent, then it is a consequence of the classification result \cite[Theorem 1.2]{BB-enhnil-quiver} that $(v,x_{\idot})$ is cyclic if and only if it is indecomposable. One can reduce the general case to the nilpotent case by considering the generalized eigenspaces of $\mathbf{x}$ in $V_0$. 
	
	Parts (i) and (ii). The moment map is flat by \cite[Theorem 3.7]{AlmostCommutingVariety}. The open subset $X^{\reg}$ satisfies statement (ii). 
	
	For part (iii), we first note that $\mu^{-1}(0)$ is reduced by \cite[Theorem 3.7]{AlmostCommutingVariety}. It is not irreducible however. Part (0) shows that $\mu^{-1}(0)^{\circ}$ equals the open subset denoted $\mathscr{M}_n$ in section 3.3 of \cite{AlmostCommutingVariety}. In particular, it is irreducible by \cite[Theorem 3.7]{AlmostCommutingVariety}. Therefore, the scheme $\mu^{-1}(0)^{\circ} /\!/ G$ is reduced and irreducible.  
\end{proof}

Lemma \ref{lem:cyclicindecomp} implies that we can freely apply the results from section \ref{sec:holossQHR}. By Proposition \ref{prop:UcircAisprime}, we deduce that the algebras $\mf{A}_{\chi}^{\circ}(U)$ are all prime. It is \textit{not} true that the algebras $\mf{A}_{\chi}(U)$ are prime in general. The following result will be required later. 

\begin{lemma}\label{lem:genericlocuslocalsystemresXreg}
	For any $\ms{M}$ in $\Add_{\chi}$, $\ms{M} |_{X^{\reg}}$ is a local system.
\end{lemma}

\begin{proof}
	The characteristic variety of $\ms{M} |_{X^{\reg}}$ is contained in $\Lambda \cap T^* X^{\reg}$. Since every object in $\Add_{\chi}$ has regular singularities, it suffices to show that $\Lambda \cap T^* X^{\reg}$ is the zero-section $T^*_{X^{\reg}} X^{\reg}$. If $\mathbf{V} = V_0 \oplus \cdots \oplus V_{\ell-1}$, then a point $(v,x) \in X$ may be thought of as a pair $v \in V_0 \subset \mathbf{V}$ and $x \in \End_{\C}(\mathbf{V})$. Then $T^* X$ embeds in $\mc{M} :=  \mathbf{V} \times \mathbf{V}^* \times \mf{gl}(\mathbf{V}) \times \mf{gl}(\mathbf{V}) $ and the moment map $\mu : T^* X \rightarrow \mf{g}$ is just the restriction of the moment map $\boldsymbol{\mu} : \mc{M} \rightarrow \mf{gl}(\mathbf{V})$. Hence $\mu^{-1}(0) = \boldsymbol{\mu}^{-1}(0) \cap T^* X$. By \cite[Lemma 2.1.3]{AlmostCommutingVariety}, this implies that if $(v,w,x,y) \in \mu^{-1}(0)$, then $w \in \mathbf{V}^*$ vanishes on $\C\langle x,y \rangle \cdot v$. Since $v$ is assumed to be cyclic for $x$, this forces $w = 0$. 
	
	Forgetting $v$, the fact that $w = 0$ implies that we may think of $(x,y)$ as a representation of the preprojective algebra $\Pi(\Q(\ell))$. As explained in \cite{GordonCyclicQuiver}, the space $\mf{h}$ embeds as diagonal matrices  in $\Rep(\Q(\ell),n \delta)$ and if $(v,x) \in X^{\reg}$, then $x$ is conjugate by $G$ to an element in $\mf{h}^{\reg}$. Therefore we may assume that $x \in \mf{h}^{\reg}$. Now, $\mf{h}^*$ similarly embeds in $\Rep (\Q(\ell)^{\mathrm{op}},n \delta)$, and for any $y \in \mf{h}^*$, we clearly have $(x,y) \in \mu_{\Q(\ell)}^{-1}(0)$. On the other hand, by \cite[Lemma 3.1]{CBmomap}, the space of all $y$ in $\Rep(\Q(\ell)^{\mathrm{op}},n \delta)$ satisfying $(x,y) \in \mu_{\Q(\ell)}^{-1}(0)$ is a torsor over $\Ext^1_{\Q(\ell)}(x,x)^*$. Since $x$ is a direct sum of pairwise non-isomorphic simple representations of dimension $\delta$, the standard Euler-Poincar\'e formula implies that $\dim \Ext^1_{\Q(\ell)}(x,x)^* = \dim \End_{\Q(\ell)}(x) = n$. Thus, we deduce that $y \in \mf{h}$. On the other hand, $y$ must belong to $\mc{N}$. This forces $y = 0$, as required. 
\end{proof} 

As in section \ref{sec:sympleavesopenXcirc}, let $j : X^{\circ} \hookrightarrow X$ be the open embedding.

\begin{proposition}\label{prop:Xcyclicres}
	The following are equivalent: 
	\begin{enumerate}
		\item[(a)] $\chi \cdot \alpha \notin \Z$ for all $\alpha \in \mc{R}_n$.
		\item[(b)] The restriction $j^* : \QCoh(\dd_X,G,\chi) \rightarrow \QCoh(\dd_{X^{\circ}},G,\chi)$ is an equivalence, with inverse $j_*$. 
	\end{enumerate}
\end{proposition}

\begin{proof}
	We recall from Lemma \ref{lem:multitauf} that the fundamental group $\pi_1(\mc{O})$ of a $G$-orbit $\mc{O}$ in $X$ is a quotient of $\Z^{\ell}$, and it follows from Lemma \ref{lem:multitauf} that $\pi_1(\mc{O}) = \Z^{\ell}$ if and only if $\mc{O}$ parameterizes indecomposable representations of $\Q_{\infty}(\ell)$. Therefore Lemma \ref{lem:cyclicindecomp} implies that $\pi_1(\mc{O}) = \Z^{\ell}$ if and only if $\mc{O} \subset X^{\circ}$. 
	
	Now assume that $\chi \cdot \alpha \in \Z$ for some $\alpha \in \mc{R}_n$. Then Lemma \ref{lem:3equiv4} says that there exists an orbit $\mc{O}$ with $\pi_1(\mc{O})$ a proper quotient of $\Z^{\ell}$ such that $\mc{O}$ admits a $(G,\chi)$-monodromic local system $\mc{L}$. Necessarily, $\mc{O} \subset X \smallsetminus X^{\circ}$. We can choose a $(G,\chi)$-monodromic structure on $\mc{L}$. Then $\mathrm{IC}(\mc{O},\mc{L})$ is a $(G,\chi)$-monodromic $\dd$-module supported on the complement of $X^{\circ}$. Thus, $j^*$ kills this module and cannot be an equivalence. 
	
	Conversely, assume that $\chi \cdot \alpha \notin \Z$ for all $\alpha \in \mc{R}_n$. We must show that there are no $(G,\chi)$-monodromic $\dd$-modules supported on the complement of $X^{\circ}$. Let us assume otherwise - $\ms{M}$ is a $(G,\chi)$-monodromic $\dd$-module supported on $X \smallsetminus X^{\circ}$. Since $\ms{M}$ is the colimit of coherent $(G,\chi)$-monodromic $\dd$-modules, we may assume it is coherent. We choose a $G$-stable open subset $U$ of $X$ such that $Y := \Supp \ms{M} \cap U$ is a smooth, closed subvariety of $U$. Let $i : Y \hookrightarrow U$ be the closed embedding. By Kashiwara's Theorem \cite[Theorem 1.6.1]{HTT}, $\ms{N} := i^{\natural} \ms{M}$ is a coherent $(G,\chi)$-monodromic $\dd$-module on $Y$. Its support equals $Y$. Therefore, \cite[Lemma 3.3.2]{HTT} says that there is a dense open subset $V \subset Y$ such that $\ms{N}|_{\mc{O}_V}$ is a \textit{non-zero} projective $\mc{O}_V$-module. We may assume that $V$ is $G$-stable. Choose a $G$-orbit $\mc{O} \subset V$ and let $k \colon \mc{O} \hookrightarrow V$ be the locally closed embedding. Then $\ms{N}' = H^0(k^* \ms{N})$ is a \textit{non-zero} quasi-coherent $\dd_{\mc{O}}$-module since  $\ms{N}|_{\mc{O}_V}$ is projective over $\mc{O}_V$. However, if $K$ is the stabilizer of some point $x_0 \in \mc{O}$, then our assumption on $\chi$, together with the fact that $\mc{O} \subset X \smallsetminus X^{\circ}$, implies that the image of $q$ in $\mathbb{T}(K)$ is not $1$. Therefore, we deduce from Proposition \ref{prop:qmonodromicqc} that there are no $(G,\chi)$-monodromic $\dd$-modules on $\mc{O}$ (recall from (\ref{eq:Sconnected}) that $\mc{O} \simeq G / K$, where $K$ is connected). This contradicts the fact that $\ms{N}' \neq 0$. 
\end{proof} 

We deduce from Proposition \ref{prop:Xcyclicres} and Proposition \ref{prop:resXcircequigeneric} that:

\begin{corollary}\label{cor:UcircUisoconditions}
	The following are equivalent:
	\begin{enumerate}
		\item[(a)] $\chi \cdot \alpha \notin \Z$ for all $\alpha \in \mc{R}_n$.
		\item[(b)] $\varphi_U: \mf{A}_{\chi}(U) \rightarrow \mf{A}_{\chi}^{\circ}(U)$ is an isomorphism \textit{for all} $U \in \Rep(G)$. 
	\end{enumerate}
\end{corollary}

\section{Semi-simplicity}\label{sec:sscriterionframed} 

In this final section, we apply the results of sections \ref{sec:QHR} and \ref{sec:admissible} to the special case $X = \Rep(\Q_{\infty}(\ell),\mathbf{v})$, in order to prove the results stated in the introduction. We begin by recalling Ariki's semi-simplicity criterion. 

\subsection{Rational Cherednik algebras} 

We will fix an identification 
$$
\mathbb{X}^*(\g) = \mathfrak{c} := \left\{ (\kappa_{0,0} , \kappa_{0,1},\kappa_{0}, \ds, \kappa_{\ell-1}) \in \C^2 \oplus \C^{\ell} \ \Big| \ \kappa_{0,0} + \kappa_{0,1} =0,  \sum_{i = 0}^{\ell-1} \kappa_i = 0 \right\}
$$
by 
\begin{equation}\label{eq:kappachi}
\chi_0 =  \frac{1}{\ell} + (\kappa_0 - \kappa_{1}) + (\kappa_{0,0} - \kappa_{0,1}) - 1, \quad \chi_i = \frac{1}{\ell} + (\kappa_i - \kappa_{i+1}), \quad 1 \le i \le \ell-1.
\end{equation}
Equivalently, $\kappa_{0,0} - \kappa_{0,1} = \delta \cdot \chi$ and 
$$
\ell \kappa_{i+1} = i - \sum_{j = 1}^{i} j \chi_j + \sum_{j = i+1}^{\ell-1} (\ell - j) \chi_j, \quad 1 \le i \le \ell-1.
$$ 
Associated to each $\mbf{\kappa} \in \mathfrak{c}$ is the \textit{cyclotomic rational Cherednik algebra} $\H_{\mbf{\kappa}}(W)$, where $W = \Z_{\ell} \wr \s_n$ is the wreath product of the symmetric group $\s_n$ with the cyclic group $\Z_{\ell}$, as defined in \cite[\S 3.1]{BerestChalykhQuasi}. It is a non-commutative algebra containing the group algebra $\C W$ as a subalgebra. Let $e \in \C W$ be the trivial idempotent. The algebra $e \H_{\mbf{\kappa}} (W) e$ contains both $\C[\h]^W$ and $\C[\h^*]^W$, where $\h$ is the reflection representation of $W$. The Harish-Chandra homomorphism allows one to identify the spherical subalgebra $e \H_{\mbf{\kappa}} (W) e$ with a certain algebra of quantum Hamiltonian reduction. Namely, it was shown in \cite{OblomkovHC} and \cite{GordonCyclicQuiver}  that\footnote{The parameters $(k,c_1, \ds, c_{\ell-1})$ used in \cite{GordonCyclicQuiver} are related to the parameters $(\kappa_{0,0} , \kappa_{0,1},\kappa_{0}, \ds, \kappa_{\ell-1})$ of \cite{BerestChalykhQuasi} by $k = \kappa_{0,0} - \kappa_{0,1}$ and $c_r = \sum_{p = 0}^{\ell-1} (\kappa_{p+1} - \kappa_p) \zeta^{pr}$, where $\zeta = \exp\left(\frac{2 \pi \sqrt{-1}}{\ell} \right)$.}:

\begin{theorem}\label{thm:HChomo}
	There is a filtered algebra isomorphism $\mathfrak{R}_{\chi} : \mf{A}_{\chi}(\C) \stackrel{\sim}{\longrightarrow} e \H_{\mbf{\kappa}} (W) e$.
\end{theorem}

Therefore, as a special case of the functor of Hamiltonian reduction described in section \ref{sec:QHR}, there is an exact functor $\Ham_{\chi} : \mathrm{Coh}(\dd_{X},G,\chi) \rightarrow \Lmod{e \H_{\kappa}(W) e}$ defined by 
$$
\Ham_{\chi}(\ms{M}) = \{m \in \Gamma(X,\ms{M}) \ | \ \nu(X) \cdot m = \chi(X) m \ \forall \ X \in \g \} = \Gamma(X,\ms{M})^G. 
$$

\begin{definition}
	Category $\Osph$ is the full subcategory of the category $\Lmod{e \H_{\kappa}(W) e}$ of finitely generated $ e \H_{\kappa}(W) e$-modules consisting of all modules on which $\C[\h^*]^W$ acts locally nilpotent.  
\end{definition}

The Harish-Chandra homomorphism $\mathfrak{R}_{\chi}$ has the property that $\mathfrak{R}_{\chi}(\C[R]^G) = \C[\h]^W$ and similarly $\mathfrak{R}_{\chi}((\sym R)^G) = \C[\h^*]^W$. This implies that the functor of Hamiltonian reduction restricts to an exact quotient functor $\Ham_{\chi} : \Add_{\chi} \rightarrow \Osph$. 

\subsection{Localization of the Cherednik algebra} The following observation is not needed elsewhere. Recall that $\kappa$ (or equivalently $\chi$) is \textit{spherical} if $ e  \H_{\kappa}(W)$ defines a Morita equivalence between $e\H_{\kappa}(W) e$ and $\H_{\kappa}(W)$.

\begin{lemma}\label{lem:factorspherical}
	Let $\ms{Q} := {}^{\perp} \Ham_{\chi}(e  \H_{\kappa}(W))$, an object of $\mathrm{Coh}(\dd_X,G,\chi)$. Then, 
	\begin{enumerate}
		\item[(a)] $\End_{\dd_X}(\ms{Q}) \simeq \H_{\kappa}(W)^{\mathrm{op}}$. 
		\item[(b)] $\ms{Q}$ is projective in $\QCoh(\dd_X,G,\chi)$ if and only if $\chi$ is spherical. 
	\end{enumerate}
\end{lemma}

\begin{proof}
	For brevity, let $A = e\H_{\kappa}(W) e$. It is a consequence of the double centralizer theorem, \cite[Theorem 1.5]{EG}, that $\End_A(e \H_{\kappa}(W)) \simeq \H_{\kappa}(W)^{\mathrm{op}}$. Then part (a) follows from adjunction and the fact that $1 \rightarrow \Ham_{\chi} \circ {}^{\perp} \Ham_{\chi}$ is an isomorphism:
	$$
	\End_{\dd}(\ms{Q}) = \Hom_{A} (e \H_{\kappa}(W), \Ham_{\chi} \circ {}^{\perp} \Ham_{\chi}(e \H_{\kappa}(W))). 
	$$
	Part (b). As a left $A$-module, $e \H_{\kappa}(W)$ is projective if and only if $\chi$ is spherical. Since ${}^{\perp} \Ham_{\chi}$ is left adjoint to the exact functor $\Ham_{\chi}$, this implies that $\ms{Q}$ is projective in $\QCoh(\dd_X,G,\chi)$ if $\chi$ is spherical. If $\chi$ is not spherical then there exists a short exact sequence $0 \rightarrow N_1 \rightarrow N_2 \rightarrow N_3 \rightarrow 0$ of modules in $\Osph$ such that 
	\begin{equation}\label{eq:notses1}
	0 \rightarrow  \Hom_{A}(e \H_{\kappa}(W), N_1) \rightarrow \Hom_{A}(e \H_{\kappa}(W), N_2) \rightarrow \Hom_{A}(e \H_{\kappa}(W),N_3) \rightarrow 0
	\end{equation}
	is not exact. Applying the right exact functor ${}^{\perp} \Ham_{\chi}$ gives an exact sequence ${}^{\perp} \Ham_{\chi}(N_1) \rightarrow {}^{\perp} \Ham_{\chi}(N_2) \rightarrow {}^{\perp} \Ham_{\chi}(N_3) \rightarrow 0$ and we let $N'$ be the quotient of ${}^{\perp} \Ham_{\chi}(N_1)$ making the sequence 
	$$
	\begin{tikzcd}
	0 \ar[r] & N' \ar[r] & {}^{\perp} \Ham_{\chi}(N_2) \ar[r] & {}^{\perp} \Ham_{\chi}(N_3)  \ar[r] & 0
	\end{tikzcd}
	$$
	exact. The fact that $1 \rightarrow \Ham_{\chi} \circ {}^{\perp} \Ham_{\chi}$ is an isomorphism implies that $\Hom_{\dd}(\ms{Q}, N') = \Hom_{\dd}(\ms{Q}, {}^{\perp} \Ham_{\chi}(N_1))$. Thus, the sequence 
	$$
	0 \rightarrow \Hom_{\dd}(\ms{Q}, N') \rightarrow \Hom_{\dd}(\ms{Q}, {}^{\perp} \Ham_{\chi}(N_2)) \rightarrow \Hom_{\dd}(\ms{Q},{}^{\perp} \Ham_{\chi}(N_3)) \rightarrow 0
	$$
	is isomorphic to sequence (\ref{eq:notses1}). In particular, it is not exact. 
\end{proof}

Lemma \ref{lem:factorspherical} implies that when $\chi$ is spherical there is an exact quotient functor 
$$
\Hom_{\dd}(\ms{Q}, - ) : \Coh(\dd_X,G,\chi) \rightarrow \Lmod{\H_{\kappa}(W)},
$$
making the diagram 
$$
\begin{tikzcd}
\Coh(\dd_X,G,\chi) \ar[rr,"{\Hom_{\dd}(\ms{Q}, - )}"] \ar[dr,"\Ham_{\chi}"'] & & \Lmod{\H_{\kappa}(W)} \ar[dl,"e \cdot - "] \\
& \Lmod{e \H_{\kappa}(W) e} &
\end{tikzcd}
$$  
commutative. 

\subsection{The cyclotomic Hecke algebra}

The cyclotomic Hecke algebra $\Ha_q(W)$, for $q \in \mathbb{T}(G)$, is the finite dimensional algebra generated by $T_0, T_1, \ds, T_{n-1}$, satisfying the braid relations
$$
T_0 T_1 T_0 T_1 = T_1 T_0 T_1 T_0, \quad T_i T_{i+1} T_i = T_{i+1} T_i T_{i+1}, \quad T_i T_j = T_j T_i,
$$
for $i,j = 1, \ds, n-1$, $|i-j| > 1$, and the additional relations
$$
\prod_{r = 0}^{\ell-1} (T_0 - u_r) = 0, \quad (T_i - q_0)(T_i - q_1) = 0, \quad \forall \ i > 0.
$$
where $q_0,q_1, u_0, \ds, u_{\ell-1} \in \Cs$. Let $q = - q_1q_0^{-1}$. By Ariki's Theorem \cite{ArikiSemisimple},

\begin{theorem}\label{thm:Ariki}
	The algebra $\mc{H}_q(W)$ is semi-simple if and only if  
	$$
	\prod_{m = 2}^n (1 - q^m) \prod_{i \neq j} (u_i - q^m u_j) \neq 0.
	$$
\end{theorem}

\begin{proof}
	The parametrization given in \cite{ArikiSemisimple} is slightly different. Namely, $\mc{H}_q(W)$ is generated by $a_0, \ds, a_{n-1}$ satisfying the relations of \cite[Definition 0]{ArikiSemisimple}. Then the isomorphism between the two algebras is given by $a_0 = T_0$ and $a_i = - q_0^{-1} T_i$ for $i > 0$.   
\end{proof}

Define
\begin{equation}\label{eq:qkappa}
q_0 = \exp \left(2 \pi \sqrt{-1} \kappa_{0,0}\right), \quad  q_1 = -\exp\left(2 \pi \sqrt{-1} \kappa_{0,1}\right), \quad u_i = \zeta^{-j} \exp\left( 2 \pi \sqrt{-1} \kappa_r\right). 
\end{equation}
Using the Knizhnik-Zamolodchikov functor, on can deduce from Theorem \ref{thm:Ariki} when category $\Osph$  is semi-simple. Namely, by \cite[Theorem 6.6]{BerestChalykhQuasi}, category $\Osph$ is semi-simple if and only if 
\begin{equation}\label{eq:Oss1}
\kappa_{0,0} - \kappa_{0,1} + \frac{j}{m} \notin \Z, \quad 2 \le m \le n, (j,m) = 1. 
\end{equation}
and
\begin{equation}\label{eq:Oss2}
m( \kappa_{0,0}- \kappa_{0,1}) + \kappa_j - \kappa_i + \frac{(i-j)}{\ell} \notin \Z, \quad -n < m < n, i \neq j. 
\end{equation}
We note that if $\Osph$  is semi-simple then the parameter $\chi$ is spherical and the simple objects of $\Osph$ are in bijection with the set of $\ell$-multi-partitions of $n$. 

\subsection{The proof of Theorem \ref{thm:mainequiv1}}\label{sec:proofthm:mainequiv1}

The proof of Theorem \ref{thm:mainequiv1} follows from Theorem \ref{thm:equivcond} and Corollary \ref{cor:threemoreconditions} below. Combining the results of section \ref{sec:localcomp} with the above numerical criterion for the semi-simplicity of $\Osph$, we deduce: 

\begin{theorem}\label{thm:equivcond}
	The following are equivalent:
	\begin{enumerate}
		\item[(a)] $\chi \cdot \alpha \notin \Z$ for $\alpha \in \mc{R}_n$.
		\vspace{2mm}
		\item[(b)] $|\mathcal{Q}_{\chi}(n,\ell)| = |\mathcal{P}_{\ell}(n)|$. 
		\vspace{2mm}
		\item[(c)] $\Ham_{\chi} : \Add_{\chi} \rightarrow \Osph$ is an equivalence. 
	\end{enumerate}
\end{theorem}

\begin{proof}
	Lemma \ref{lem:3equiv4} is precisely the statement that (a) $\Leftrightarrow$ (b). Notice that 
	\begin{multline*}
	\mc{R}_n =  \{ m \delta \ | \ 1 \le m \le n\} \cup \{ m \delta + \varepsilon_i + \cdots + \varepsilon_j \ | \ 0 \le m \le n-1, 1 \le i \le j \le \ell-1 \} \\
	\cup \{ m \delta - \varepsilon_i - \cdots - \varepsilon_j \ | \ 1 \le m \le n-1, 1 \le i \le j \le \ell-1 \}.
	\end{multline*}
	Applying the identification (\ref{eq:kappachi}) between $\chi$ and $\mbf{\kappa}$, we see that (a) holds if and only if 
	\begin{enumerate}
		\item[(A)] $\langle \chi , m \delta \rangle = m \delta \cdot \chi = m (\kappa_{0,0} - \kappa_{0,1}) \notin \Z$, $\forall \ 1 \le m \le n$,
		$$
		\Leftrightarrow \quad \kappa_{0,0} - \kappa_{0,1} + \frac{j}{m} \notin \Z, \quad 1 \le m \le n, (j,m) = 1. 
		$$
		\item[(B)] $\langle \chi , m \delta + \varepsilon_i + \cdots + \varepsilon_j \rangle = m \delta \cdot \chi + (\chi_i + \cdots + \chi_j) \notin \Z$, $\forall \ 0 \le m \le n-1,  1 \le i \le j \le \ell-1$,
		$$
		\Leftrightarrow \quad m(\kappa_{0,0} - \kappa_{0,1}) + (\kappa_i - \kappa_{j+1}) + \frac{j+1-i}{\ell} \notin \Z, \quad 0 \le m \le n-1,  1 \le i \le j \le \ell-1. 
		$$
		\item[(C)] $\langle \chi , m \delta - \varepsilon_i - \cdots - \varepsilon_j \rangle = m \delta \cdot \chi - (\chi_i + \cdots + \chi_j) \notin \Z$, $\forall \ 1 \le m \le n-1, 1 \le i \le j \le \ell-1$,
		$$
		\Leftrightarrow \quad  m(\kappa_{0,0} - \kappa_{0,1}) - (\kappa_i - \kappa_{j+1}) - \frac{j+1-i}{\ell} \notin \Z, \quad 1 \le m \le n-1, 1 \le i \le j \le \ell-1 
		$$
	\end{enumerate}
	This system of equations is equal to the set of equations (\ref{eq:Oss1}) and (\ref{eq:Oss2}), together with the additional condition $\kappa_{0,0} - \kappa_{0,1} \notin \Z$, which comes from $\delta \cdot \chi \notin \Z$. This implies that $\left| \Irr \Osph \right|$ equals $\left|\mathcal{P}_{\ell}(n)\right|$ when (a) holds. Thus, (a) and (b) together imply that the number of simple objects in $\Add_{\chi}$ equals the number of simple objects in $\Osph$. Since $\Ham_{\chi} : \Add_{\chi} \rightarrow \Osph$ is a quotient functor, we deduce that (c) holds. Conversely, if (c) holds then $\left| \Irr \Add_{\chi} \right| = \left| \Irr \Osph \right|$. Since we have 
	$$
	\left| \Irr \Osph \right| \le  \left|\mathcal{P}_{\ell}(n)\right| \quad \textrm{and} \quad \left| \Irr \Add_{\chi}\right|  = \left|\mathcal{Q}_{\chi}(n,\ell)\right| \ge \left|\mathcal{P}_{\ell}(n)\right|
	$$
	for all $\chi$, we see that (c) implies (b). 
\end{proof}

In the statement of Theorem \ref{thm:mainequiv1}, we have taken $k := \kappa_{0,0} - \kappa_{0,1}$. 

\begin{corollary}\label{cor:threemoreconditions}
	The following are equivalent:
	\begin{enumerate}
		\item[(a)] $\Ham_{\chi} : \Coh(\dd_{X},G,\chi) \rightarrow \Lmod{e \H_{\mbf{\kappa}} e}$ is an equivalence. 
		\vspace{2mm}
		\item[(b)] $\Gamma(X,\ms{M})^G \neq 0$ for all non-zero $(G,\chi)$-monodromic $\dd_{X}$-modules $\ms{M}$. 
		\vspace{2mm}
		\item[(c)] $\chi \cdot \alpha \notin \Z$ for all $\alpha \in \mc{R}_n$.
	\end{enumerate}
\end{corollary}

\begin{proof}
	Since $\Ham_{\chi}$ is a quotient functor, and $\Ham_{\chi}(\ms{M}) = \Gamma(X,\ms{M})^G$ on the level of vector spaces, (a) is equivalent to (b). If (a) holds, then in particular $\Ham_{\chi} : \Add_{\chi} \rightarrow \Osph$ is an equivalence, and hence (c) holds by Theorem \ref{thm:equivcond}. Conversely, if $\Ham_{\chi} : \Add_{\chi} \rightarrow \Osph$ is an equivalence then we may take $U_0 = \C$, the trivial $G$-module, in Theorem \ref{thm:Hamequivall} and we deduce that $\Ham_{\chi} : \Coh(\dd_{X},G,\chi) \rightarrow \Lmod{e \H_{\mbf{\kappa}} e}$ is an equivalence.
\end{proof}

In recent work, T. Shoji has shown that there is a close relationship between certain perverse sheaves on $GL(V) \times V^{\ell-1}$ and representations of the group $\s_n \wr \Z_{\ell}$; see the survey \cite{ShojiSurvey} and references within. It seem likely that the categories studied in \cite{ShojiSurvey} are related by symplectic duality to the admissible $\dd$-modules we have considered here. We hope to make this statement precise in future work. 

\subsection{The proof of Theorem \ref{thm:genericbehabousqhr1}}\label{sec:genericbehabousqhr1proof}

In this section we give the proof of Theorem \ref{thm:genericbehabousqhr1}. We assume that $\chi \cdot \alpha \notin \Z$ for all $\alpha \in \mc{R}_n$. This implies that $\Osph$ is semi-simple, which in turn means that $\mf{A}_{\chi}(\C)$ is a simple algebra. Moreover, for all $U$ containing $\C$ as a summand, Theorem \ref{thm:Hamequivall} implies that $\mf{A}_{\chi}(U)$ is Morita equivalent to $\mf{A}_{\chi}(\C)$, and hence is also simple. Next, if $U$ is an arbitrary representation, let $U' = U \oplus \C$. We have $U' \ge \C$ and $U' \ge U$. This implies that $\mf{A}_{\chi}(U')$ is simple. Then, Lemma \ref{lem:closedembeddingspecUU} implies that $\Spec \mf{A}_{\chi}(U)$ embeds in $\Spec \mf{A}_{\chi}(U')$ i.e. $\mf{A}_{\chi}(U)$ has a unique prime ideal. By Theorem \ref{cor:UcircUisoconditions}, $\mf{A}{\chi}(U)$ is isomorphic to $\mf{A}{\chi}^{\circ}(U)$. As noted in section \ref{sec:qhrframedcyclicproof}, the latter ring is prime by Proposition \ref{prop:UcircAisprime}. Therefore, we deduce that $\mf{A}{\chi}(U)$ is simple. 

Next, we must check that all rings are Morita equivalent. But since we can identify $\mf{A}{\chi}(U)$ with $\mf{A}{\chi}^{\circ}(U)$, and we have shown that all of these rings are simple, this follows from Proposition \ref{prop:nonzeroUUbi} (b). 

Finally, since each $\mf{A}_{\chi}(U)$ is simple, Proposition \ref{prop:simpleholonomicO} implies that if $M$ in $\mc{O}_{\chi}(U)$ is non-zero then $V(M) = X/\!/G$.

\subsection{The proof of Theorem \ref{thm:mainequiv}}\label{sec:proofmainframedquiveradd}

As in \cite{mirabolicHam}, the \textit{Harish-Chandra} $\dd$-module is defined to be
$$
\mc{G}_{\chi} := \dd_X / \dd_X \nu_{\chi}(\mf{g}) + \dd_X (\sym R)^G_+.
$$
It is an object of $\Add_{\chi}$. The goal of this section is to give a proof the following theorem, from which we will deduce Theorem \ref{thm:mainequiv}:

\begin{theorem}\label{thm:intss}
If $\chi \cdot \delta \in \Z$ then $\mc{G}_{\chi}$ is not semi-simple. Hence, the category $\Add_{\chi}$ is not semi-simple. 
\end{theorem}

Recall that $R = \Rep (\Q, n \delta)$ so that $X = V \times R$. The action of $G(n \delta)$ on $V \times R$ is the diagonal action, with the action of $G(n \delta)$ on $V$ factoring through $G_0 \simeq GL(V)$. If $\Cs_{\Delta} \subset G(n \delta)$ is the one-dimensional diagonal torus then the action of $G(n \delta)$ on $R$ factors through $PG(n \delta) := G(n \delta) / \Cs_{\Delta}$. This implies that $G(n \delta)$ acts trivially on $\bigwedge^{\dim R} R$. The moment map $\nu : \mf{g} \rightarrow \dd(X)$ can be decomposed as $\nu = \nu_R + \nu_V$, where $\nu_R: \mf{g} \rightarrow \dd(R) \subset \dd(X)$ and $\nu_V: \mf{g} \rightarrow  \dd(V) \subset \dd(X)$. We define the Harish-Chandra $\dd$-module in this context to be 
$$
\mc{F}_{\chi} = \dd_R/ \dd_R \nu_{R,\chi}(\mf{g}) + \dd_R (\sym R)^G_+.
$$

\begin{lemma}
	The module $\mc{F}_{\chi}$ is non-zero if and only if $\chi \cdot \delta = 0$. 
\end{lemma}

\begin{proof}
	If $\chi \cdot \delta \neq 0$ then $\nu_R (\mathbf{1}_{\Delta}) = 0$ but $\chi(\mathbf{1}_{\Delta}) \neq 0$. This implies that $\dd(R) \nu_{R,\chi}(\mf{g}) = \dd(R)$ and hence $\mc{F}_{\chi} = 0$. If $\chi \cdot \delta = 0$ then the result \cite{OblomkovHC} shows that there is a surjective\footnote{One can actually show that the morphism is an isomorphism. Details will appear elsewhere.} algebra homomorphism 
	$$
	\mf{R}_{\chi}' : (\dd(R) / \dd(R) \nu_{R,\chi}(\mf{g}))^G \rightarrow e \H_{\kappa}(\s_n \wr \Z_{\ell}) e,
	$$
	mapping $(\sym R)^G$ isomorphically onto $\C[\h^*]^{\s_n \wr \Z_{\ell}}$. This implies that 
	$$
	\Gamma(R,\mc{F}_{\chi})^G \twoheadrightarrow e \H e / e \H e \ \C[\h^*]^{\s_n \wr \Z_{\ell}}_+. 
	$$
	The module on the right is non-zero by the PBW property for rational Cherednik algebras. Thus, $\Gamma(R,\mc{F}_{\chi}) \neq 0$. 
\end{proof}

If $\mf{g}' = \mathrm{Lie} \ PG(n \delta)$, then the above lemma shows that for $\chi \cdot \delta = 0$, we can consider $\chi$ as a character of $\mf{g}'$ and 
$$
\mc{F}_{\chi} = \dd_R / \dd_R \nu_{R,\chi}(\mf{g}') + \dd_R (\sym R)^{PG}_+.
$$
If we decreed that the canonical generator $1 \in \mc{F}_{\chi}$ is $PG(n \delta)$-invariant then $\mc{F}_{\chi}$ is a $(PG(n \delta), \chi)$-monodromic $\dd$-module. Equivalently, it is a $(G(n \delta), \chi)$-monodromic $\dd$-module such that $\Gamma(R,\mc{F}_{\chi})^{\Cs_{\Delta}} = \Gamma(R,\mc{F}_{\chi})$. We make the skyscraper module $S_0 = \dd_V / \dd_V \C[V]_+ = \dd_V \cdot v_0$ into a $G(n\delta)$-module by saying that the generator $v_0$ is invariant. Since
$$
\nu_V(\mathbf{1}_0) \cdot v_0 = \left(n - \sum_{i = 1}^n \partial_i x_i \right) \cdot v_0 = n v_0, 
$$
this means that $S_0$ is a $(G(n\delta),\Tr_0)$-monodromic $\dd$-module. Define $\chi' := \chi - \Tr_0$, so that 
\begin{equation}\label{eq:chiprime}
\chi'(\mathbf{1}_i) = \left\{ \begin{array}{lll}
\chi(\mathbf{1}_i) & = n \chi_i & i \neq 0 \\
\chi(\mathbf{1}_0) - n & = n \chi_0 - n & i = 0
\end{array}\right.  
\end{equation}
Then $\mc{F}_{\chi'} \boxtimes S_0$ is a $(G(n\delta),\chi' + \Tr_0) = (G(n\delta),\chi)$-monodromic $\dd$-module. Recall from Proposition \ref{prop:dualholo} that if $\ms{M}$ is a holonomic $(G,\chi)$-monodromic module then $\mathbb{D}(\ms{M})$ is a $(G,-\chi)$-monodromic module. 

\begin{lemma}\label{lem:commuteequivdual}
	Let $\mathrm{det}_0 : G(n \delta) \rightarrow \Cs$ be the character $g \mapsto \det(g_0)$. 
	\begin{enumerate}
		\item[(a)] $\mathbb{D}(S_0) \simeq S_0 \otimes \mathrm{det}_0^{\otimes 2}$. 
		\vspace{2mm}
		\item[(b)] $\mathbb{D}(\mc{G}_{\chi}) \simeq \mc{G}_{- \chi + \Tr_0} \otimes \mathrm{det}_0$.
	\end{enumerate}
\end{lemma}

\begin{proof}
	Part (a). Forgetting the equivariant structure, $\mathbb{D}(S_0) = S_0$. On the other hand, as shown above $S_0$ is a $(G(n\delta),\Tr_0)$-monodromic $\dd$-module. Thus, by Proposition \ref{prop:dualityOchiG}, the module $\mathbb{D}(S_0)$ is $(G(n\delta),-\Tr_0)$-monodromic. Since $d \, \mathrm{det}_0^{\otimes 2} = 2 \Tr_0$, it follows that $\mathbb{D}(S_0) \simeq S_0 \otimes \mathrm{det}_0^{\otimes 2}$. 
	
	Part (b). As in \cite{MirabolicCharacter}, if we let $\mf{k} = \nu_{\chi}(\mf{g}) \oplus (\sym R)^{G}_+$, then $\dd(X)$ is a flat right $\mf{k}$-module. Therefore Lemma \ref{lem:Afreemfgmod} says that 
	$$
	\Ext_{\dd}^{N}(\mc{G}_{\chi},\dd) \simeq \mf{k} \dd \setminus \dd
	$$
	as a right $\dd$-module. Here we have used the fact that $\mf{k}$ is the direct sum of a reductive Lie algebra and an abelian Lie algebra, which implies that the modular character $\delta$ defined in section~\ref{sec:dualitymonodromic} is zero. Therefore, as a non-equivariant $\dd$-module, 
	$$
    \mathbb{D}(\mc{G}_{\chi}) \simeq \mf{k} \dd \setminus \dd \otimes_{\mc{O}} \Omega_X^{\otimes -1}. 
    $$
    If $s \in \bigwedge^{\dim X} X$ is non-zero, it defines a nowhere vanishing section of $\Omega_X^{\otimes -1}$. As a $\mf{g}$-module, $ \bigwedge^{\dim X} X \simeq  \bigwedge^{\dim V} V$, and hence $x \cdot s = \Tr_0(x) s$ for all $x \in \mf{g}$. This means that 
    $$
    x \cdot (1 \otimes s) = (- x \cdot 1) \otimes s + 1 \otimes (x \cdot s) = (- \chi + \Tr_0)(x) (1 \otimes s). 
    $$
    Similarly, $v \cdot (1 \otimes s) = (- v) \otimes s$ for $v \in R \subset \sym R$. Since multiplication by $-1$ clearly commutes with the action of $G$, we deduce that $(\sym R)_+^G \cdot (1 \otimes s) = 0$. We deduce that $\mathbb{D}(\mc{G}_{\chi}) \simeq \mc{G}_{- \chi + \Tr_0}$, if we forget the equivariant structure. On the other hand, we know from Proposition \ref{prop:dualityOchiG} that $\mathbb{D}(\mc{G}_{\chi})$ is $(G(n\delta),-\chi)$-monodromic. Thus, $\mathbb{D}(\mc{G}_{\chi}) \simeq \mc{G}_{- \chi + \Tr_0} \otimes \mathrm{det}_0$ as required. 
\end{proof}

\begin{lemma}\label{lem:quotemcG}
	If $\chi \cdot \delta = 1$, then $\mc{F}_{\chi} \boxtimes S_0$ is a quotient of $\mc{G}_{\chi}$. 
\end{lemma}

\begin{proof}
	If $\left\{ e_{i,j}^{(0)} \right\}$ is the standard basis of $\mf{g}_0 \subset \mf{g}$, then the fact that $e_{i,j}^{(0)} \cdot x_k = - \delta_{k,i} x_j$ implies that 
	$$
	\nu_V\left( e_{i,j}^{(0)}\right) = - x_j \partial_i = - \partial_i x_j, \quad  \nu_V\left(e_{i,i}^{(0)} - e_{j,j}^{(0)}\right) = x_j \partial_j - x_i \partial_j = \partial_j x_j - \partial_i x_i, \quad \forall \ i \neq j
	$$
	and $\nu_{V}(\mathbf{1}_0) = - \sum_{i = 1}^n x_i \partial_i = n - \sum_{i = 1}^n \partial_i x_i$. We have $\nu_{V}(\mf{g}_i) = 0$ for all $i \neq 0$. This shows that for all $x \in [\mf{g}_0,\mf{g}_0]$, we have $\nu_{V}(x) \in \dd(X) \C[V]_+$. Then 
	$$
	\nu(x) - \chi(x) \equiv \nu_R - \chi'(x) \mod \dd(X) \C[V]_+
	$$
	for all $x \in \mf{g}$ because 
	$$
	\nu(\mathbf{1}_0) - n \chi_0 = \nu_R(\mathbf{1}_0) + \nu_V(\mathbf{1}_0) - n \chi_0 \equiv \nu_R(\mathbf{1}_0) - n \chi_0 + n \mod \dd(X) \C[V]_+. 
	$$ 
	We conclude that 
	$$
	\dd_X \nu_{\chi}(\mf{g}) + \dd_X (\sym R)^G_+ \subseteq \dd_X \nu_{R,\chi'}(\mf{g}) + \dd_X (\sym R)^G_+ + \dd_X \C[V]_+,
	$$
	and hence $\mc{G}_{\chi} \twoheadrightarrow \mc{F}_{\chi} \boxtimes S_0$. This morphism sends the invariant generator of $\mc{G}_{\chi}$ onto the invariant generator of $\mc{F}_{\chi} \boxtimes S_0$, which implies that this is a morphism of $(G(n\delta),\chi)$-monodromic $\dd$-modules. 
\end{proof} 

\begin{proof}[Proof of Theorem \ref{thm:intss}]
	If $\chi \cdot \delta \in \Z$, then we may assume by Lemma \ref{lem:twistequivalence} that $\chi \cdot \delta = 0$. Set $\psi = - \chi + \Tr_0$, so that $\psi \cdot \delta = 1$. Then Lemma \ref{lem:commuteequivdual} (2) implies that $\mathbb{D}(\mc{G}_{\psi}) \simeq \mc{G}_{\chi} \otimes \mathrm{det}_0$. Hence, combining Lemma \ref{lem:commuteequivdual} (1) and Lemma \ref{lem:quotemcG}, we deduce that 
	$$
	\mathbb{D}(\mc{F}_{\psi}) \boxtimes S_0 \otimes \mathrm{det}_0 \hookrightarrow \mc{G}_{\chi}
	$$
	The module $\ms{G}_{\chi}$ has the property that no quotient is killed by the functor $\Ham_{\chi}$ of Hamiltonian reduction. Therefore, if we assume that $\mc{G}_{\chi}$ is semi-simple, it follows that no submodule of $\mc{G}_{\chi}$ is killed by Hamiltonian reduction.
	
	On the other hand,  
	\begin{align*}
	\Ham_{\chi}(\mathbb{D}(\mc{F}_{\psi}) \boxtimes (S_0 \otimes \mathrm{det}_0)) & = \Gamma(X,\mathbb{D}(\mc{F}_{\psi}) \boxtimes (S_0 \otimes \mathrm{det}_0))^G \\
	 & \subset \Gamma(X,\mathbb{D}(\mc{F}_{\psi}) \boxtimes (S_0 \otimes \mathrm{det}_0))^{\Cs_{\Delta}} \\
	 & = \Gamma(R,\mathbb{D}(\mc{F}_{\psi})) \boxtimes \Gamma(V,(S_0 \otimes \mathrm{det}_0))^{\Cs_{\Delta}},
	\end{align*}
	since $\mathbb{D}(\mc{F}_{\psi})$ is a $(PG(n \delta), \chi)$-monodromic module. Then, the fact that all the weights of $\Cs_{\Delta}$ on $\Gamma(V,(S_0 \otimes \mathrm{det}_0))$ are $\ge n$ implies that $\Gamma(V,(S_0 \otimes \mathrm{det}_0))^{\Cs_{\Delta}} = 0$. Thus, $\Ham_{\chi}$ kills $\mathbb{D}(\mc{F}_{\psi}) \boxtimes (S_0 \otimes \mathrm{det}_0)$, contradicting our assumption that $\mc{G}_{\chi}$ is semi-simple. 	   
\end{proof}

\begin{proof}[Proof of Theorem \ref{thm:mainequiv}]
	Finally, we can give the proof of Theorem \ref{thm:mainequiv}. It was explained in the paragraph preceding Theorem \ref{thm:mainequiv} that (c) holds if and only if $\chi \cdot \alpha \notin \Z$ for all $\alpha \in \mc{R}_n$.	
	
	Assume first that (a): $\Add_{\chi}$ is semi-simple. This implies that $\Osph$ is semi-simple. Hence, as explained in the proof of Theorem \ref{thm:equivcond}, this implies that $\chi \cdot \alpha \notin \Z$ for all $\alpha \in \mathsf{R}^+$ with $\varepsilon_0 \cdot \alpha < n$. Therefore, we just need to show that $\chi \cdot \delta \notin \Z$. But this is precisely the statement of Theorem \ref{thm:intss}. Conversely, if $\chi \cdot \alpha \notin \Z$ for all $\alpha \in \mc{R}_n$, then by Theorem \ref{thm:equivcond}, category $\Osph$ is semi-simple and $\Ham_{\chi} : \Add_{\chi} \rightarrow \Osph$ is an equivalence. We deduce that $\Add_{\chi}$ is semi-simple. Thus, (a) holds if and only if $\chi \cdot \alpha \notin \Z$ for all $\alpha \in \mc{R}_n$. 
	
	Assume that (b): $\ms{M} |_{X^{\reg}} \neq 0$ for all non-zero $\ms{M}$ in $\Add_{\chi}$. The open set $X^{\reg}$ is the intersection of $X^{\circ}$ with the open set $X^{\mathrm{ss}}$, which is defined to be the pre-image in $X$ of the regular locus $\h^{\reg}/W$ in $\h /W$ under the quotient map $X \rightarrow X/\!/G \simeq \mf{h} / W$. Therefore, there exists an element $t \in \C[X]^G$ such that $X^{\reg} = X^{\circ} \cap (t \neq 0)$. In particular, if $\ms{M} |_{X^{\reg}} \neq 0$ then $\ms{M} |_{X^{\circ}} \neq 0$ and we deduce from Proposition \ref{prop:Xcyclicres} that $\chi \cdot \alpha \notin \Z$ for all $\alpha \in \mc{R}_n$. Conversely, if $\chi \cdot \alpha \notin \Z$ for all $\alpha \in \mc{R}_n$, and $\ms{M} \in \Add_{\chi}$ is non-zero, then once again Proposition \ref{prop:Xcyclicres} implies that $\ms{M} |_{X^{\circ}} \neq 0$. The fact that and $\Ham_{\chi} : \Add_{\chi} \rightarrow \Osph$ is an equivalence implies that $\Ham_{\chi}(\ms{M}) \neq 0$. Moreover, the fact that category $\Osph$ is semi-simple implies that every module in $\Osph$ is free over $\C[\h]^W = \C[X]^G$. In particular, $\Ham_{\chi}(\ms{M}) \neq 0$ implies that $\Ham_{\chi}(\ms{M})[t^{-1}] \neq 0$. Hence,  
	$$
	0 \neq \Ham_{\chi}(\ms{M})[t^{-1}] \subset \Gamma(X,\ms{M})[t^{-1}]= \Gamma(X^{\mathrm{ss}},\ms{M}),
	$$
	and thus $\ms{M} |_{X^{\reg}} \neq 0$ i.e. (b) holds if and only if $\chi \cdot \alpha \notin \Z$ for all $\alpha \in \mc{R}_n$. 
\end{proof}

\subsection{Extensions of local systems}

Our classification allows one to make an elementary ``cleaness'' result for generic $\chi$. Namely, for each $(\lambda;\nu) \in \mc{Q}(n,\ell)$, let $j^{(\lambda;\nu)} : \mc{O}_{(\lambda;\nu)} \hookrightarrow X$ be the locally closed embedding. The irreducible $(G,\chi)$-monodromic local system on $\mc{O}_{(\lambda;\nu)}$ (when it exists) is denoted $\mc{L}_{\chi}$. 

\begin{corollary}\label{cor:clean}
If $\chi \cdot \alpha \notin \Z$ for all $\alpha \in \mc{R}_n$ then 
$$
j^{(\lambda;\emptyset)}_!  \mc{L}_{\chi} = j^{(\lambda;\emptyset)}_{!,*} \mc{L}_{\chi}  = j^{(\lambda;\emptyset)}_{*} \mc{L}_{\chi}, \quad \forall \lambda \in \mc{P}_{\ell}(n) = \mc{Q}_{\chi}(n,\ell).
$$
Moreover, $\Ham_{\chi}(j^{(\lambda;\emptyset)}_{!,*} \mc{L}_{\chi}) \neq 0$. 
\end{corollary}

It seems natural to expect, based on Corollary \ref{cor:clean} that $\Ham_{\chi}(j^{(\lambda;\nu)}_{!,*} \mc{L}_{\chi}) \neq 0$ implies that $\nu = \emptyset$ for all $\chi$ and that the Fourier transform of $j^{(\lambda;\emptyset)}_{!} \mc{L}_{\chi}$ should always give an admissible $\dd$-module that maps to the Verma module in $\Osph$ (assuming for simplicity that $\mbf{\kappa}$ is not aspherical). However, as the example in section \ref{sec:neq1} shows, such expectations are far too naive.

\subsection{Example: the enhanced nilcone}\label{sec:neq1}

In this section, we describe in detail what our results mean in the case of the original enhanced nilpotent cone. This is the situation originally considered by Gan-Ginzburg \cite{AlmostCommutingVariety}, and relates admissible $\dd$-modules to the spherical subalgebra $e \H_{\kappa} (\s_n) e$ of the rational Cherednik algebra $\H_{\kappa}(\s_n)$ associated to the symmetric group. The set $\mathcal{Q}(n,1)$ is equal to the set $\mc{P}_2(n)$ of all bipartitions of $n$. Given a partition $\nu = (\nu_1, \ds )$, we define $\gcd(\nu)$ to be the greatest common divisor of the $\nu_i$ and set $\gcd(\emptyset) = 0$. The following proposition was used in the proof of \cite[Theorem 29]{BellSchKostka}.

\begin{proposition}\label{prop:ell1}
Fix $\chi \in \C$. 
\begin{enumerate}
\item For each $(\lambda;\nu) \in \mc{P}_2(n)$,  
$$
\pi_1(\mathcal{O}_{(\lambda;\nu)}) = \Z /  \gcd(\nu) \Z.
$$
\item $\mathcal{O}_{(\lambda;\nu)}$ admits a $(\GL_n,\chi)$-monodromic local system if and only if $\gcd(\nu) \chi \in \Z$. 
\item The category $\Add_{\chi}$ is not semi-simple if and only if $\chi = \frac{r}{m}$, with $1 \le m \le n$ and $r \in \Z$, $(r,m) = 1$. 
\end{enumerate}
\end{proposition}

\begin{proof}
Part (1) is just Proposition \ref{thm:cokerfun}. Part (2) then follows from Corollary \ref{cor:irrepO}. Finally, part (3) is a consequence of Theorem \ref{thm:mainequiv}, noting that $\chi = \frac{r}{m}$, with $1 \le m \le n$ and $r \in \Z$, $(r,m) = 1$ is equivalent to the statement $\chi \cdot \alpha \in \Z$ for some $\alpha \in \mc{R}_n$. 
\end{proof}

Consider the case $n = 2$. In this case the orbits in $\mc{N}(V) \times V$ are labeled by bipartitions of $2$. Let $\mc{O}_{\reg}$ denote the regular nilpotent orbit in $\mc{N}(V)$, so that $\mc{N}(V) = \mc{O}_{\reg} \cup \{ 0 \}$. 
\begin{align*}
\mc{O}_{((2);\emptyset)} & = \{ (X,v) \in \mc{O}_{\reg} \times V \ | \  X v \neq 0 \}, & \pi_1(\mc{O}) = \Z. \\
\mc{O}_{((1,1);\emptyset)} & = \{ (X,v) \in \mc{O}_{\reg} \times V \ | \ X v = 0 \}, & \pi_1(\mc{O}) = \Z. \\
\mc{O}_{((1);(1))} & = \{ (0,v) \in \{ 0 \} \times V \ | \ v \neq 0 \}, & \pi_1(\mc{O}) = 1. \\
\mc{O}_{(\emptyset;(2))} & = \mc{O}_{\reg} \times \{ 0 \},  & \pi_1(\mc{O}) = \Z_2. \\
\mc{O}_{(\emptyset;(1,1))} & = \{ (0,0) \}, & \pi_1(\mc{O}) = 1. \\
\end{align*}
When $\chi \notin \frac{1}{2} \Z$, the category $\Orb_{\chi}$ is semi-simple, with two simple objects $j^{((2);\emptyset)}_{!,*} \mc{L}_{\chi}$ and $j^{((1,1);\emptyset)}_{!,*} \mc{L}_{\chi}$. When $\chi \in \frac{1}{2} + \Z$, the category $\Orb_{\chi}$ is not semi-simple and has three simple objects
$$
L_1 := j^{((2);\emptyset)}_{!,*}\mc{L}_{\chi}, \quad L_2 :=  j_{!,*}^{((1,1);\emptyset)} \mc{L}_{\chi}, \quad L_3 := j_{!,*}^{(\emptyset;(2))} \mc{L}_{\chi}.
$$
Applying \cite[Theorem 7.8]{CEE} c.f. \cite[Proposition 6.4.1]{mirabolicHam}, one can deduce that $L_2$ is never killed by Hamiltonian reduction, that $\Ham_{\chi}(L_1) \neq 0$ if and only if $\chi \in \frac{1}{2} + \Z_{\ge 0}$ and $\Ham_{\chi}(L_3) \neq 0$ if and only if $\chi \in - \frac{1}{2} + \Z_{< 0}$. 

If $\chi \in \Z$ is integral then again $\Orb_{\chi}$ is not semi-simple and has five simple objects
$$
S_1 := j^{((2);\emptyset)}_{!,*}\mc{L}_{\chi}, \quad S_2 :=  j_{!,*}^{((1,1);\emptyset)} \mc{L}_{\chi}, \quad S_3 := j_{!,*}^{(\emptyset;(2))} \mc{L}_{\chi}, 
$$
$$
S_4 := j_{!,*}^{(\emptyset;(1,1))} \mc{L}_{\chi}, \quad S_5 := j_{!,*}^{((1);(1))} \mc{L}_{\chi}.
$$
In this case, when $\chi \in \Z_{\ge 0}$, $S_1$ and $S_5$ are not killed by $\Ham_{\chi}$, but $S_2,S_3$ and $S_4$ are, and when $\chi \in \Z_{< 0}$, $S_3$ and $S_4$ are not killed, but the others are.


\begin{thebibliography}{10}

\bibitem{AH}
P.~N.~Achar and A.~Henderson.
\newblock Orbit closures in the enhanced nilpotent cone.
\newblock {\em Adv. Math.}, 219(1):27--62, 2008.

\bibitem{ArikiSemisimple}
S.~Ariki.
\newblock On the semi-simplicity of the {H}ecke algebra of {$({\bf Z}/r{\bf
  Z})\wr {S}_n$}.
\newblock {\em J. Algebra}, 169(1):216--225, 1994.


\bibitem{BB-enhnil-quiver}
G.~Bellamy, and M.~Boos.
\newblock The (cyclic) enhanced nilpotent cone via quiver 
    representations.
\newblock {\em Manuscripta Math.}, Online first, 2018. 

\bibitem{mirabolicHam}
G.~Bellamy and V.~Ginzburg.
\newblock Hamiltonian reduction and nearby cycles for mirabolic
  {$\mathscr{D}$}-modules.
\newblock {\em Adv. Math.}, 269:71--161, 2015.

\bibitem{BellSchKostka}
G.~Bellamy and T.~Schedler.
\newblock Filtrations on Springer fiber cohomology and Kostka polynomials.
\newblock {\em Lett. Math. Phys.}, 108(3):679--698, 2018.

\bibitem{BerestChalykhQuasi}
Y.~Berest and O.~Chalykh.
\newblock Quasi-invariants of complex reflection groups.
\newblock {\em Compos. Math.}, 147(3):965--1002, 2011.





\bibitem{B2}
M.~Boos.
\newblock Finite parabolic conjugation on varieties of nilpotent matrices.
\newblock {\em Algebr. Represent. Theory}, 17(6):1657--1682, 2014.

\bibitem{BorelDmod}
A.~Borel, P.-P. Grivel, B.~Kaup, A.~Haefliger, B.~Malgrange, and F.~Ehlers.
\newblock {\em Algebraic {$D$}-modules}, volume~2 of {\em Perspectives in
	Mathematics}.
\newblock Academic Press, Inc., Boston, MA, 1987.

\bibitem{BKGelfandKirillov}
W.~Borho and H.~Kraft.
\newblock \"uber die {G}elfand-{K}irillov-{D}imension.
\newblock {\em Math. Ann.}, 220(1):1--24, 1976.

\bibitem{BrionQuiver}
M.~Brion.
\newblock Representations of quivers.
\newblock In {\em Geometric methods in representation theory. {I}}, volume~24
  of {\em S\'emin. Congr.}, pages 103--144. Soc. Math. France, Paris, 2012.


\bibitem{CEE}
D.~Calaque, B.~Enriquez, and P.~Etingof.
\newblock Universal {KZB} equations: the elliptic case.
\newblock In {\em Algebra, arithmetic, and geometry: in honor of {Y}u. {I}.
  {M}anin. {V}ol. {I}}, volume 269 of {\em Progr. Math.}, pages 165--266.
  Birkh\"auser Boston Inc., Boston, MA, 2009.
  
\bibitem{CBmomap}
W.~Crawley-Boevey.
 \newblock Geometry of the moment map for representations of quivers.
  \newblock {\em Compositio Math.}, 126(3):257--293, 2001.

  Congr. Rep., pages 445--499. Eur. Math. Soc., Z\"urich, 2011.

\bibitem{DoGinT}
G.~Dobrovolska, V.~Ginzburg, and R.~Travkin.
\newblock Moduli spaces, indecomposable objects and potentials over a finite field.
\newblock {\em arXiv}, 1612.01733v1, 2016. 


\bibitem{DoddCycle}
C.~Dodd.
\newblock Injectivity of the cycle map for finite-dimensional {$W$}-algebras.
\newblock {\em Int. Math. Res. Not. IMRN}, (19):5398--5436, 2014.



\bibitem{EG}
P.~Etingof and V.~Ginzburg.
\newblock Symplectic reflection algebras, {C}alogero-{M}oser space, and
deformed {H}arish-{C}handra homomorphism.
\newblock {\em Invent. Math.}, 147(2):243--348, 2002.

\bibitem{MirabolicCharacter}
M.~Finkelberg and V.~Ginzburg.
\newblock On mirabolic {$\mathscr{D}$}-modules.
\newblock {\em Int. Math. Res. Not. IMRN}, (15):2947--2986, 2010.


\bibitem{AlmostCommutingVariety}
W.~L. Gan and V.~Ginzburg.
\newblock Almost-commuting variety, {$\mathscr{D}$}-modules, and {C}herednik
  algebras.
\newblock {\em IMRP Int. Math. Res. Pap.}, pages 26439, 1--54, 2006.
\newblock With an appendix by Ginzburg.


\bibitem{AdmissibleModules}
V.~Ginzburg.
\newblock Admissible modules on a symmetric space.
\newblock {\em Ast\'erisque}, (173-174):9--10, 199--255, 1989.
\newblock Orbites unipotentes et repr{\'e}sentations, III.

\bibitem{Primitive}
V.~Ginzburg.
\newblock On primitive ideals.
\newblock {\em Selecta Math. (N.S.)}, 9(3):379--407, 2003.

\bibitem{Hompropquantum}
K.~R. Goodearl and J.~J. Zhang.
\newblock Homological properties of quantized coordinate rings of semisimple
groups.
\newblock {\em Proc. Lond. Math. Soc. (3)}, 94(3):647--671, 2007.

\bibitem{GordonCyclicQuiver}
I.~G. Gordon.
\newblock A remark on rational {C}herednik algebras and differential operators
  on the cyclic quiver.
\newblock {\em Glasg. Math. J.}, 48(1):145--160, 2006.

\bibitem{GunnighamAbelian}
S.~Gunningham.
\newblock Generalized {S}pringer theory for {$D$}-modules on a reductive {L}ie
algebra.
\newblock {\em Selecta Math. (N.S.)}, 24(5):4223--4277, 2018.

\bibitem{Holland}
M.~P. Holland.
\newblock Quantization of the {M}arsden-{W}einstein reduction for extended
  {D}ynkin quivers.
\newblock {\em Ann. Sci. \'Ecole Norm. Sup. (4)}, 32(6):813--834, 1999.

\bibitem{HottaKashiwara}
R.~Hotta and M.~Kashiwara.
\newblock The invariant holonomic system on a semisimple {L}ie algebra.
\newblock {\em Invent. Math.}, 75:327--358, 1984.

\bibitem{HTT}
R.~Hotta, K.~Takeuchi, and T.~Tanisaki.
\newblock {\em {$D$}-modules, Perverse Sheaves, and Representation Theory},
  volume 236 of {\em Progress in Mathematics}.
\newblock Birkh\"auser Boston Inc., Boston, MA, 2008.
\newblock Translated from the 1995 Japanese edition by Takeuchi.

\bibitem{Joh}
C.~P.~Johnson.
\newblock {\em Enhanced nilpotent representations of a cyclic quiver}.
\newblock ProQuest LLC, Ann Arbor, MI, 2010.
\newblock Thesis (Ph.D.)--The University of Utah.


\bibitem{KnappCohomology}
A.~W. Knapp.
\newblock {\em Lie groups, {L}ie algebras, and cohomology}, volume~34 of {\em
	Mathematical Notes}.
\newblock Princeton University Press, Princeton, NJ, 1988.



\bibitem{Kraft}
H.~Kraft.
\newblock {\em Geometrische {M}ethoden in der {I}nvariantentheorie}.
\newblock Aspects of Mathematics, D1. Friedr. Vieweg und Sohn, Braunschweig,
  1984.


\bibitem{LSKernel}
T.~Levasseur and J.~T.~Stafford.
\newblock The kernel of an homomorphism of {H}arish-{C}handra.
\newblock {\em Ann. Sci. \'Ecole Norm. Sup. (4)}, 29(3):385--397, 1996. 

\bibitem{LS}
T.~Levasseur and J.~T. Stafford.
\newblock Semi-simplicity of invariant holonomic systems on a reductive {L}ie
algebra.
\newblock {\em Amer. J. Math.}, 119(5):1095--1117, 1997.

\bibitem{BernsteinLosev}
I.~Losev.
\newblock Bernstein inequality and holonomic modules.
\newblock {\em Adv. Math.}, 308:941--963, 2017.
\newblock With an appendix by Losev and P. Etingof.

\bibitem{MNDerived}
K.~McGerty and T.~Nevins.
\newblock Derived equivalence for quantum symplectic resolutions.
\newblock {\em Selecta Math. (N.S.)}, 20(2):675--717, 2014.

\bibitem{MR}
J.~C. McConnell and J.~C. Robson.
\newblock {\em Noncommutative {N}oetherian {R}ings}, volume~30 of {\em Graduate
	Studies in Mathematics}.
\newblock American Mathematical Society, Providence, RI, revised edition, 2001.
\newblock With the cooperation of L. W. Small.


\bibitem{MNKNstrata}
K.~McGerty and T.~Nevins.
\newblock Compatibility of {$t$}-structures for quantum symplectic resolutions.
\newblock {\em Duke Math. J.}, 165(13):2529--2585, 2016.

\bibitem{MirkovicVilonen}
I.~Mirkovi{\'c} and K.~Vilonen.
\newblock Characteristic varieties of character sheaves.
\newblock {\em Invent. Math.}, 93(2):405--418, 1988.

\bibitem{OblomkovHC}
A.~Oblomkov.
\newblock Deformed {H}arish-{C}handra homomorphism for the cyclic quiver.
\newblock {\em Math. Res. Lett.}, 14(3):359--372, 2007.


\bibitem{ShojiSurvey}
T.~Shoji.
\newblock Springer correspondence for complex reflection groups.
\newblock {\em arXiv}, 1511.03353v1, 2015.


\bibitem{Tr}
R.~Travkin.
\newblock Mirabolic {R}obinson-{S}chensted-{K}nuth correspondence.
\newblock {\em Selecta Math. (N.S.)}, 14(3-4):727--758, 2009.

\bibitem{VdenBGeq}
M.~Van~den Bergh.
\newblock Some generalities on {$G$}-equivariant quasi-coherent
{$\mathcal{O}_X$} and {$\mathcal{D}_X$}-modules.
\newblock {\em Unpublished}.

\bibitem{GenCatOWebster}
B.~Webster.
\newblock On generalized category {$\mathcal{O}$} for a quiver variety.
\newblock {\em Math. Ann.}, 368(1-2):483--536, 2017.





\end{thebibliography}


	
	\def\cprime{$'$} \def\cprime{$'$} \def\cprime{$'$} \def\cprime{$'$}
	\def\cprime{$'$} \def\cprime{$'$} \def\cprime{$'$} \def\cprime{$'$}
	\def\cprime{$'$} \def\cprime{$'$} \def\cprime{$'$} \def\cprime{$'$}
	\def\cprime{$'$} \def\cprime{$'$}

\end{document}